\documentclass[reqno, 12pt]{amsart}

\usepackage{amsmath}
\usepackage{amsfonts}
\usepackage{amssymb}
\usepackage{amsthm}
\usepackage{mathrsfs}
\usepackage{bbm}
\usepackage{graphicx}
\usepackage{tikz}
\usetikzlibrary{calc,decorations.markings}
\usetikzlibrary{arrows.meta}
\usetikzlibrary{positioning}
\usetikzlibrary{cd, intersections, calc, decorations.pathmorphing, arrows, decorations.pathreplacing}
\usepackage{color}
\usepackage{comment}
\usepackage[normalem]{ulem}

\usepackage{epsfig,amssymb,amsmath,version}
\usepackage{amssymb,version,graphicx,fancybox,mathrsfs}
\usepackage{subfigure}
\usepackage{color}
\usepackage{stmaryrd}
\usepackage{multirow}
\usepackage{booktabs,siunitx}
\usepackage{booktabs}
\usepackage{multicol,epstopdf}
\usepackage{xypic}
\usepackage[]{graphicx}
\usepackage[all]{xy}
\usepackage{tikz,ifthen}
\usetikzlibrary{positioning, math, decorations.markings, arrows.meta}
\usetikzlibrary{calc,shapes,cd}
\usetikzlibrary{knots} 
\usepgfmodule{decorations}
\allowdisplaybreaks
\usepackage{bm}

\usepackage{adjustbox}

\usepackage{bmpsize}

\usepackage{mathtools}

\usepackage{enumerate}

\usepackage{subfiles}
\usepackage[colorlinks, linkcolor=blue, citecolor=red, urlcolor=cyan,pagebackref]{hyperref}

\allowdisplaybreaks

\numberwithin{equation}{section}

\usepackage{cleveref}
\crefname{thm}{Theorem}{Theorems}
\crefname{cor}{Corollary}{Corollaries}
\crefname{lem}{Lemma}{Lemmas}
\crefname{sublem}{Sublemma}{Sublemmas}
\crefname{prop}{Proposition}{Propositions}
\crefname{def}{Definition}{Definitions}
\crefname{example}{Example}{Examples}
\crefname{claim}{Claim}{Claims}
\crefname{conj}{Conjecture}{Conjectures}
\crefname{conv}{Notation}{Notations}
\crefname{rem}{Remark}{Remarks}
\crefname{rmk}{Remark}{Remarks}
\crefname{prob}{Problem}{Problems}
\crefname{figure}{Figure}{Figures}
\crefname{table}{Table}{Tables}
\crefname{section}{Section}{Sections}
\crefname{subsection}{Section}{Sections}
\crefname{appendix}{Appendix}{Appendices}
\crefname{introthm}{Theorem}{Theorems}
\crefname{introcor}{Corollary}{Corollaries}
\crefname{introconj}{Conjecture}{Conjectures}

\newtheorem{thm}{Theorem}[section]
\newtheorem{prop}[thm]{Proposition}
\newtheorem{cor}[thm]{Corollary}
\newtheorem{lem}[thm]{Lemma}

\newtheorem{introthm}{Theorem}
\newtheorem{introcor}[introthm]{Corollary}

\theoremstyle{definition}
\newtheorem{dfn}[thm]{Definition}

\theoremstyle{remark}

\newtheorem{rem}[thm]{Remark}

\newcommand{\beq}{\begin{equation}}
\newcommand{\eeq}{\end{equation}}

\tikzset{
    squigarrow/.style={-{Classical TikZ Rightarrow[length=4pt]}, decorate, decoration={snake, amplitude=1.8pt, pre length=2pt, post length=3pt}}
}

\newcommand{\bC}{\mathbb C}

\newcommand{\bJ}{\mathbb J}
\newcommand{\bN}{\mathbb N}
\newcommand{\bP}{\mathbb P}

\newcommand{\bR}{\mathbb R}
\newcommand{\bT}{\mathbb T}
\newcommand{\bZ}{\mathbb Z}

\newcommand{\cA}{\mathcal A}

\newcommand{\cF}{\mathcal F}

\newcommand{\cN}{\mathcal N}
\newcommand{\cO}{\mathcal O}

\newcommand{\cS}{\mathcal S}

\newcommand{\cX}{\mathcal X}

\newcommand{\cZ}{\mathcal Z}

\newcommand{\fB}{\mathfrak{B}}

\newcommand{\fS}{\mathfrak S}

\newcommand{\sfC}{\mathsf{C}}

\newcommand{\sfH}{\mathsf{H}}
\newcommand{\sfK}{\mathsf{K}}
\newcommand{\sfP}{\mathsf{P}}
\newcommand{\sfQ}{\mathsf{Q}}

\newcommand{\al}{\alpha}
\newcommand{\alp}{\alpha_{m^\ast}}

\newcommand{\bk}{\mathbf k}

\newcommand{\barA}{\overline{\cA}}
\newcommand{\barH}{\overline{\sfH}}
\newcommand{\barP}{\overline{\sfP}}
\newcommand{\barQ}{\overline{\sfQ}}

\DeclareMathOperator{\SL}{\mathrm{SL}}
\DeclareMathOperator{\skeleton}{\textbf{sk}}
\DeclareMathOperator{\im}{\mathrm{Im}}
\DeclareMathOperator{\kernel}{\mathrm{Ker}}

\DeclareMathOperator{\Int}{\mathrm{Int}}

\DeclareMathOperator{\Specm}{\mathrm{MaxSpec}}
\DeclareMathOperator{\Frac}{\mathrm{Frac}}

\DeclareMathOperator{\rankZ}{\text{rank}_{\mathcal{Z}}}

\newcommand{\gaa}{\mathsf{g}}
\newcommand{\tra}{{\rm tr}_{\lambda}^A}
\newcommand{\A}{\mathcal A_{\hat q}(\Sigma,\lambda)}

\newcommand{\rA}{\overline{\mathcal A}_{\hat q}(\Sigma,\lambda)}
\newcommand{\rAp}{\overline{\mathcal A}^{+}(\Sigma,\lambda)}

\newcommand{\Sn}{\cS_n(\Sigma)}

\newcommand{\cAl}{\cA_{\hat q}(\Sigma,\lambda)}
\newcommand{\cXl}{\cX_{\hat q}(\Sigma,\lambda)}
\newcommand{\barX}{\overline{\cX}}

\newcommand{\Xbll}{\cX^{\rm bl}_{\hat q}(\Sigma,\lambda)}

\def\bk{\mathbf{k}}

\def\barH{\overline{\sfH}}

\def\barK{\overline{\sfK}}
\def\barKl{\overline{\sfK}_{\lambda^\ast}}
\def\barKt{\overline{\sfK}_\tau}

\def\barQ{\overline{\sfQ}}
\def\barQl{\overline{\sfQ}_{\lambda^\ast}}
\def\barV{\overline{V}}
\def\barVl{\overline{V}_\lambda}
\def\barVlast{\overline{V}_{\lambda^\ast}}
\def\barVt{\overline{V}_\tau}
\def\obV{{\mathring{\barV}}}
\def\obVlast{\obV_{\lambda^\ast}}
\def\obVl{\obV_{\lambda}}

\def\proj{\mathbf{pr}}

\def\tfk{\mathbf{k}}
\def\tft{\mathbf{t}}

\def\tfb{\mathbf{b}}
\def\tff{\mathbf{h}}

\def\Lp{\Lambda_{\partial}}

\def\rSn{\overline \cS_n(\Sigma)}
\def\Si{\Sigma}

\newcommand{\cev}[1]{\reflectbox{\ensuremath{\vec{\reflectbox{\ensuremath{#1}}}}}}

\DeclareMathOperator{\Azumaya}{\mathsf{Azm}}

\title{Center of stated $\mathrm{SL}(n)$-skein algebras: even roots of unity}

\author[Hiroaki Karuo]{Hiroaki Karuo}
\address{Hiroaki Karuo, Department of Mathematics, Gakushuin University, Mejiro, Toshima-ku, Tokyo, Japan.}
\email{hiroaki.karuo@gakushuin.ac.jp}

\author[Zhihao Wang]{Zhihao Wang}
\address{Zhihao Wang, School of Mathematics, Korea Institute for Advanced Study (KIAS), 85 Hoegi-ro, Dongdaemun-gu, Seoul 02455, Republic of Korea}
\email{zhihaowang@kias.re.kr}

\date{}

\begin{document}
\maketitle

\begin{abstract}
We describe the center of quantum tori appearing in quantum higher Teichm\"uller theory when the quantum parameter is an even root of unity. 
This is a subsequent of the work in the odd roots of unity case. 
The centers of the quantum tori help to understand those of (reduced) stated $\mathrm{SL}(n)$-skein algebras via quantum trace maps. 
Moreover, we compute the PI-degree of the quantum tori, which are the same with those of (reduced) stated $\mathrm{SL}(n)$-skein algebras. 
As corollaries, we give matrix decompositions for anti-symmetric integer matrices for the quantum tori.
\end{abstract}

\setcounter{tocdepth}{1}
\tableofcontents

\def\cR{\mathcal R}

\section{Introduction}
\subsection{Background}
For an oriented surface $\Sigma$ and a domain $\cR$ with a distinguished invertible element $\hat{q}$, one can define a noncommutative $\cR$-algebra associated with $\Sigma$, called the $\SL(n)$-skein algebra, which is a quantization of the $\SL(n,\bC)$-character variety of $\Sigma$ when $\cR=\bC$ \cite{Sik05, Bul97, PS19}. The generators of this algebra are oriented $n$-valent graphs, and relations are imposed based on the $U_q(\frak{sl}_n)$-Reshetikhin--Turaev functor \cite{RT91}. It is known that the $\SL(n)$-skein algebra is isomorphic to the Kauffman bracket skein algebra \cite{Prz99, Tur91} for $n=2$, and is isomorphic to the Kuperberg skein algebra \cite{Kup96} for $n=3$. Specifically, in the case of $n=2$, it is deeply related to the quantum Teichm\"uller space \cite{BW11, Le19}, topological quantum field theory \cite{BHMV95}, and quantum moduli algebras \cite{Fai20, BFR23}, etc. 

From the construction of the quantum trace map by Bonahon and Wong \cite{BW11}, \cite{Le18} (resp. \cite{CL22}) defined the stated (resp. reduced stated) $\SL(2)$-skein algebra by introducing ``stated arcs" and appropriate relations \cite{Le18, CL22} as a generalization of the $\SL(2)$-skein algebra,. 
There are some important relationship between (reduced) stated $\SL(2)$-skein algebras and objects and concepts in other fields, e.g., quantum cluster algebras \cite{Mul16, LY23}, quantum Teichm\"uller spaces \cite{Le19, CL22}, topological quantum field theory \cite{CL22a}, and factorization homology \cite{BZBJ18, Coo23}. 
An advantage of use of skein algebras is to capture phenomena diagrammatically. This implies that its significance is not only for low-dimensional topology. 

As an analog of $\SL(2)$, the (reduced) stated $\SL(n)$-skein algebra was introduced for general $n$ with stated oriented $1$-$n$-valent graphs \cite{LS21, LY23}; see also \cite{Hig23} for $n=3$. 
The following results ensure that this generalization is reasonable. 
\begin{enumerate}
    \item The stated SL$(n)$-skein algebra of a bigon is isomorphic to the quantum coordinate ring $\cO_q(\SL(n))$ as a Hopf algebra (for $n=2$, see \cite{CL22}; for general $n$, see \cite{LS21}).
    \item For a family of surfaces and  generic $\hat{q}$, \cite{Fai20, BFR23} showed that it is isomorphic to the unrestricted quantum moduli algebra (a.k.a. the graph algebra)  introduced by Alekseev--Grosse--Schomerus \cite{AGS95} and Buffenoir--Roche \cite{BR95} independently.
    \item A homomorphism ($X$-quantum trace map) to the Fock--Goncharov algebra \cite{FG06, FG09} appearing in the context of the (quantum) higher Teichm\"uller theory is constructed, and under certain conditions, it is injective (for $n=2$, see \cite{BW11, Le19}; for $n=3$, see \cite{Kim20, Kim21, Dou24}; for general $n$, see \cite{LY23}).
\end{enumerate}

To understand the algebraic structures of (reduced) stated $\SL(n)$-skein algebras, a fundamental and important way is to understand their representation theory. 
Thanks to \cite{FKBL19,BG02}, the Unicity theorem is helpful to understand irreducible representations of noncommutative algebras . The Unicity theorem states that if an $\cR$-algebra satisfies 
\begin{enumerate}
    \item it is finitely generated as an $\cR$-algebra,
    \item it has no zero-divisors,
    \item it is finitely generated as a module over its center $Z$,
\end{enumerate}
then its finite-dimensional irreducible representations can be understood using the Azumaya locus (a Zariski dense open subset of Maximal spectrum of its center). For (reduced) stated $\SL(n)$-skein algebras, (1) and (2) (resp. (3)) are shown in \cite{LY23} (resp. in \cite{KW24,Wan23}).

To describe the center of the (reduced stated/stated) $\SL(2)$ or $\SL(3)$-skein algebra, its basis and an appropriate filtration was essential in \cite{FKBL19,KW24a, Yu23}. However, applying this approach to the (reduced stated/stated) $\SL(n)$-skein algebra is challenging because of lack of a `good' basis. 
In \cite{KW24}, the authors computed the PI-degrees for (reduced) stated $\SL(n)$-skein algebras only in the odd root of unity case using quantum tori, which are defined in the context of higher Teichm\"uller theory. The main purpose of this paper is to generalize the results to the even root of unity case.

\subsection{Main results}
Let $\Si$ be a triangulable pb surface without interior punctures.
Besides the $X$-quantum trace map,
L{\^e} and Yu also constructed the $A$-quantum trace maps for the stated ${\rm SL}(n)$-skein algebra $\cS_n(\Si)$ and the reduced stated ${\rm SL}(n)$-skein algebra $\overline{\cS}_n(\Si)$ \cite{LY23} (see Theorem~\ref{traceA}).

Let $\lambda$ be a triangulation of $\Si$.
The $A$-version quantum trace maps are algebra homomorphisms 
\begin{align*}
   {\rm tr}_{\lambda}^A\colon \cS_n(\Si)&\longrightarrow \A \\
    \overline{\rm tr}_{\lambda}^A\colon\overline{\cS}_n(\Si)&\longrightarrow \rA,
\end{align*}
where $\A$ and $\rA$ are quantum tori defined by anti-symmetric matrices $\mathsf{P}_\lambda$
and $\overline{\mathsf{P}}_\lambda$ respectively (see Section~\ref{sec:quantum_trace}).
It was shown in \cite{LY22,LY23} that ${\rm tr}_{\lambda}^A$ is injective, and $\overline{\rm tr}_{\lambda}^A$ is injective when $n=2$ or when $\Si$ is a polygon
(see Theorem~\ref{traceA}).

When the quantum parameter $\hat q$ is a root of unity, all irreducible representations of $\cS_n(\Si)$ has dimensions not more than the square root of
$\rankZ{\cS}_n(\Si)$ \cite{KW24,Wan23}, where 
$$\rankZ{\cS}_n(\Si)=\dim_{\text{Frac}( {\mathcal Z(\cS_n(\Si))})}\Big(\cS_n(\Si)\otimes_{\mathbb C}\text{Frac}( \mathcal Z(\cS_n(\Si)))\Big).$$
Here $\text{Frac}(\mathcal Z(\cS_n(\Si)))$ is the 
the fractional field of $\mathcal Z(\cS_n(\Si))$.
In particular, there is a large family of irreduciblre representations of $\cS_n(\Si)$ whose dimensions are the square root of
$\rankZ{\cS}_n(\Si)$. The study of this special family of irreducible representations is an important and challenging question for (stated $\SL(n)$-)skein algebras; see, e.g., 
\cite{GJS19, KK22, Yu23a, FKBL23}.
A similar result holds for $\overline{\cS}_n(\Si)$ when $\Si$ is a polygon.

It follows from the following inclusions \cite{LY23}
(see Theorem~\ref{traceA})
\begin{align*}
  \cA^{+}_{\hat{q}}(\Sigma,\lambda)\subset  {\rm tr}_{\lambda}^A(\cS_n(\Si))\subset \A \\
  \overline{\cA}^{+}_{\hat{q}}(\Sigma,\lambda)\subset  \overline{\rm tr}_{\lambda}^A(\overline{\cS}_n(\Si))\subset \rA
\end{align*}
that $\rankZ \cS_n(\Sigma)=\rankZ \A$
and $\rankZ \overline{\cS}_n(\Sigma)=\rankZ \rA$ when the corresponding $A$-quantum trace map is injective \cite{KW24}.

In the rest of this subsection, we use the following symbols and notations. Let $\hat{q}^2$ be a root of unity of even order $m''$, and set $m'=m''/\gcd(m'',n)$. Suppose that $m'$ is even, and set $d=\gcd(m',2n)$ and $m =m'/d$. 

For a triangulable essentially bordered pb surface $\Sigma$ without interior punctures, 
suppose that its underlying compact surface $\overline\Sigma$ has $b$ boundary components and $t$ boundary components with even punctures.
Set $|W|:=(n-1)(\#\partial\Sigma)$ and $r(\Sigma) := \# (\partial \Sigma) - \chi(\Sigma)$, where $\chi(\Sigma):=-2g+2-b$ denotes the Euler characteristic of $\Sigma$. 

In contrast to \cite{KW24}, we consider the case when $m'$ is even in this paper. 
\begin{introthm}[Theorem~\ref{thm:rank}]\label{mainthm-tro-2}
Let $\Sigma$ be a triangulable essentially bordered pb surface without interior punctures, $\lambda$ be a triangulation of $\Sigma$. 
Suppose that $\overline\Sigma$ has $b$ boundary components and among them there are $t$ even boundary components. 
We assume $b=t$ when both of $m^\ast$ and $n$ are even and $m$ is odd.
We have 
\begin{align*}
    &\rankZ \cS_n(\Sigma)=\rankZ \A\\
    &= \begin{cases}
   2^{|W|-r(\Sigma)+t}d^{r(\Sigma)-t}m^{(n^2-1)r(\Sigma)-t(n-1)}   &\text{$m^\ast$ is odd and $n$ is odd}\\
   2^{-2g-2\lfloor\frac{b-t}{2}\rfloor} d^{r(\Sigma)-t}m^{(n^2-1)r(\Sigma)-t(n-1)} &\text{$m^\ast$ is odd and $n$ is  even}\\
2^{|W|-r(\Sigma)+t+(b-t)(1-n)}
d^{r(\Si)-t}
   m^{(n^2-1)r(\Sigma)-t(n-1)} &\text{$m^{\ast}$ is even and }\begin{cases}
   \text{$n$ is odd}\\
    \text{$n$  and $m$ are even}
   \end{cases}\\
   2^{|W|-r(\Sigma)+t}
d^{r(\Si)-t}
   m^{(n^2-1)r(\Sigma)-t(n-1)} &\text{$m^{\ast}$ and $n$ are even, $m$ and $n'$ are odd}\\
   2^{-2g}d^{r(\Si)-t}m^{(n^2-1)r(\Sigma)-t(n-1)} &\text{$m^{\ast}$ is even, and $n'$ is even}
\end{cases}
\end{align*}
where $d$ and $m$ are defined in Section~\ref{notation}. 
\end{introthm}

When we prove Theorem~\ref{mainthm-tro-2}, we divide it into four cases according to the parities of $m^\ast$
and $n$. The case when both $m^\ast$ and $n$ are even is the most technical one, which is divided into three subcases as listed in Theorem~\ref{mainthm-tro-2}.

When $n=2$, Theorem~\ref{mainthm-tro-2} partly recovers the PI-degrees of the stated ${\rm SL}(2)$-skein algebras given in \cite[Theorem 5.3]{Yu23}. 
When both of $m^\ast$ and $n'$ are even and $b=t$, there are non-trivial cental elements of stated ${\rm SL}(n)$-skein algebras (see the proof of Lemma~\ref{lem:X/X}), which did not appear in the previous works \cite{Yu23, KW24}.
In particular, for $n=2$, there is no such a case when both of $m^\ast$ and $n'$ are even. 

The isomorphism between stated $\SL(n)$-skein algebras and unrestricted quantum moduli algebras (a.k.a. graph algebras) in \cite{BFR23} should be extended to roots of unity case. 
For a family of surfaces, \cite{BFR23} and the above results imply that we can access to the representation theory of the graph algebra obtained from $\Sigma$. 
In particular, this implies that Theorem~\ref{mainthm-tro-2} generalizes \cite[Theorem 1.3]{BR24} to the case when genus is  greater than $0$ in an appropriate sense. 
Baseilhac--Faitg--Roche \cite{BFR24} are also working to extend Theorem 1.3 of \cite{BR24} to arbitrary genera and simple Lie algebras.

\begin{introthm}[Theorem~\ref{thm-PI-reducedA}]\label{mainthm-tro-3}
Assume that $m'$ is even and $n$ is odd.
Let $\Sigma$ be a triangulable essentially
bordered pb surface without interior punctures, and $\lambda=\mu$ be a triangulation of $\Sigma$
introduced in Section~\ref{sub:quantum-torus-reduced}. Then we have 
$$\rankZ \rA=
   2^{|W|-r(\Sigma)+t}d^{r(\Sigma)-t}m^{|\overline{V}_\lambda|-t(n-1)-(b-t)\lfloor\frac{n}{2}\rfloor}$$
where $d,m$ are defined in Section~\ref{notation} and $|\overline V_{\lambda}|=(n^2-1)r(\Sigma)-\binom{n}{2}(\#\partial\Sigma)$ given in \eqref{eq:cardinarity}. 

In particular, we have 
$\rankZ \overline{\cS}_n(\Si)=\rankZ \rA$ when $\Si$ is a polygon.
\end{introthm}

Note that all the results in \cite{KW24} are for the case when $m'$ is even.
In general, it is quite more difficult to deal with the even root of unity case than the odd root of unity case since it has many different aspects  depending on the parities of $m, m^\ast, n$ and $n'$. 
The novelty of this paper is to determine the exponent of $2$ in Theorems~\ref{mainthm-tro-2} and \ref{mainthm-tro-3}, which has no any clues from \cite{KW24}. 
Although we used many results of \cite{KW24} in this paper, we developed many original and technical propositions and lemmas to deal with the even root of unity case, such as Lemmas~\ref{lem-k2-zero}, \ref{eq;bar_sharp}, \ref{lem-parity-dif}, \ref{lem:X/X}, \ref{lem-k2-zero-reduced}, Propositions~\ref{prop5.4},  \ref{propY}, and etc.

Anti-symmetric matrices $\mathsf{P}_\lambda$
and $\overline{\mathsf{P}}_\lambda$ for $\cA$-quantum tori are defined in \eqref{eq-anti-matric-P-def} whose entries are in $n\mathbb Z$ (Lemma~\ref{lem:invertible_KH}).
\cite[Theorem 4.1]{New72} (Theorem~\ref{thm-decom-symmetric}) implies that 
there exist integral matrices $\mathsf{X}$ and $\overline{\mathsf{X}}$ whose determinants are $\pm 1$ such that 
\begin{align}
\mathsf{X}^T\mathsf{P}_\lambda\mathsf{X}=\text{diag}\left\{
    \begin{pmatrix}
        0   &n z_1\\
        -nz_1& 0
    \end{pmatrix},\cdots,
    \begin{pmatrix}
        0   & nz_k\\
        -nz_k& 0
    \end{pmatrix},
    0,\cdots,0
    \right\}
\end{align} and 
\begin{align}
\overline{\mathsf{X}}^T\overline{\mathsf{P}}_\lambda\overline{\mathsf{X}}=\text{diag}\left\{
    \begin{pmatrix}
        0   &n \bar z_1\\
        -n\bar z_1& 0
    \end{pmatrix},\cdots,
    \begin{pmatrix}
        0   & n\bar z_l\\
        -n\bar z_l& 0
    \end{pmatrix},
    0,\cdots,0
    \right\}.
\end{align}
where $z_i\mid z_{i+1}$ for $1\leq i\leq k-1$ and $\bar{z}_i\mid \bar{z}_{i+1}$ for $1\leq i\leq l-1$.
With the decomposition, we have explicit entries $z_i$ and $\bar{z}_i$ as an analog of \cite{BL07, KW24, Wan25}. 
\begin{introcor}[Corollary~\ref{prop-anti-decom-P}]
Let $\Sigma$ be a triangulable essentially bordered pb surface without interior punctures, $\lambda$ be a triangulation of $\Sigma$. 
We assume $b=t$ when $n$ is even.
With the anti-symmetric matrix decomposition of $\mathsf{P}_\lambda$ as in \eqref{anti-decom-P}. Then:
\begin{enumerate}
    \item When $n$ is odd, we have
    \begin{align*}
z_i=
\begin{cases}
w_i &\text{for $1\leq i\leq \dfrac{|W|}{2}$},\medskip\\
2 w_i &\text{for $\dfrac{|W|}{2}+1\leq i\leq \dfrac{(n^2-1)r(\Si) -b(n-1)}{2}$},\medskip\\
4 w_i &\text{for $\dfrac{(n^2-1)r(\Si) -b(n-1)}{2}+1\leq i\leq  \dfrac{(n^2-1)r(\Sigma)-t(n-1)}{2}$},\\
\end{cases}
\end{align*}
where 
\begin{align*}
    w_i=
\begin{cases}
1 &\text{for $1\leq i\leq \dfrac{r(\Si)-t}{2}$},\medskip\\
n &\text{for $\dfrac{r(\Si)-t}{2}+1\leq  i\leq  \dfrac{(n^2-1)r(\Sigma)-t(n-1)}{2}$},
\end{cases}
\end{align*}

\item When $n$ is even, we have
\begin{align*}
z_i=
\begin{cases}
1 &\text{for $1\leq i\leq \dfrac{r(\Sigma)-t-2g}{2}$},\medskip\\
2  &\text{for $\dfrac{r(\Sigma)-t-2g}{2}+1\leq i\leq \dfrac{r(\Sigma)-t}{2}$},\medskip\\
n &\text{for $\dfrac{r(\Sigma)-t}{2}+1\leq i\leq  \dfrac{|W|+2g}{2}$}\medskip\\
2n &\text{for $\dfrac{|W|+2g}{2}+1\leq i\leq  \dfrac{(n^2-1)r(\Sigma)-t(n-1)}{2}$},
\end{cases}
\end{align*}
where $|\overline V_{\lambda}|=(n^2-1)r(\Sigma)-\binom{n}{2}(\#\partial\Sigma)$ given in \eqref{eq:cardinarity} and 
$$\bar w_i=
\begin{cases}
1 &\text{for $1\leq i\leq \dfrac{r(\Si)-t}{2}$},\medskip\\
n &\text{for $\dfrac{r(\Si)-t}{2}+1\leq i \leq\dfrac{|\overline V_{\lambda}|-t(n-1)-(b-t)\lfloor\frac{n}{2}\rfloor}2$.}
\end{cases}$$
\end{enumerate}
\end{introcor}

Similarly, we also have matrix decompositions for $\overline{\mathsf P}_\lambda$
\begin{introcor}[Cororally~\ref{prop-anti-decom-P-reduced}]
Let $\Sigma$ be a triangulable essentially bordered pb surface without interior punctures, and $\lambda=\mu$ be a triangulation of $\Sigma$
introduced in Section~\ref{sub:quantum-torus-reduced}. 
Assume $n$ is odd.
With the
anti-symmetric matrix decomposition of $\overline{\mathsf{P}}_\lambda$ as in \eqref{anti-decom-barP}. 
We have
    \begin{align}\label{red-eq-def-zi}
\bar z_i=
\begin{cases}
\bar w_i &\text{for $1\leq i\leq \dfrac{(n-1)(\sharp\partial\Si)}{2}$},\medskip\\
2 \bar w_i &\text{for $\dfrac{(n-1)(\sharp\partial\Si)}{2}+1\leq i\leq \dfrac{|\overline V_{\lambda}|-t(n-1)-(b-t)\lfloor\frac{n}{2}\rfloor}2$.}
\end{cases}
\end{align}
\end{introcor}

\subsection{Further directions}
In the current setting (absence of interior punctures) with at least 2 boundary punctures in each connected component, it is known that the reduced stated skein algebra $\overline{\cS}_2(\Sigma)$ is isomorphic to the quantum cluster algebra associated with $\mathfrak{sl}_2$ \cite{Mul16, LY22}. 
It is also known that $\overline{\cS}_3(\Sigma)$ is contained in the quantum cluster algebra associated with $\mathfrak{sl}_3$ in the same setting, and it is conjectured that they are isomorphic \cite{IY23, LY23}. 
From these results, one can expect some relationship between $\overline{\cS}_n(\Sigma)$ and the quantum cluster algebra associated with $\mathfrak{sl}_n$. 
Moreover, the results of this paper would be helpful to understand finite dimensional irreducible representations of quantum cluster algebras associated with $\mathfrak{sl}_n$. 

We left some cases to compute the PI-degree of (reduced) stated $\SL(n)$-skein algebras because of simple descriptions. However, we expect similar results even in the left cases and we describe some lemmas in general setting with respect to the parities of $m, m^\ast, n$ and $n'$ for successive works. See, e.g., Remark~\ref{rem;generalize}.

\subsection*{Acknowledgements}
The authors are grateful to St\'ephane Baseilhac and Philippe Roche for sharing their results and a project in progress on (unrestricted) quantum moduli algebras when the authors were working on \cite{KW24}. 
H. K. was partially supported by JSPS KAKENHI Grant Number JP23K12976 and by The Sumitomo Foundation Grant Number 2402039. 
Most of this research was carried out while Z. W. was a Dual PhD student at the University of Groningen (The Netherlands) and Nanyang Technological University (Singapore), supported by a PhD scholarship from the University of Groningen and a research scholarship from Nanyang Technological University.  
Z. W. was supported by a KIAS Individual Grant (MG104701) at the Korea Institute for Advanced Study.

\section{Preliminaries}
In this section, we will recall some definitions and known results related to (reduced) stated $\SL(n)$-skein algebras defined in \cite{LS21,LY23}.

\subsection{Notations}\label{notation}
Suppose $n\geq 2$ is an integer.  
Let $\bN$ denote the set of all non-negative integers, and $\bZ_k:=\bZ/k\bZ$.

We will refer the following conditions as Condition $(\ast)$: \\  
Let $\hat q^2$ be a primitive $m''$-th root of unity.  
Let $d'$ denote the greatest common divisor of $n$ and $m''$, and set $m'=m''/d'$. Throughout the paper, we suppose $m'$ is even. Define $d$ to be the greatest common divisor of $2n$ and $m'$, and set $m =m'/d$. 
Let $m^\ast:=m'/2$, $d^\ast:=d/2$ and $n':=2n/d$.
Note that 
\begin{align*}
    \text{$n'$ is even}&\Leftrightarrow
    \text{$n$ is a multiple of $d$},\\
    \text{$n'$ is odd}&\Leftrightarrow
    \text{$n$ is not a multiple of $d$}.
\end{align*}

Our ground ring is a commutative domain $\cR$ with an invertible element $\hat{q}$. We set $q= \hat{q}^{2n^2}$ so that $q^{1/2n^2} = \hat q$, and define the following constants:
\begin{align*}
\mathbbm{c}_{i}= (-q)^{n-i} q^{\frac{n-1}{2n}},\quad
\mathbbm{t}= (-1)^{n-1} q^{\frac{n^2-1}{n}},\quad 
\mathbbm{a} =   q^{\frac{n+1-2n^2}{4}}.
\end{align*}

\subsection{Punctured bordered surfaces}
A {\bf punctured bordered surface} (or {\bf pb surface} for simplicity) $\Sigma$ is obtained from a compact oriented surface $\overline{\Si}$ by removing finite points, such that each connected boundary component of $\Sigma$ is diffeomorphic to an open interval, where each removed point is called a \textbf{puncture}. The puncture in the interior of $\overline{\Si}$ is called the {\bf interior puncture}, the one that lies in $\partial \overline{\Si}$ is called the {\bf boundary puncture}.
Each connected component of $\partial\Sigma$ is called a \textbf{boundary edge} of $\Sigma$, which is homeomorphic to $(0,1)$.
An {\bf essentially bordered pb surface}
is a pb surface such that every connected component has a non-empty boundary.

An \textbf{even boundary component} 
(resp. \textbf{odd boundary component}) 
of $\overline{\Sigma}$ is a connected boundary component of $\overline{\Sigma}$ that has even 
(resp. odd) 
number of boundary punctures of $\Sigma$.

The orientation of $\partial \Si$ induced by the orientation of $\Si$ is called the {\bf positive orientation} of $\partial \Si$. 
The one opposite to the positive orientation of $\partial\Si$ is called the {\bf negative orientation} of $\partial \Si$.
In this paper, we always assume that the orientations of depicted 
surfaces point to the readers, i.e., the surfaces are equipped with the orientations in counter-clockwise.  
The orientations of boundary components of a surface also have positive orientations, i.e., these are induced from the orientation of the surface.

An {\bf ideal arc} $c$ of $\Sigma$ is an embedding from $(0,1)$ to $\Sigma$ which can be extended to an immersion $[0,1]\to \overline{\Sigma}$ such that $c(0)$ and $c(1)$ are (possibly the same) punctures. We identify an ideal arc with its image on $\Si$.
An ideal arc is \textbf{trivial} if it is null-homotopic.

A pb surface 
$\Sigma$ is \textbf{triangulable} if every connected component of it has at least one ideal point and is neither the once- or twice-punctured sphere, the monogon, nor the bigon. 
A \textbf{(ideal) triangulation} of a triangulable pb surface $\Sigma$ is a maximal collection of non-trivial ideal arcs which are pairwise disjoint and pairwise non-isotopic. 
In this paper, we consider triangulations up to isotopy.

\subsection{$n$-webs and their diagrams}
In the thickened pb surface $\Si\times(-1,1)$, we identify $\Si$ with $\Si\times \{0\}$. For any point $(x,t)\in\Si\times(-1,1)$, its \textbf{height} is $t$. 

An {\bf $n$-web} $\alpha$ in $\Si\times(-1,1)$ is a disjoint union of oriented closed curves and a directed finite graph properly embedded into $\Si\times(-1,1)$, satisfying the following conditions:
\begin{enumerate}
    \item $\alpha$ has only $1$-valent or $n$-valent vertices. Each $n$-valent vertex is a source or a  sink. Let $\partial \alpha$ denote the set of $1$-valent vertices, called the \textbf{endpoints} of $\alpha$. For any boundary component $c$ of $\Si$, the points of $\partial\alpha\cap (c\times(-1,1))$ have distinct heights.
    \item Every edge of the graph is an embedded oriented  closed interval  in $\Si\times(-1,1)$.
    \item $\alpha$ is equipped with a transversal \textbf{framing}. 
    \item The set of half-edges at each $n$-valent vertex is equipped with a  cyclic order. 
    \item $\partial \alpha$ is contained in $\partial\Si\times (-1,1)$ and the framing at these endpoints is given by the positive direction of $(-1,1)$.
\end{enumerate}
Consider $n$-webs up to (ambient) \textbf{isotopy} which are continuous deformations of $n$-webs in their class. 
The empty $n$-web, denoted by $\emptyset$, is also considered as an $n$-web, with the convention that $\emptyset$ is only isotopic to itself. 

A {\bf state} of $\alpha$ is a map $s\colon\partial\alpha\rightarrow \{1,2,\cdots,n\}$. A {\bf stated $n$-web} in $\Si\times(-1,1)$ is an $n$-web equipped with a state.

The (stated) $n$-web $\alpha$ is in {\bf vertical position} if 
\begin{enumerate}
    \item the framing at any point is given by the positive direction of $(-1,1)$,
    \item $\alpha$ is in general position with respect to the projection  $\text{pr}\colon \Si\times(-1,1)\rightarrow \Si\times\{0\}$,
    \item at every $n$-valent vertex, the cyclic order of half-edges as the image of $\text{pr}$ is given by the positive orientation of $\Si$ (drawn counter-clockwise in pictures).
\end{enumerate}

For every (stated) $n$-web $\alpha$, we can isotope $\alpha$ to be in vertical position. For each boundary component $c$ of $\Si$, the heights of $\partial\alpha\cap (c\times(-1,1))$ determine a total order on  $c\cap \text{pr}(\alpha)$.
Then a {\bf (stated) $n$-web diagram} of $\alpha$ is $\text{pr}(\alpha)$ equipped with the usual over/under information at each double point and a total order on $c\cap \text{pr}(\alpha)$ for each connected component $c$ of $\partial\Si$.

A stated $n$-web diagram $\alpha$ is called {\bf negatively ordered}  if the linear order on $\alpha\cap c$, for each boundary component $c$ of $\Si$, is indicated by the negative orientation of $c$.

\subsection{Stated $\SL(n)$-skein algebras and their reduced version}\label{sec:skein}
Let $\fS_n$ denote the permutation group on the set $\{1,2,\cdots,n\}$. 
For $i\in\{1,2,\cdots,n\}$, let $\bar{i}$  denote $n+1-i$. 

\def\M {M,\cN}
\def \Sv{\cS_n(\Si)}

The \textbf{stated $\SL(n)$-skein algebra} $\cS_n(\Si)$ of $\Si$ is
the quotient module of the $\cR$-module freely generated by the set 
 of all isotopy classes of stated 
$n$-webs in $\Sigma\times (-1,1)$ subject to  relations \eqref{w.cross}-\eqref{wzh.eight} equipped with the multiplication by stacking, i.e., for any two stated $n$-webs $\alpha,\alpha'\in\Si\times(-1,1)$, the product $\alpha\alpha'$ is defined by stacking $\alpha$ on $\alpha'$.

\beq\label{w.cross}
q^{\frac{1}{n}} 
\raisebox{-.20in}{

\begin{tikzpicture}
\tikzset{->-/.style=

{decoration={markings,mark=at position #1 with

{\arrow{latex}}},postaction={decorate}}}
\filldraw[draw=white,fill=gray!20] (-0,-0.2) rectangle (1, 1.2);
\draw [line width =1pt,decoration={markings, mark=at position 0.5 with {\arrow{>}}},postaction={decorate}](0.6,0.6)--(1,1);
\draw [line width =1pt,decoration={markings, mark=at position 0.5 with {\arrow{>}}},postaction={decorate}](0.6,0.4)--(1,0);
\draw[line width =1pt] (0,0)--(0.4,0.4);
\draw[line width =1pt] (0,1)--(0.4,0.6);
\draw[line width =1pt] (0.4,0.6)--(0.6,0.4);
\end{tikzpicture}
}
- q^{-\frac {1}{n}}
\raisebox{-.20in}{
\begin{tikzpicture}
\tikzset{->-/.style=

{decoration={markings,mark=at position #1 with

{\arrow{latex}}},postaction={decorate}}}
\filldraw[draw=white,fill=gray!20] (-0,-0.2) rectangle (1, 1.2);
\draw [line width =1pt,decoration={markings, mark=at position 0.5 with {\arrow{>}}},postaction={decorate}](0.6,0.6)--(1,1);
\draw [line width =1pt,decoration={markings, mark=at position 0.5 with {\arrow{>}}},postaction={decorate}](0.6,0.4)--(1,0);
\draw[line width =1pt] (0,0)--(0.4,0.4);
\draw[line width =1pt] (0,1)--(0.4,0.6);
\draw[line width =1pt] (0.6,0.6)--(0.4,0.4);
\end{tikzpicture}
}
= (q-q^{-1})
\raisebox{-.20in}{

\begin{tikzpicture}
\tikzset{->-/.style=

{decoration={markings,mark=at position #1 with

{\arrow{latex}}},postaction={decorate}}}
\filldraw[draw=white,fill=gray!20] (-0,-0.2) rectangle (1, 1.2);
\draw [line width =1pt,decoration={markings, mark=at position 0.5 with {\arrow{>}}},postaction={decorate}](0,0.8)--(1,0.8);
\draw [line width =1pt,decoration={markings, mark=at position 0.5 with {\arrow{>}}},postaction={decorate}](0,0.2)--(1,0.2);
\end{tikzpicture}
},
\eeq 
\beq\label{w.twist}
\raisebox{-.15in}{
\begin{tikzpicture}
\tikzset{->-/.style=
{decoration={markings,mark=at position #1 with
{\arrow{latex}}},postaction={decorate}}}
\filldraw[draw=white,fill=gray!20] (-1,-0.35) rectangle (0.6, 0.65);
\draw [line width =1pt,decoration={markings, mark=at position 0.5 with {\arrow{>}}},postaction={decorate}](-1,0)--(-0.25,0);
\draw [color = black, line width =1pt](0,0)--(0.6,0);
\draw [color = black, line width =1pt] (0.166 ,0.08) arc (-37:270:0.2);
\end{tikzpicture}}
= \mathbbm{t}
\raisebox{-.15in}{
\begin{tikzpicture}
\tikzset{->-/.style=
{decoration={markings,mark=at position #1 with
{\arrow{latex}}},postaction={decorate}}}
\filldraw[draw=white,fill=gray!20] (-1,-0.5) rectangle (0.6, 0.5);
\draw [line width =1pt,decoration={markings, mark=at position 0.5 with {\arrow{>}}},postaction={decorate}](-1,0)--(-0.25,0);
\draw [color = black, line width =1pt](-0.25,0)--(0.6,0);
\end{tikzpicture}}
,
\eeq
\beq\label{w.unknot}
\raisebox{-.20in}{
\begin{tikzpicture}
\tikzset{->-/.style=
{decoration={markings,mark=at position #1 with
{\arrow{latex}}},postaction={decorate}}}
\filldraw[draw=white,fill=gray!20] (0,0) rectangle (1,1);
\draw [line width =1pt,decoration={markings, mark=at position 0.5 with {\arrow{>}}},postaction={decorate}](0.45,0.8)--(0.55,0.8);
\draw[line width =1pt] (0.5 ,0.5) circle (0.3);
\end{tikzpicture}}
= (-1)^{n-1} [n]\ 
\raisebox{-.20in}{
\begin{tikzpicture}
\tikzset{->-/.style=
{decoration={markings,mark=at position #1 with
{\arrow{latex}}},postaction={decorate}}}
\filldraw[draw=white,fill=gray!20] (0,0) rectangle (1,1);
\end{tikzpicture}}
,\ \text{where}\ [n]=\frac{q^n-q^{-n}}{q-q^{-1}},
\eeq
\beq\label{wzh.four}
\raisebox{-.30in}{
\begin{tikzpicture}
\tikzset{->-/.style=
{decoration={markings,mark=at position #1 with
{\arrow{latex}}},postaction={decorate}}}
\filldraw[draw=white,fill=gray!20] (-1,-0.7) rectangle (1.2,1.3);
\draw [line width =1pt,decoration={markings, mark=at position 0.5 with {\arrow{>}}},postaction={decorate}](-1,1)--(0,0);
\draw [line width =1pt,decoration={markings, mark=at position 0.5 with {\arrow{>}}},postaction={decorate}](-1,0)--(0,0);
\draw [line width =1pt,decoration={markings, mark=at position 0.5 with {\arrow{>}}},postaction={decorate}](-1,-0.4)--(0,0);
\draw [line width =1pt,decoration={markings, mark=at position 0.5 with {\arrow{<}}},postaction={decorate}](1.2,1)  --(0.2,0);
\draw [line width =1pt,decoration={markings, mark=at position 0.5 with {\arrow{<}}},postaction={decorate}](1.2,0)  --(0.2,0);
\draw [line width =1pt,decoration={markings, mark=at position 0.5 with {\arrow{<}}},postaction={decorate}](1.2,-0.4)--(0.2,0);
\node  at(-0.8,0.5) {$\vdots$};
\node  at(1,0.5) {$\vdots$};
\end{tikzpicture}}=(-q)^{\frac{n(n-1)}{2}}\cdot \sum_{\sigma\in \fS_n}
(-q^{\frac{1-n}n})^{\ell(\sigma)} \raisebox{-.30in}{
\begin{tikzpicture}
\tikzset{->-/.style=
{decoration={markings,mark=at position #1 with
{\arrow{latex}}},postaction={decorate}}}
\filldraw[draw=white,fill=gray!20] (-1,-0.7) rectangle (1.2,1.3);
\draw [line width =1pt,decoration={markings, mark=at position 0.5 with {\arrow{>}}},postaction={decorate}](-1,1)--(0,0);
\draw [line width =1pt,decoration={markings, mark=at position 0.5 with {\arrow{>}}},postaction={decorate}](-1,0)--(0,0);
\draw [line width =1pt,decoration={markings, mark=at position 0.5 with {\arrow{>}}},postaction={decorate}](-1,-0.4)--(0,0);
\draw [line width =1pt,decoration={markings, mark=at position 0.5 with {\arrow{<}}},postaction={decorate}](1.2,1)  --(0.2,0);
\draw [line width =1pt,decoration={markings, mark=at position 0.5 with {\arrow{<}}},postaction={decorate}](1.2,0)  --(0.2,0);
\draw [line width =1pt,decoration={markings, mark=at position 0.5 with {\arrow{<}}},postaction={decorate}](1.2,-0.4)--(0.2,0);
\node  at(-0.8,0.5) {$\vdots$};
\node  at(1,0.5) {$\vdots$};
\filldraw[draw=black,fill=gray!20,line width =1pt]  (0.1,0.3) ellipse (0.4 and 0.7);
\node  at(0.1,0.3){$\sigma_{+}$};
\end{tikzpicture}},
\eeq
where the encircled $\sigma_+$  denotes the minimum crossing positive braid representing a permutation $\sigma\in \fS_n$ and $\ell(\sigma)=\#\{(i,j)\mid 1\leq i<j\leq n,\ \sigma(i)>\sigma(j)\}$ denotes the length of $\sigma\in \fS_n$.

\beq
   \raisebox{-.30in}{
\begin{tikzpicture}
\tikzset{->-/.style=
{decoration={markings,mark=at position #1 with
{\arrow{latex}}},postaction={decorate}}}
\filldraw[draw=white,fill=gray!20] (-1,-0.7) rectangle (0.2,1.3);
\draw [line width =1pt](-1,1)--(0,0);
\draw [line width =1pt](-1,0)--(0,0);
\draw [line width =1pt](-1,-0.4)--(0,0);
\draw [line width =1.5pt](0.2,1.3)--(0.2,-0.7);
\node  at(-0.8,0.5) {$\vdots$};
\filldraw[fill=white,line width =0.8pt] (-0.5 ,0.5) circle (0.07);
\filldraw[fill=white,line width =0.8pt] (-0.5 ,0) circle (0.07);
\filldraw[fill=white,line width =0.8pt] (-0.5 ,-0.2) circle (0.07);
\end{tikzpicture}}
   = 
   \mathbbm{a} \sum_{\sigma \in \fS_n} (-q)^{\ell(\sigma)}\,  \raisebox{-.30in}{
\begin{tikzpicture}
\tikzset{->-/.style=
{decoration={markings,mark=at position #1 with
{\arrow{latex}}},postaction={decorate}}}
\filldraw[draw=white,fill=gray!20] (-1,-0.7) rectangle (0.2,1.3);
\draw [line width =1pt](-1,1)--(0.2,1);
\draw [line width =1pt](-1,0)--(0.2,0);
\draw [line width =1pt](-1,-0.4)--(0.2,-0.4);
\draw [line width =1.5pt,decoration={markings, mark=at position 1 with {\arrow{>}}},postaction={decorate}](0.2,1.3)--(0.2,-0.7);
\node  at(-0.8,0.5) {$\vdots$};
\filldraw[fill=white,line width =0.8pt] (-0.5 ,1) circle (0.07);
\filldraw[fill=white,line width =0.8pt] (-0.5 ,0) circle (0.07);
\filldraw[fill=white,line width =0.8pt] (-0.5 ,-0.4) circle (0.07);
\node [right] at(0.2,1) {$\sigma(n)$};
\node [right] at(0.2,0) {$\sigma(2)$};
\node [right] at(0.2,-0.4){$\sigma(1)$};
\end{tikzpicture}},\label{eq:sticking}
\eeq
\beq \label{wzh.six}
\raisebox{-.20in}{
\begin{tikzpicture}
\tikzset{->-/.style=
{decoration={markings,mark=at position #1 with
{\arrow{latex}}},postaction={decorate}}}
\filldraw[draw=white,fill=gray!20] (-0.7,-0.7) rectangle (0,0.7);
\draw [line width =1.5pt,decoration={markings, mark=at position 1 with {\arrow{>}}},postaction={decorate}](0,0.7)--(0,-0.7);
\draw [color = black, line width =1pt] (0 ,0.3) arc (90:270:0.5 and 0.3);
\node [right]  at(0,0.3) {$i$};
\node [right] at(0,-0.3){$j$};
\filldraw[fill=white,line width =0.8pt] (-0.5 ,0) circle (0.07);
\end{tikzpicture}}   = \delta_{\bar j,i }\,  \mathbbm{c}_{i} \raisebox{-.20in}{
\begin{tikzpicture}
\tikzset{->-/.style=
{decoration={markings,mark=at position #1 with
{\arrow{latex}}},postaction={decorate}}}
\filldraw[draw=white,fill=gray!20] (-0.7,-0.7) rectangle (0,0.7);
\draw [line width =1.5pt](0,0.7)--(0,-0.7);
\end{tikzpicture}},
\eeq
\beq \label{wzh.seven}
\raisebox{-.20in}{
\begin{tikzpicture}
\tikzset{->-/.style=
{decoration={markings,mark=at position #1 with
{\arrow{latex}}},postaction={decorate}}}
\filldraw[draw=white,fill=gray!20] (-0.7,-0.7) rectangle (0,0.7);
\draw [line width =1.5pt](0,0.7)--(0,-0.7);
\draw [color = black, line width =1pt] (-0.7 ,-0.3) arc (-90:90:0.5 and 0.3);
\filldraw[fill=white,line width =0.8pt] (-0.55 ,0.26) circle (0.07);
\end{tikzpicture}}
= \sum_{i=1}^n  \mathbbm{c}_{\bar i}^{-1}\, \raisebox{-.20in}{
\begin{tikzpicture}
\tikzset{->-/.style=
{decoration={markings,mark=at position #1 with
{\arrow{latex}}},postaction={decorate}}}
\filldraw[draw=white,fill=gray!20] (-0.7,-0.7) rectangle (0,0.7);
\draw [line width =1.5pt,decoration={markings, mark=at position 1 with {\arrow{>}}},postaction={decorate}](0,0.7)--(0,-0.7);
\draw [line width =1pt](-0.7,0.3)--(0,0.3);
\draw [line width =1pt](-0.7,-0.3)--(0,-0.3);
\filldraw[fill=white,line width =0.8pt] (-0.3 ,0.3) circle (0.07);
\filldraw[fill=black,line width =0.8pt] (-0.3 ,-0.3) circle (0.07);
\node [right]  at(0,0.3) {$i$};
\node [right]  at(0,-0.3) {$\bar{i}$};
\end{tikzpicture}},
\eeq
\beq\label{wzh.eight}
\raisebox{-.20in}{

\begin{tikzpicture}
\tikzset{->-/.style=

{decoration={markings,mark=at position #1 with

{\arrow{latex}}},postaction={decorate}}}
\filldraw[draw=white,fill=gray!20] (-0,-0.2) rectangle (1, 1.2);
\draw [line width =1.5pt,decoration={markings, mark=at position 1 with {\arrow{>}}},postaction={decorate}](1,1.2)--(1,-0.2);
\draw [line width =1pt](0.6,0.6)--(1,1);
\draw [line width =1pt](0.6,0.4)--(1,0);
\draw[line width =1pt] (0,0)--(0.4,0.4);
\draw[line width =1pt] (0,1)--(0.4,0.6);
\draw[line width =1pt] (0.4,0.6)--(0.6,0.4);
\filldraw[fill=white,line width =0.8pt] (0.2 ,0.2) circle (0.07);
\filldraw[fill=white,line width =0.8pt] (0.2 ,0.8) circle (0.07);
\node [right]  at(1,1) {$i$};
\node [right]  at(1,0) {$j$};
\end{tikzpicture}
} =q^{-\frac{1}{n}}\left(\delta_{{j<i} }(q-q^{-1})\raisebox{-.20in}{

\begin{tikzpicture}
\tikzset{->-/.style=

{decoration={markings,mark=at position #1 with

{\arrow{latex}}},postaction={decorate}}}
\filldraw[draw=white,fill=gray!20] (-0,-0.2) rectangle (1, 1.2);
\draw [line width =1.5pt,decoration={markings, mark=at position 1 with {\arrow{>}}},postaction={decorate}](1,1.2)--(1,-0.2);
\draw [line width =1pt](0,0.8)--(1,0.8);
\draw [line width =1pt](0,0.2)--(1,0.2);
\filldraw[fill=white,line width =0.8pt] (0.2 ,0.8) circle (0.07);
\filldraw[fill=white,line width =0.8pt] (0.2 ,0.2) circle (0.07);
\node [right]  at(1,0.8) {$i$};
\node [right]  at(1,0.2) {$j$};
\end{tikzpicture}
}+q^{\delta_{i,j}}\raisebox{-.20in}{

\begin{tikzpicture}
\tikzset{->-/.style=

{decoration={markings,mark=at position #1 with

{\arrow{latex}}},postaction={decorate}}}
\filldraw[draw=white,fill=gray!20] (-0,-0.2) rectangle (1, 1.2);
\draw [line width =1.5pt,decoration={markings, mark=at position 1 with {\arrow{>}}},postaction={decorate}](1,1.2)--(1,-0.2);
\draw [line width =1pt](0,0.8)--(1,0.8);
\draw [line width =1pt](0,0.2)--(1,0.2);
\filldraw[fill=white,line width =0.8pt] (0.2 ,0.8) circle (0.07);
\filldraw[fill=white,line width =0.8pt] (0.2 ,0.2) circle (0.07);
\node [right]  at(1,0.8) {$j$};
\node [right]  at(1,0.2) {$i$};
\end{tikzpicture}
}\right),
\eeq
where   
$\delta_{j<i}= 
\begin{cases}
1  & j<i\\
0 & \text{otherwise}
\end{cases},\ 
\delta_{i,j}= 
\begin{cases} 
1  & i=j\\
0  & \text{otherwise}
\end{cases}$, and 
in each picture, a white dot denotes an arbitrarily direction of a part of the $n$-web diagram and the other white dots denote the same directions with respect to the $n$-valent vertices or the boundary edge. The black dot in Equation~\ref{wzh.seven} represents the direction which is opposite with the white dot with respect to the boundary edge.
Each arrow on the boundary edge represents the height order of the endpoints of the $n$-webs with increasing heights in the direction of the arrow.

Let $f$ be a proper embedding $\Sigma_1\rightarrow\Sigma_2$ between two pb surfaces. There is an $\cR$-module homomorphism $f_*\colon\cS_n(\Si_1)\rightarrow \cS_n(\Si_2)$ defined as following.
Let $\alpha$ be a negatively ordered stated $n$-web diagram in $\Si_1$, define $f_*(\alpha)\in \cS_n(\Si_2)$ represented by the negatively ordered stated $n$-web diagram $f(\alpha)$ \cite{LY23}.

For a boundary puncture $p$ of a pb surface $\Sigma$, let $C(p)_{ij}$ and $\cev{C}(p)_{ij}$ be stated corner arcs depicted as in Figure \ref{Fig;badarc}.
For a boundary puncture $p$ which is not on a monogon component of $\Sigma$, set 
$$C_p=\{C(p)_{ij}\mid i<j\},\quad\cev{C}_p=\{\cev{C}(p)_{ij}\mid i<j\}.$$  
Each element of $C_p\cup \cev{C}_p$ is called a \emph{bad arc} at $p$. 
\begin{figure}[h]
    \centering
    \includegraphics[width=150pt]{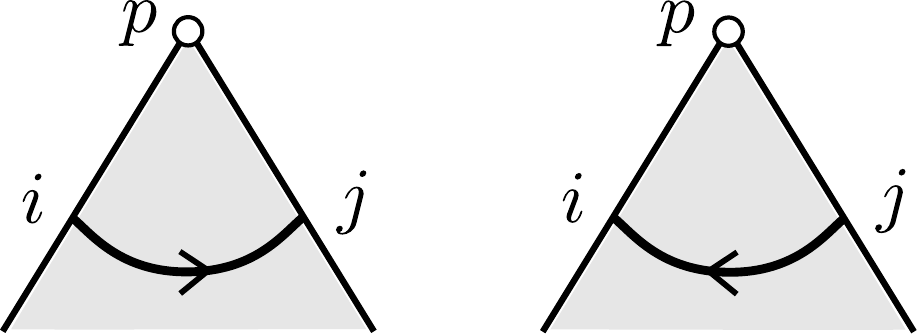}
    \caption{The left is $C(p)_{ij}$ and the right is $\cev{C}(p)_{ij}$.}\label{Fig;badarc}
\end{figure}

\def\barS{\overline{\cS}}
For a pb surface $\Sigma$, 
the \textbf{reduced stated $\SL(n)$-skein algebra} $\barS_n(\Sigma)$ is the quotient algebra of $\cS_n(\Sigma)$ by the two-sided ideal generated by bad arcs, defined in \cite{LY23}.

\subsection{Webs on the bigon}\label{sec:bigon}

The {\bf bigon} $\fB$ is obtained from a closed disk $D$ by removing two points from $\partial D$. We can label the two boundary components of the bigon  by $e_l$ and $e_r$. A bigon with this labeling is called a {\bf directed bigon},
 see an example
$
\raisebox{-.20in}{

\begin{tikzpicture}
\tikzset{->-/.style=

{decoration={markings,mark=at position #1 with

{\arrow{latex}}},postaction={decorate}}}

\filldraw[draw=black,fill=gray!20] (0.5 ,0.5) circle (0.5);
\filldraw[draw=black,fill=white] (0.5,0) circle (0.05);
\filldraw[draw=black,fill=white] (0.5,1) circle (0.05);
\node [left] at(0,0.5) {$e_{l}$};
\node [right] at(1,0.5) {$e_{r}$};
\end{tikzpicture}
}
$. 
We can draw $\fB$ like $
\raisebox{-.20in}{

\begin{tikzpicture}
\tikzset{->-/.style=

{decoration={markings,mark=at position #1 with

{\arrow{latex}}},postaction={decorate}}}

\filldraw[draw=white,fill=gray!20] (0,0) rectangle (1, 1);
\draw[line width =1pt] (0,0)--(0,1);
\draw[line width =1pt] (1,0)--(1,1);
\end{tikzpicture}
}
$, and use $u_{ij}$
to denote 
$
\raisebox{-.20in}{

\begin{tikzpicture}
\tikzset{->-/.style=

{decoration={markings,mark=at position #1 with

{\arrow{latex}}},postaction={decorate}}}

\filldraw[draw=white,fill=gray!20] (0,0) rectangle (1, 1);
\draw [line width =1pt,decoration={markings, mark=at position 0.5 with {\arrow{>}}},postaction={decorate}](0,0.5)--(1,0.5);
\draw[line width =1pt] (0,0)--(0,1);
\draw[line width =1pt] (1,0)--(1,1);
\node [left] at(0,0.5) {$i$};
\node [right] at(1,0.5) {$j$};
\end{tikzpicture}
}
$, where $i,j\in\{1,2,\cdots,n\}$.

Let $\binom{\bJ}{k}$ denote the set of all $k$-element subsets of $\bJ =\{1,\dots, n\}$. 
For $I=\{i_1,\cdots,i_k\},J=\{j_1,\cdots,j_k\}\in \binom{\bJ}{k}$, define 
\begin{align}\label{def-MIJ-u}
    M^I_J(\mathbf{u})=
    \sum_{\sigma\in S_k} (-q)^{\ell(\sigma)}u_{i_{i}j_{\sigma(1)}}\cdots u_{i_k j_{\sigma(k)}}\in \cS_n(\fB).
\end{align}

\section{Quantum tori in higher Teichm\"uller theory}
\label{sec-quantumtrace}
\subsection{Quantum tori and monomial subalgebras}
Assume \(m\) is a positive integer and let \(\mathsf{P}\) be an \(m\times m\) antisymmetric matrix.  
For a ground ring \(\mathcal{R}\), the \textbf{quantum torus} \(\mathbb{T}(\mathsf{P})\) associated with \(\mathsf{P}\) is defined by  
\[
\mathbb{T}(\mathsf{P})
   := \mathcal{R}\big\langle x_1^{\pm1},x_2^{\pm1},\dots,x_m^{\pm1}\big\rangle
      \big/\!\bigl(x_i x_j = \hat q^{\,2\mathsf{P}_{ij}} x_j x_i,\ 1\le i,j\le m \bigr).
\]

For \(\mathbf{k}=(k_1,\dots,k_m)\in\mathbb{Z}^m\), the \textbf{Weyl–normalized monomial} is  
\begin{align}
x^{\mathbf{k}}
   := [x_1^{k_1}x_2^{k_2}\cdots x_m^{k_m}]_{\mathrm{Weyl}}
   = \hat q^{-\sum_{1\le i<j\le m} k_i k_j \mathsf{P}_{ij}}
     x_1^{k_1}x_2^{k_2}\cdots x_m^{k_m}.
\label{eq:Weyl}
\end{align}
The set \(\{x^{\mathbf{k}}\mid \mathbf{k}\in\mathbb{Z}^m\}\) forms an \(\mathcal{R}\)-basis of \(\mathbb{T}(\mathsf{P})\).

For \(\mathbf{k}_1,\mathbf{k}_2\in\mathbb{Z}^m\) we have the multiplication rules  
\begin{align*}
x^{\mathbf{k}_1}x^{\mathbf{k}_2}
   = \hat q^{\langle \mathbf{k}_1,\mathbf{k}_2\rangle_{\mathsf{P}}}
      x^{\mathbf{k}_1+\mathbf{k}_2},\qquad 
x^{\mathbf{k}_1}x^{\mathbf{k}_2}
   = \hat q^{2\langle \mathbf{k}_1,\mathbf{k}_2\rangle_{\mathsf{P}}}
      x^{\mathbf{k}_2}x^{\mathbf{k}_1},
\end{align*}
where \(\langle \mathbf{k}_1,\mathbf{k}_2\rangle_{\mathsf{P}}
      := \mathbf{k}_1 \mathsf{P}\mathbf{k}_2^{\!T}\),
treating the \(\mathbf{k}_i\) as row vectors and \(\mathbf{k}_2^{\!T}\) as the transpose of \(\mathbf{k}_2\).

Let \(\mathbb{T}_{+}(\mathsf{P})\) be the \(\mathcal{R}\)-submodule of \(\mathbb{T}(\mathsf{P})\) generated by
\(\{x^{\mathbf{k}}\mid \mathbf{k}\in\mathbb{N}^m\}\).

For a submonoid \(\Lambda\subset\mathbb{Z}^m\), the \textbf{\(\Lambda\)-monomial subalgebra} of \(\mathbb{T}(\mathsf{P})\) is  
\[
\mathbb{T}(\mathsf{P};\Lambda)
   := \mathcal{R}\text{-span}\bigl\{x^{\mathbf{k}}\mid \mathbf{k}\in\Lambda\bigr\}
   \subset \mathbb{T}(\mathsf{P}).
\]

Let \(\mathcal{Z}\bigl(\mathbb{T}(\mathsf{P})\bigr)\) denote the center of \(\mathbb{T}(\mathsf{P})\).

\begin{lem}\cite[Lemma 3.1]{KW24}\label{quantum}
     If $\hat{q}^2$ is a root of unity of order $d$, we have 
    $\mathcal{Z}(\mathbb{T}(\mathsf{P})) = \cR\text{-span}\{x^{\bf k}\mid \langle {\bf k},{\bf k}'\rangle_{\mathsf{P}}=0\text{ in } \mathbb{Z}_d,\forall {\bf k}'\in\mathbb{Z}^m  \},$ 
    where $\mathbb{Z}_d:=\mathbb{Z}/d\mathbb{Z}$.
\end{lem}

\begin{lem}\cite[Lemma 3.2]{KW24}\label{PI}
    From Lemma \ref{quantum},  
    $\mathcal Z(\mathbb T(\mathsf{P})) = \{x^{\mathbf{k}}\mid \mathbf{k}\in\Lambda\}$, where $\Lambda$ is a submonoid of $\mathbb Z^m$. Suppose $\mathbb Z^m/\Lambda=\{ \mathbf{k}+\Lambda\mid \mathbf{k}\in S\subset \mathbb Z^m\}$. 
    If $|S|=|\frac{\mathbb Z^m}{\Lambda}|$ is finite, then the rank of $\mathbb T(\mathsf{P})$ as a $\mathcal Z(\mathbb T(\mathsf{P}))$-module is $|\frac{\mathbb Z^m}{\Lambda}|$.
\end{lem}

\subsection{Quantum torus frames}\label{sub-sec-torus-frame}

Let \(A\) be an \(\mathcal{R}\)-domain and let \(S \subset A\) be a set of nonzero elements.  
Define \(\mathsf{Pol}(S)\) as the \(\mathcal{R}\)-subalgebra of \(A\) generated by \(S\), and let  
\[
\mathsf{LPol}(S)
   := \bigl\{\,a \in A \,\bigm|\, \exists\ \text{\(S\)-monomial } m \text{ such that } a m \in \mathsf{Pol}(S)\bigr\}.
\]
If \(A = \mathsf{LPol}(S)\), we say that \(S\) \textbf{weakly generates} \(A\).

\begin{dfn}[\cite{LY23}]
Let \(A\) be an \(\mathcal{R}\)-domain and \(S = \{a_1,\dots,a_r\} \subset A\) a finite set.
We call \(S\) a \textbf{quantum torus frame} for \(A\) if the following conditions hold:
\begin{enumerate}
    \item for $1\leq i,j\leq r$, $a_ia_j=\hat q^{2k_{ij}} a_ja_i\neq 0$ for some integer $k_{ij}$;
    \item \(S\) weakly generates \(A\);
    \item the set \(\{a_1^{n_1}\cdots a_r^{n_r}\mid n_i \in \mathbb{N}\}\) is \(\mathcal{R}\)-linearly independent.
\end{enumerate}
\end{dfn}

\subsection{Weight lattice of the Lie algebra $\mathfrak{sl}_n(\bC)$}\label{sec:weight}
Let $\mathbbm{J}=\{1,2,\dots, n\}$. 
For a subset $\mathbf{i}=\{i_1,i_2,\dots, i_k\}\subset \mathbbm{J}$, 
let $\bar{\mathbf{i}}=\{\bar{i}_1,\bar{i}_2,\dots, \bar{i}_k\}$ and 
$\bar{\mathbf{i}}^c=\mathbbm{J}\setminus \bar{\mathbf{i}}$. 

Let $\mathsf{L}$ be the weight lattice of the Lie algebra $\mathfrak{sl}_n(\bC)$, i.e., the free abelian group generated by $\mathsf{w}_1,\cdots, \mathsf{w}_n$, modulo $\mathsf{w}_1+\cdots+ \mathsf{w}_n=0$. 

Consider a symmetric bilinear form defined by 
\begin{align}
\langle \mathsf{w}_i, \mathsf{w}_j\rangle:=\delta_{ij}-1/n. \label{eq;bracket}   
\end{align}

Set
\begin{align}\label{eq:varpi}
\varpi_i := \mathsf{w}_1 + \dots  + \mathsf{w}_i,\quad i = 1, \dots, n - 1
\end{align}
Then, we have 
\begin{align}\label{eq:varpi_varpi}
\langle \varpi_i, \varpi_j\rangle = \min\{i, i'\}-ii'/n,
\end{align}

For an edge $e$ of $\partial \Sigma$ and a stated $n$-web diagram $\alpha$, define
\begin{align}\label{eq-def-degree-d}
\mathbbm{d}_{e}(\alpha)=\sum_{x\in \partial\alpha\cap e}\mathsf{w}^{\ast}_{\overline{s(x)}}\in \mathsf{L},
\end{align}
where $s(x)$ is the state at $x$, and 
$\mathsf{w}_{\overline{s(x)}}^{\ast}=\begin{cases}
\mathsf{w}_{\overline{s(x)}}  & \text{if $\alpha$ points out of $\Sigma$ at $x$},\\
-\mathsf{w}_{s(x)} & \text{if $\alpha$ points into $\Sigma$ at $x$}.
\end{cases}$

\subsection{$n$-triangulation}

\begin{figure}[h]
    \centering
    \includegraphics[width=260pt]{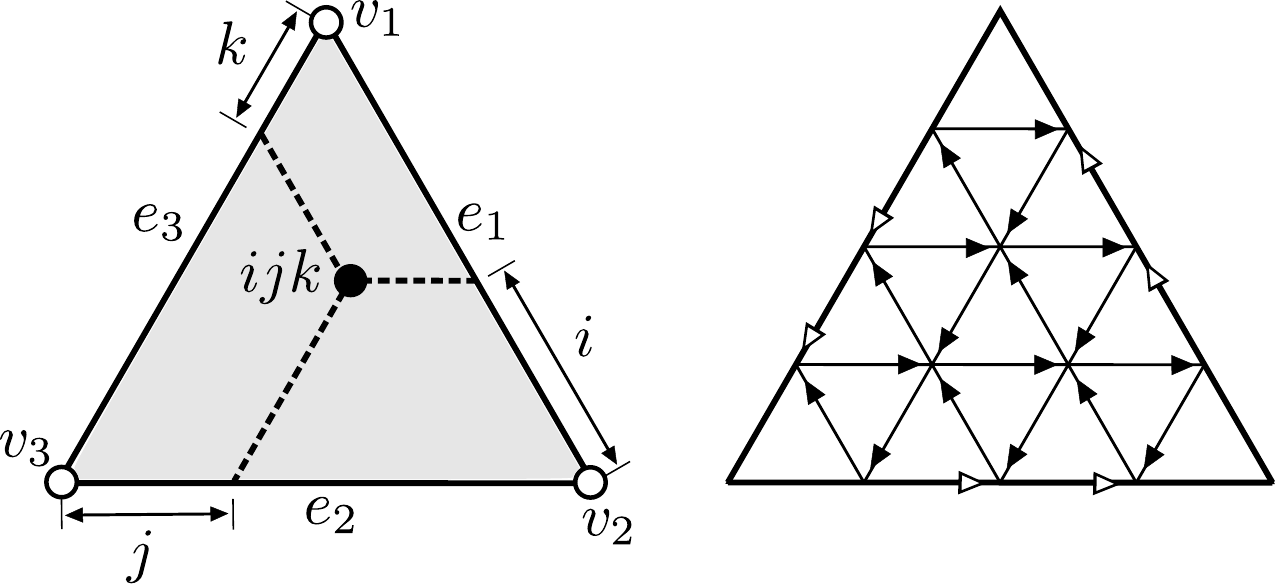}
    \caption{Barycentric coordinates $ijk$ and a $4$-triangulation with its quiver}\label{Fig;coord_ijk}
\end{figure}

Consider barycentric coordinates of $\bP_3$ so that
\begin{equation}
\bP_3=\{(i,j,k)\in\bR^3\mid i,j,k\geq 0,\ i+j+k=n\}\setminus\{(0,0,n),(0,n,0),(n,0,0)\}, 
\end{equation}
where $(i,j,k)$ (or $ijk$ for simplicity) are the barycentric coordinates. 
Let $v_1=(n,0,0)$, $v_2=(0,n,0)$, $v_3=(0,0,n)$. 
Let $e_i$ denote the edge on $\partial \bP_3$ whose endpoints are $v_i$ and $v_{i+1}$. 
See Figure \ref{Fig;coord_ijk}.

The \textbf{$n$-triangulation} of $\bP_3$ is the subdivision of $\bP_3$ into $n^2$ small triangles using lines $i,j,k=\text{constant integers}$. 
For the $n$-triangulation of $\bP_3$, the vertices and edges of all small triangles, except for $v_1,v_2,v_3$ and the small edges adjacent to them, form a quiver $\Gamma_{\bP_3}$.
An \textbf{arrow} is the direction of a small edge defined as follows. If a small edge $e$ is in the boundary $\partial\bP_3$ then $e$ has the counterclockwise direction of $\partial \bP_3$. If $e$ is in the interior then its direction is the same with that of a boundary edge of $\bP_3$ parallel to $e$. Assign weight $1$ to any boundary arrow and weight $2$ to any interior arrow.

The set $\barV=\barV_{\bP_3}$  
of points with integer barycentric coordinates of $\bP_3$:
\begin{align}
\barV=\barV_{\bP_3} = \{ijk \in \bP_3 \mid i, j, k \in \bZ\}.
\end{align}
Each element of $\barV$ is called a \textbf{small vertex}. 

\subsection{The (extended) Fock--Goncharov algebra and balanced part}\label{subsec:FGalg}
\def\barX{\overline{\cX}}

The triangle $\bP_3$ has a unique triangulation consisting of the 3 boundary edges up to isotopy, and we abuse ${\bP_3}$ to denote the triangulation. 

By splitting $\Sigma$ along all edges of a triangulation $\lambda$ not isotopic to a boundary edge, we have a disjoint union of ideal triangles. Each triangle is called a \textbf{face} of $\lambda$. Let $\cF_\lambda$ denote the set of faces. Then
\begin{equation}\label{eq.glue}
\Sigma = \Big( \bigsqcup_{\tau\in\cF_\lambda} \tau \Big) /\sim,
\end{equation}
where each face $\tau$ is a copy of $\bP_3$, and $\sim$ is the identification of edges of the faces to recover $\lambda$. 
Each face $\tau$ is characterized by a \textbf{characteristic map} $f_\tau\colon \tau \to \Sigma$, which is a homeomorphism when we restrict $f_\tau$ to $\Int\tau$ or the interior of each edge of $\tau$.

An \textbf{$n$-triangulation} of $\lambda$ is a collection of $n$-triangulations of the faces $\tau$ which are compatible with the gluing $\sim$,  where compatibility means, for any edges $b$ and $b'$ glued via $\sim$, the vertices on $b$ and $b'$ are identified. Consider the \textbf{reduced vertex set}
$$\barV_\lambda=\bigcup_{\tau\in\cF_\lambda}\barV_\tau, \quad \barV_\tau=f_\tau(\barV).$$
Each element of $\barV_\lambda$ is also called a \textbf{small vertex}. 
The images of the weighted quivers $\Gamma_\tau$ by $f_\tau$ form a quiver $\Gamma_\lambda$ on $\Sigma$.
Note that when edges $b$ and $b'$ are glued, a small edge on $b$ and the corresponding small edge of $b'$ have opposite directions, i.e., the resulting arrows are of weight $0$.

Let $\barQ_\lambda\colon \barV_\lambda\times \barV_\lambda \to \bZ$ be the signed adjacency matrix of the weighted quiver $\Gamma_\lambda$ defined by 
\begin{equation}
\barQ_\lambda(v,v') = \begin{cases} w \quad & \text{if there is an arrow from $v$ to $v'$ of weight $w$},\\
0 &\text{if there is no arrow between $v$ and $v'$}, 
\end{cases}
\end{equation}
especially we use $\barQ_{\bP_3}$ to denote $\barQ_\lambda$ when $\Sigma=\bP_3$. 

The ($n$-th root version) \textbf{Fock-Goncharov algebra} is the quantum torus associated with  $\barQ_\lambda$:
\begin{equation}
\barX_{\hat{q}}(\Sigma,\lambda)= \bT(\barQ_\lambda) = \cR \langle x_v^{\pm 1}, v \in \barV_\lambda \rangle / (x_v x_{v'}= \hat{q}^{\, 2 \barQ_\lambda(v,v')} x_{v'}x_v \text{ for } v,v'\in \barV_\lambda ).
\end{equation}

Consider a matrix $M_{\bP_3}\colon \barV\times\barV\to\bZ$ associated to ${\bP_3}$, and the induced map $f_\tau\colon \barV\to\barV_\lambda$ by the characteristic map. Define the \textbf{zero-extension} 
$M_\tau\colon\barV_\lambda\times\barV_\lambda\to\bZ$ of $M_{\bP_3}$ by
\begin{equation}
M_\tau(u,v)=\sum_{u'\in f_\tau^{-1}(u)}\sum_{v'\in f_\tau^{-1}(v)} M_{\bP_3}(u',v').
\end{equation}

When we see $\Gamma_\lambda$ as a union of copies of $\Gamma_{\bP_3}$, $\barQ_\lambda$ can be regarded as
\begin{equation}\label{eq-Q-def}
\barQ_\lambda = \sum_{\tau\in\cF_\lambda}\barQ_\tau,
\end{equation}
where $\barQ_\tau$ is the zero-extension of $\barQ_{\bP_3}$.

Attach an triangle $\bP_3$ to each boundary edge of $\Sigma$, and let $\Sigma^\ast$ be the resulting surface. See Figure \ref{Fig;attaching}.
Suppose the attaching edge is $e_1$ of $\bP_3$. 
For a triangulation $\lambda$ of $\Sigma$, let $\lambda^\ast$ be the triangulation of $\Sigma^\ast$ whose restriction to $\Sigma$ is $\lambda$. 
Let $\barV_{\lambda^\ast}$ denote the reduced vertex set of the extended $n$-triangulation. 
The $X$-vertex set $V_\lambda\subset\barV_{\lambda^\ast}$ is the subset of all small vertices not on $e_3$ in the attached triangles. 
The $A$-vertex set  $V'_\lambda\subset\barV_{\lambda^\ast}$ is the subset of all small vertices not on $e_2$ in the attached triangles.

\begin{lem}[{\cite[Lemma 11.2]{LY23}}]\label{lem:cardinarity}
Let $\Sigma$ be a triangulable pb surface and $\lambda$ be a triangulation of $\Sigma$. Let $\binom{n}{2}:=\frac{n(n-1)}{2}$ and $r(\Sigma)=\# \partial\Sigma- \chi(\Sigma)$, where $\chi(\Sigma)$ is the Euler characteristic.
Then we have 
\begin{align}
(a)\ |V_\lambda|=(n^2-1)r(\Sigma),\qquad (b)\ |\barV_\lambda|=(n^2-1)r(\Sigma)-\binom{n}{2}(\#\partial\Sigma). \label{eq:cardinarity}
\end{align}
\end{lem}
\begin{figure}[h]
    \centering
    \includegraphics[width=90pt]{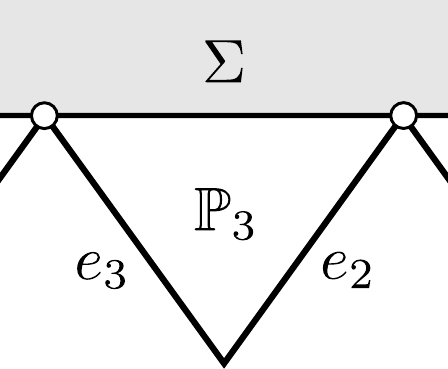}
    \caption{Attaching triangles.}\label{Fig;attaching}
\end{figure}

Let $\sfQ_\lambda\colon V_\lambda \times V_\lambda \to \bZ$ be the restriction of $ \barQ_{\lambda^\ast}\colon \barV_{\lambda^\ast} \times \barV_{\lambda^\ast} \to \bZ$. The \textbf{extended $X$-algebra} is defined as
$$\cX_{\hat{q}}(\Sigma,\lambda) = \bT(\sfQ_\lambda).$$
There is a natural identification of subalgebras $\barX_{\hat{q}}(\Sigma,\lambda)\subset\cX_{\hat{q}}(\Sigma,\lambda)\subset\barX_{\hat{q}}(\Sigma^\ast,\lambda^\ast)$.

Let $\proj_i\colon \barV \to\bZ\ (i=1,2,3)$ be the functions defined by
\begin{equation}
\proj_1(ijk)=i,\quad \proj_2(ijk)=j,\quad \proj_3(ijk)=k.\label{def:proj}
\end{equation}
Note that in \cite{LY23} they use $\bk_i$ instead of $\proj_i$. 
Let $\overline{\Lambda}_{\bP_3}\subset\bZ^{\barV}$ be the subgroup generated by $\proj_1,\proj_2,\proj_3$ and $(n\bZ)^{\barV}$. Elements in $\overline{\Lambda}_{\bP_3}$ are called \textbf{balanced}.

A vector $\bk\in\bZ^{\barV_\lambda}$ is \textbf{balanced} if its pullback $f_\tau^\ast\bk$ to ${\bP_3}$ is balanced for every face of $\lambda$, where for every face $\tau$ and its characteristic map  $f_\tau\colon\tau\to\Sigma$, the pullback $f_\tau^\ast\bk$ is a vector $\barV\to\bZ$ given by $f_\tau^\ast\bk(v)=\bk(f_\tau(v))$. 
Let $\overline{\Lambda}_\lambda$ denote the subgroup of $\bZ^{\barV_\lambda}$ generated by all the balanced vectors.

The \textbf{balanced Fock-Goncharov algebra} is the monomial subalgebra
$$\barX^{\rm bl}_{\hat{q}}(\Sigma,\lambda)=\bT(\barQ_\lambda;\overline{\Lambda}_\lambda).$$
Its extended version is defined as 
$$\cX^{\rm bl}_{\hat{q}}(\Sigma,\lambda)=\bT(\sfQ_\lambda)\cap \barX_{\hat{q}}^{\rm bl}(\Sigma^\ast,\lambda^\ast)=\bT(\sfQ_\lambda;\Lambda_\lambda),$$
where $\bT(\sfQ_\lambda)$ is considered a subalgebra of $\bT(\barQ_{\lambda^{\ast}})$ by the natural embedding, and $\Lambda_\lambda=\overline{\Lambda}_{\lambda^\ast}\cap\bZ^{V_\lambda}$ is the subgroup of balanced vectors.
Here, we regard $\bZ^{V_{\lambda}}$ as the zero-extension to a subgroup of $\bZ^{\barV_{\lambda^\ast}}$.

\subsection{The $A$-version quantum tori} \label{sec;A_tori}
In this subsection, assume $\Sigma$ has no interior punctures.

For a small vertex $v\in \barV_\lambda$ and an ideal triangle $\nu \in \cF_\lambda$, we define its \textbf{skeleton} $\skeleton_\tau(v)\in \bZ [\barV_\tau]$ as follows. 

For a face $\nu\in\cF_\lambda$ containing $v$, suppose $v=(ijk)\in V_\nu$. Draw a weighted directed graph $Y_v$ properly embedded into $\nu$ as in the left of Figure~\ref{Fig;skeleton}, where an edge of $Y_v$ has weight $i$, $j$ or $k$ according as the endpoint lying on the edge $e_1$, $e_2$ or $e_3$ respectively. 
\begin{figure}[h]
    \centering
    \includegraphics[width=380pt]{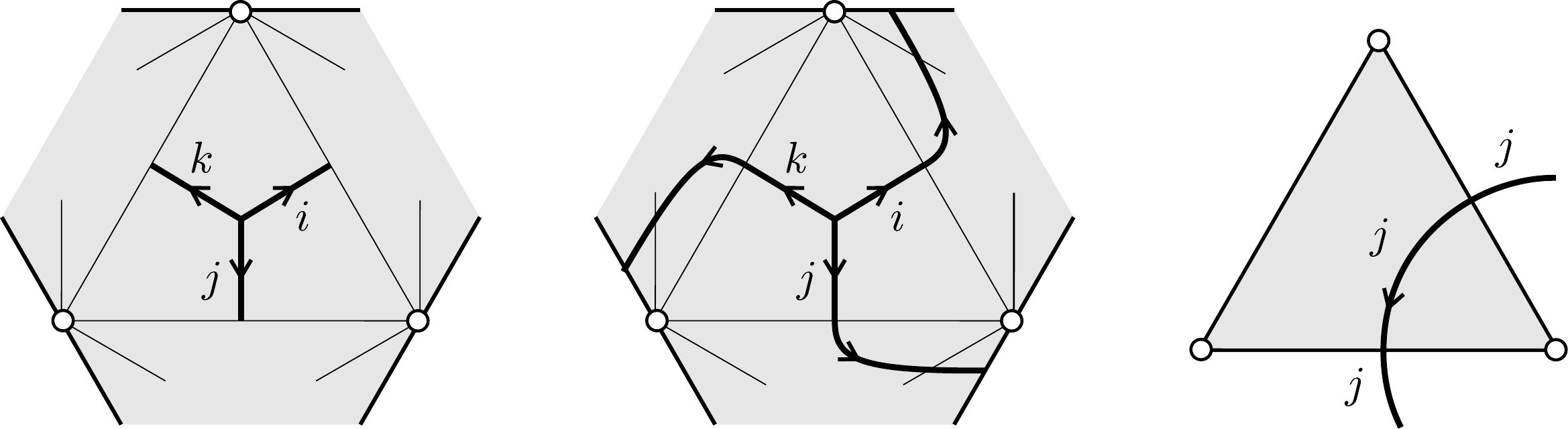}
    \caption{Left: Weighted graph $Y_v$,\quad Middle: Elongation $\widetilde{Y}_v$,\quad Right: Turning left}\label{Fig;skeleton}
\end{figure}

Elongate the nonzero-weighted edges of $Y_v$ to have an embedded weighted directed graph $\widetilde{Y}_v$ as drawn in the middle of Figure~\ref{Fig;skeleton}. Here, each edge is elongated by turning left whenever it enters a triangle. 
The part of the elongated edge in a triangle $\tau$ is called a \textbf{(arc) segment} of $\widetilde{Y}_v$ in $\tau$. In addition, we also regard $Y_v$ as a segment of $\widetilde{Y}_v$, called the \textbf{main segment}.

For the main segment $s=Y_v$, define $Y(s) = v \in \barV_\nu$. For an arc segment $s$ in a triangle $\tau$, define $Y(s)\in \barV_\tau$ to be the small vertex of the following weighted graph. 
$$
s=\begin{array}{c}\includegraphics[scale=0.27]{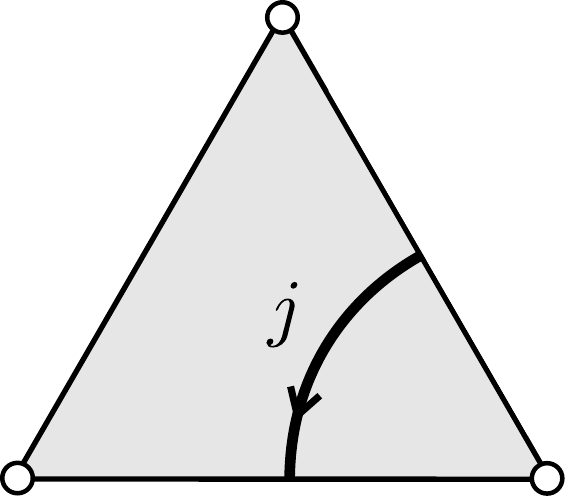}\end{array}
\longrightarrow\quad Y(s):=\begin{array}{c}\includegraphics[scale=0.27]{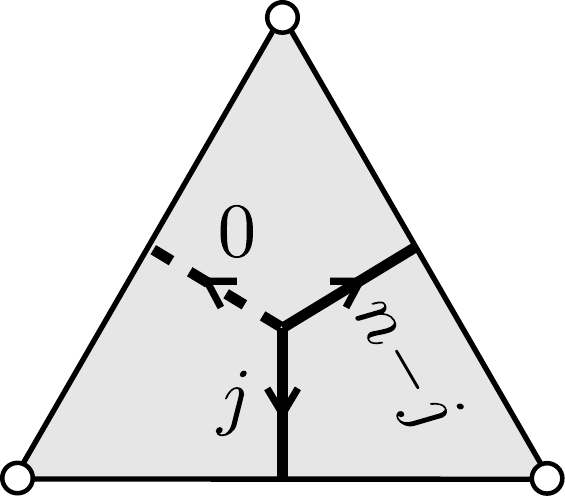}\end{array}
$$
Define $\skeleton_\tau(v)$ by
\begin{equation}
\skeleton_\tau(v) = \sum_{s \subset \tau\cap\tilde{Y}_v} Y(s) \in \bZ [\barVt],
\end{equation}
where the sum is over all the segments of $\widetilde{Y}_v$ in $\tau$.
It is shown that $\skeleton_\tau(v)$ is well-defined \cite[Lemma 11.4]{LY23}. 

Define a $\bZ_3$-invariant map
$$\barK_{\bP_3}\colon \barV\times\barV\to\bZ$$
such that if $v=ijk$ and $v'=i'j'k'$ satisfy $i'\leq i$ and $j'\geq j$ then 
\begin{align}
\barK_{\bP_3}(v,v')=jk'+ki'+i'j.
\end{align}
Under identifying $\bP_3$ and a face $\tau$ of $\lambda$, we also use $\barK_{\tau}$ as $\barK_{\bP_3}$. 

Recall that $\cF_\lambda$ denotes the set of all the faces of the triangulation $\lambda$. 
For $u, v \in \barVl$ and a face $\tau \in \cF_\lambda$ containing $v$,  define 
\begin{equation}\label{eq-surgen-exp}
\barK_{\lambda}(u,v)=\barKt(\skeleton_\tau(u),v)=\sum_{s\subset\tau\cap\tilde{Y}_u}\barKt(Y(s),v).
\end{equation}
It is known that $\barK_{\lambda}$ is well-defined \cite[Lemma 11.5]{LY23}.

Define a map $p\colon \barVlast\setminus\barVl\to\barVlast\setminus V'_\lambda$ as follows. 
Every $v \in\barVlast\setminus\barVl$ has coordinates $ijk$ in an attached triangle with $k\ne0$, and $\barVlast\setminus V'_{\lambda}$ consists of vertices $ijk$ in attached triangles with $i=0$. Then
\begin{equation}\label{eq-cov-pdef}
p(v)=(0,n-k,k)\quad \text{in the same triangle}.
\end{equation}
The change-of-variable matrix $\sfC\colon V'_\lambda\times\barVlast\to\bZ$ is given by
\begin{align*}
\sfC(v,v)&=1,&&v\in V'_\lambda,\nonumber\\
\sfC(v,p(v))&=-1,&&v\in V'_\lambda\setminus\barVl,\\
\sfC(v,v')&=0,&&\text{otherwise}.\nonumber
\end{align*}

For the extended triangulation $\lambda^\ast$ of $\lambda$, consider $\barK_{\lambda^\ast}\colon \barVlast\times\barVlast\to\bZ$. 
Note that the product $\sfC\barK_{\lambda^\ast}$ is a bilinear form on $V'_\lambda\times\barVlast$. 
The matrix $\sfK_\lambda$ is the restriction of $\sfC\barK_{\lambda^\ast}$:
\begin{equation}\label{eq-Klambda}
\sfK_\lambda=(\sfC\barK_{\lambda^\ast})|_{V'_\lambda\times V_\lambda}.
\end{equation}

From \cite[Lemma~11.9]{LY23}, we define the anti-symmetric integer matrices $\barP_{\lambda}$ and $\sfP_\lambda$ by 
\begin{align}\label{eq-anti-matric-P-def}
\barP_\lambda:=\barK_\lambda\barQ_\lambda\barK^t_\lambda,\qquad 
\sfP_\lambda:=\sfK_\lambda\sfQ_\lambda\sfK^t_\lambda
\end{align}

The following are \textbf{$A$-version quantum tori} and \textbf{quantum spaces} of $(\Sigma,\lambda)$:
\begin{align*}
\barA_{\hat{q}}(\Sigma,\lambda) =\bT(\bar{\sfP}_\lambda),\qquad
\overline{\cA^+_{\hat{q}}}(\Sigma,\lambda)=\bT_+(\bar{\sfP}_\lambda),\\
\cA_{\hat{q}}(\Sigma,\lambda)=\bT(\sfP_\lambda),\qquad
\cA_{\hat{q}}^{+}(\Sigma,\lambda)=\bT_+(\sfP_\lambda).
\end{align*}
In the following, let $a_v$ denote the generator of $\barA_{\hat{q}}(\Sigma,\lambda)$ (resp. $\cA_{\hat{q}}(\Sigma,\lambda)$) corresponding to $v\in\barV_\lambda$ (resp. $V'_\lambda$), and 
let $a^{\bf k}$ denote the Weyl normalized monomial of $\prod_{v}a_{v}^{{\bf k}(v)}$ for ${\bf k}\in \bZ^{\barV_\lambda}$ (resp. ${\bf k}\in\bZ^{V'_\lambda}$).

\begin{thm}[{\cite[Theorem 11.7]{LY23}}]\label{thm-transition-LY}
The $\cR$-linear maps 
\begin{align}
{\psi_\lambda}&\colon  \cAl \to \cXl, \quad  a^\mathbf{k}\mapsto x^{\mathbf{k} \sfK_{\lambda}}, \ ({\mathbf{k}} \in \bZ^{V'_{\lambda}})\\
\overline \psi_\lambda &\colon \barA_{\hat{q}}(\Sigma,\lambda)\to
\overline{\mathcal X}_{\hat q}(\Sigma,\lambda),
\quad  a^\mathbf{k}\mapsto x^{\mathbf{k} \overline{\mathsf{K}}_{\lambda}}, \ ({\mathbf{k}} \in \bZ^{\overline V_{\lambda}})
\end{align}
are $\cR$-algebra embeddings with $\im \psi_\lambda=\Xbll$ and 
$\im \overline\psi_\lambda=\barX^{\rm bl}_{\hat{q}}(\Sigma,\lambda)$. 
\end{thm}

\subsection{$\SL(n)$-quantum trace map}\label{sec:quantum_trace}

For a pb surface $\Sigma$ and $\cR = \bZ[\hat{q}^{\pm1}]$, 
it is known that there is a unique reflection $\omega \colon \cS_n (\Sigma)\to \cS_n (\Sigma)$ (i.e., $\omega(\hat q) = \hat q^{-1}$, $\omega$ is $\mathbb Z$-linear, $\omega^2$ is the identity, and $\omega(xy)= \omega(y) \omega(x)$ for any $x,y\in \cS_n (\Sigma)$) such that, for a stated $n$-web diagram $\alpha$, $\omega(\alpha)$ is defined from $\alpha$ by switching all the crossings and reversing the height order on each boundary edge \cite[Theorem 4.6]{LS21}. 

A stated $n$-web diagram $\alpha$ in a pb surface $\Sigma$ is \textbf{reflection-normalizable} if $\omega(\alpha) = \hat{q}^{2k}\alpha$ for $k \in \bZ$. Note that such $k$ is unique. Over any ground ring $\cR$ with a distingished invertible element $\hat q$, 
we define the \textbf{reflection-normalization} of $\alpha$ as $\hat{q}^{k}\alpha$.

Let $I,J\in \binom{\bJ}{k}$, and
let $a$ be an oriented $n$-web diagram in a pb surface $\Sigma$, and $N(a)$ be a small open tubular neighborhood of $a$ in $\Sigma$. By an identification of $\fB$ and $N(a)$ which induces $\cS_n(\fB)\to \cS_n(\Si)$, let $M^I_J(a)$ be the image of $M^I_J(\mathbf{u})$ (see \eqref{def-MIJ-u}) 
and depict it as
\begin{align}
\begin{array}{c}\includegraphics[scale=0.43]{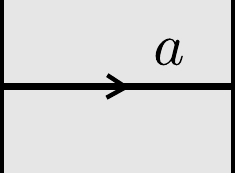}\end{array}
\longrightarrow I\!\begin{array}{c}\includegraphics[scale=0.43]{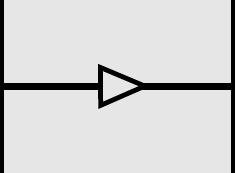}\end{array}\! J:=M^I_J(a). \label{eq:quantum_minor}
\end{align}

For $I\subset\bJ=\{1,\cdots,n\}$, define
$\bar{I}=\{\bar{i}\mid i\in I\},\ 
I^c=\bJ\setminus I,\  \bar{I}^c=(\bar{I})^c.$
Lemma 4.13 in \cite{LY23} claims that the stated $n$-web diagram 
\begin{align*}
\alpha=\begin{array}{c}\includegraphics[scale=0.27]{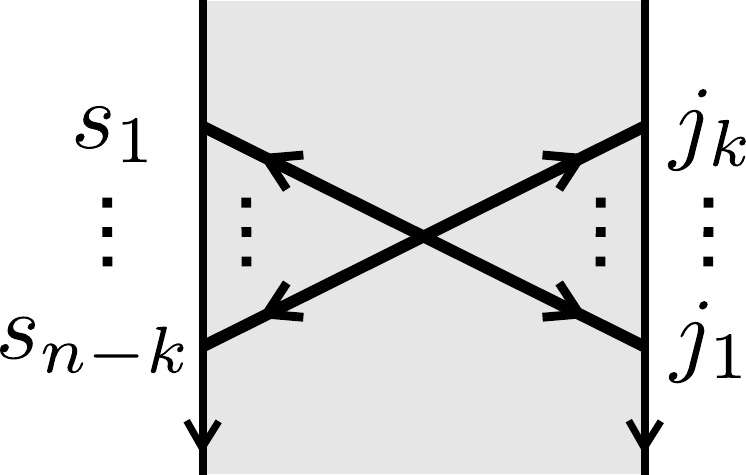}\end{array}. 
\end{align*}
is reflection-normalizable, and 
$M^{I}_{J}({\bf u})$ is equal to $\alpha$ up to a sign and a power of $q^{1/2n}$, where $J=\{j_1,j_2,\dots,j_k\},\bar{I}^c=\{s_1,s_2,\dots, s_{n-k}\}$.

For $v=(ijk) \in  \barV_\nu\subset \barV_\lambda$ with a triangle $\nu$ of $\lambda$,  consider the graph $\widetilde{Y}_v$ defined in Section~\ref{sec;A_tori}. 
By replacing a $k$-labeled edge of $\widetilde{Y}_v$ with $k$-parallel edges, we obtain a stated $n$-web $\gaa''_v$, adjusted by a sign: 
$$\widetilde{Y}_v=
\begin{array}{c}\includegraphics[scale=0.40]{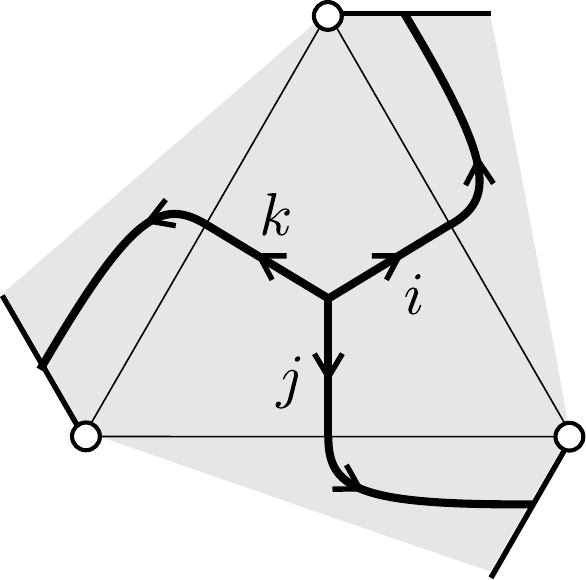}\end{array}\longrightarrow \gaa''_v:=(-1)^{\binom{n}{2}}\begin{array}{c}\includegraphics[scale=0.50]{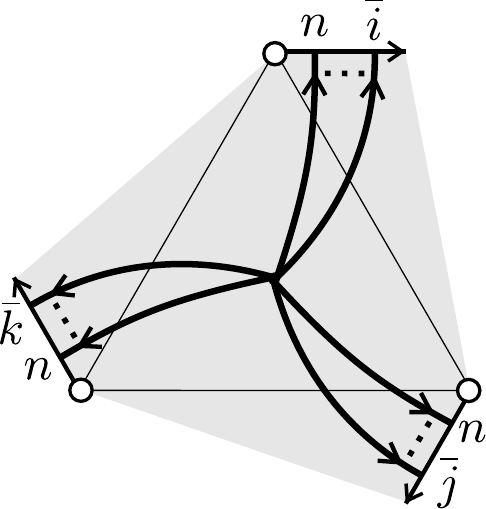}\end{array}. $$
It is known  $\gaa''_v$ is reflection-normalizable \cite[Lemma 4.12]{LY23}.

For an attached triangle $\nu=\bP_3$ and $v=(ijk)= \barV_{\nu}\subset V'_\lambda\setminus \barV_\lambda$, 
let $c$ denote the oriented corner arc of $\Sigma$ 
starting on $e$ and turning left all the time:  
$$
\begin{array}{c}\includegraphics[scale=0.35]{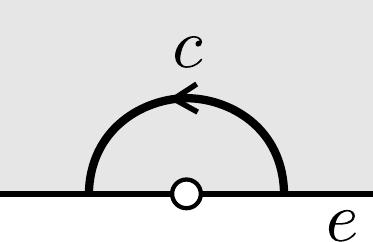}\end{array}\longrightarrow \gaa''_v:=M^{[j+1,j+i]}_{[\bar{i},n]}(c)\begin{array}{c}\includegraphics[scale=0.47]{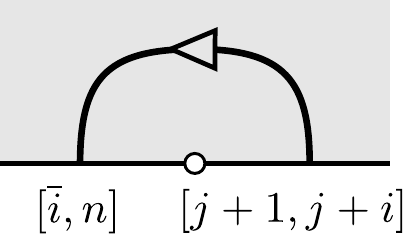}\end{array}. $$
Then it is known that the obtained element 
$\gaa''_v$ is reflection-normalizable \cite[Lemmas 4.13 and 4.10]{LY23}

Define $\gaa_v$ as the reflection-normalization of $\gaa_v''$. 
Let $\bar\gaa_v$ be the image of $\gaa_v$ in $\rSn$ by the projection $\Sn\to \rSn$. Note that $\bar\gaa_v = 0$ if $v \in V'_\lambda\setminus\barV_\lambda$.

Let us recall the quantum trace map for stated $\SL(n)$-skein algebras following \cite{LY23}. 

\begin{thm} [{\cite[Theorems 13.1 and 15.5]{LY23}}]\label{traceA}
Let $\Sigma$ be a triangulable essentially bordered pb surface without interior punctures, and let $\lambda$ be a triangulation of $\Sigma$. 

(a) There exists an algebra embedding $tr_{\lambda}^A:\cS_n(\Sigma)\rightarrow \mathcal A_{\hat{q}}(\Sigma,\lambda)$ such that $\tra(\gaa_v) = a_v$ for all $v\in V_{\lambda}'$ and 
$\psi_\lambda\circ {\rm tr}_{\lambda}^A=
{\rm tr}_{\lambda}^X.$ Furthermore, we have 
\begin{align}\label{eq-sandwitch-property}
\cA^{+}_{\hat{q}}(\Sigma,\lambda)\subset \tra(\Sn)\subset \cA_{\hat{q}}(\Sigma,\lambda).
\end{align}

(b) There exists an algebra homomorphism $\overline{\rm tr}_{\lambda}^A:\rSn\rightarrow \rA$ such that $\overline{\rm tr}_{\lambda}^A(\bar{\gaa}_v) = a_v$ for all $v\in \overline{V}_{\lambda}$
and 
$\overline{\psi}_\lambda\circ \overline{{\rm tr}}_{\lambda}^A=
\overline{{\rm tr}}_{\lambda}^X.$
Furthermore, we have
\begin{align}
    \rAp\subset \overline{\rm tr}_{\lambda}^A(\rSn)\subset \rA.
\end{align}
If $n=2,3$, or $n>3$ and $\Sigma$ is a polygon, then $\overline{\rm tr}_{\lambda}^A$ is injective.
\end{thm}

\begin{cor}[{\cite[Corollary 13.4]{LY23}}]\label{cor-quantum-frame}
    Let $\Sigma$ be an essentially bordered pb surface without interior punctures. 

    (a) The set $\{\gaa_v\mid v\in V_{\lambda}' \}$ is a quantum torus frame for $\cS_n(\Sigma)$ (see Section~\ref{sub-sec-torus-frame}).

    (b) If $\Sigma$ is a polygon, the set $\{\bar{\gaa}_v\mid v\in \overline{V}_{\lambda}\}$ is a quantum torus frame for $\overline{\cS}_n(\Sigma)$ (see Section~\ref{sub-sec-torus-frame}).
\end{cor}

\begin{cor}[{\cite[Corollary 12.2]{LY23}}]\label{cor:LY12}
Let $\Sigma$ be a triangulable pb surface with a triangulation $\lambda$, and let $\alpha$ be a stated web diagram. 
\begin{enumerate}
    \item Let $v$ be a small vertex on the boundary of a triangulable surface $\Sigma$. The image $\overline{{\rm tr}}^X_{\lambda}(\alpha)$ is homogeneous in $x_v$ of degree $n\langle \mathbbm{d}_e(\alpha), \varpi_i\rangle$. Here $v$ is the $i$-th small vertex on the boundary edge containing
$v$ if we list the boundary small vertices in the positive direction.

\item 
Let $v\in W$.
The image ${\rm tr}^X_{\lambda}(\alpha)$ is homogeneous in $x_v$ of degree $n\langle \mathbbm{d}_e(\alpha), \varpi_i\rangle$. Here $v$ is the $i$-th small vertex on the boundary edge of $\Sigma^{*}$ containing
$v$ if we list the boundary small vertices in the positive direction.
\end{enumerate}
\end{cor}
\begin{proof}
(1)  The claim is indeed that of Corollary 12.2 of \cite{LY23}. 
(2)  There exists an embedding $\iota\colon\Sigma\rightarrow\Sigma^{*}$ as illustrated in Figure 23 in \cite{LY23}.
  As mentioned in Section 12.3 of \cite{LY23}, we have 
  ${\rm tr}^X_{\lambda}(\alpha)=\overline{{\rm tr}}^X_{\lambda^{*}}(\iota(\alpha))$.
  Then (1) implies (2).
\end{proof}

\section{Center of stated $\SL(n)$-skein algebras}\label{sec-centers}

Throughout this section, assume $\Sigma$ is a triangulable essentially bordered pb surface and without interior punctures. 

It was shown in \cite{KW24} that 
$\cS_n(\Si)$ is an almost Azumaya algebra (Definition~\ref{def:almost_Azumaya}) when $\hat q$ is a root of unity.
It follows from the unicity theorem \cite{FKBL19,BG02} that a large family of finite dimensional irreducible representations of $\cS_n(\Si)$ are uniquely determined by their actions of the center of $\cS_n(\Si)$.
So the investigation of the center of $\cS_n(\Si)$ is very important to understand the representation theory of $\cS_n(\Si)$.
Thanks to \eqref{eq-sandwitch-property}, it suffices to study the center of $\A$ instead of that of $\cS_n(\Si)$.
In this section, we will formulate the center of $\A$ when $m'$ is even (Theorem~\ref{center_torus}), which is a complementary work for the case when $m'$ is odd in \cite{KW24}.
Then, in Section~\ref{sec-Unicity-Theorem}, we will use Theorem~\ref{center_torus} to compute the rank of $\A$ over its center, which equals the square of the dimension of all its irreducible representations.

\subsection{Ordered vertex sets}\label{sec:vertex}
For a triangulation $\lambda$ of a triangulable pb surface $\Sigma$, recall the small vertex set $\overline V_{\lambda}$ defined in Section \ref{subsec:FGalg}. Let $\obV_{\lambda}\subset \overline{V}_\lambda$ be the subset consisting of all the small vertices contained in the interior of $\Sigma$.

We review a convention from \cite{KW24} of how we order certain vertices in $\overline{V}_{\lambda^{\ast}}$.

Fix a triangulation $\lambda$ of an essentially bordered pb surface $\Sigma$ and consider the extended triangulation $\lambda^\ast$. 
Let $v_i,u_i,w_i\ (i=1,\dots,n-1)$ denote the vertices on the boundary of an attached triangle as in Figure \ref{Fig;coord_uvw}. 
Note that $v_i$'s  are on the attached edge $e_1$, and 
$$v_i=(n-i,i,0),\quad u_i=(i,0,n-i),\quad w_i=(0,i,n-i).$$ 
For two attached triangles $\tau$ and $\tau'$, we will use 
$v_i,u_i,w_i$ (resp. $v_i',u_i',w_i'$) to denote $(n-i,i,0), (i,0,n-i), (0,i,n-i)$ of $\tau$ (resp. $\tau'$).
\begin{figure}[h]
    \centering
    \includegraphics[width=140pt]{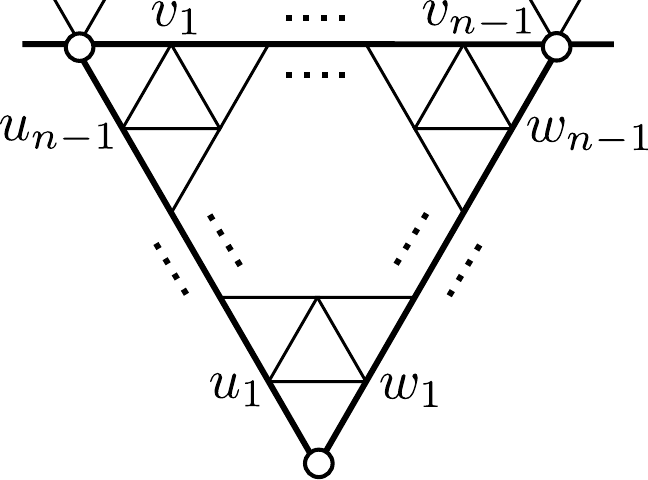}
    \caption{The vertices $v_i$ are on the attached edge.}\label{Fig;coord_uvw}
\end{figure}

Suppose $\overline{\Sigma}$ has boundary components $\partial_1,\cdots,\partial_b$, and $\partial_i$ contains $r_i$ punctures for each $1\leq i\leq b$.
For each $1\leq i\leq b$, we label the boundary edges contained in $\partial_i$

Let $W$ (resp. $U$) be the ordered set of the small vertices on $V_\lambda\setminus\obVlast$ (resp.$V'_\lambda\setminus\obVlast$) 
with the order defined as follows. 
On the boundary component $\partial_i$ of $\overline{\Sigma}$, we label the connected components of $\partial_i\cap \Sigma$ by $e_1,e_2, \dots,e_{r_i}$ consecutively following the positive orientation of $\partial \overline{\Sigma}$. 
For each $e_i$, there is a unique attached triangle $\tau_j$ in $\lambda^\ast$ containing $e_i$. 
Each $w\in V_\lambda\setminus\obVlast$ is determined by $(\partial_i,\tau_j,w_k)$, where $w$ is the small vertex with  the coordinate $w_k$ in the attached triangle $\tau_j$, which is attached to $\partial_i$. 
We identify $(\partial_i,\tau_j,w_k)$ with $(b-i,r_i-j,n-k)$.
Consider the lexicographic order on $W$ with respect to $(b-i,r_i-j,n-k)$ and suppose that $W$ is equipped with the order. 
In the same manner, we define an order on $U$ and suppose $U$ is also equipped with the order.

\subsection{Block decomposition of matrices}

The bilinear form $\barH_\lambda\colon \barV_\lambda\times\barV_\lambda\to\bZ$ is defined as
\begin{itemize}
\item for $v, v'\in \barV_\lambda$ not on the same boundary edge, $\barH_\lambda(v,v')=-\frac{1}{2}\barQ_\lambda(v,v')\in\bZ$,
\item for $v,v'\in \barV_\lambda$ on the same boundary edge, 
$$\barH_\lambda(v,v') = \begin{cases} 1 \qquad &\text{if $v=v'$},\\
-1 &\text{if there is an arrow from $v$ to $v'$}, \\
0 &\text{otherwise}.
\end{cases}$$
\end{itemize}
Define $\sfH_\lambda$ as the restriction of $\barH_{\lambda^\ast}$ to $V_\lambda\times V'_\lambda$. 
\begin{lem}[{\cite[A part of Lemma 11.9]{LY23}}]\label{lem:invertible_KH}
The following matrix identities holds.
\begin{enumerate}
    \item $n(\mathsf{K}_\lambda -\mathsf{K}_\lambda^T) =\mathsf{P}_\lambda$.
    \item $\barK_\lambda\barH_\lambda=nI$ and $\sfK_\lambda\sfH_\lambda=nI$. 
\end{enumerate} 
\end{lem}

\begin{prop}[{\cite[Propositon 11.10]{LY23}}]\label{prop:LY23_11.10}
Let $\bk$ be a vector in $\bZ^{\barV_\lambda}$. Then the following are equivalent.
\begin{enumerate}
\item $\bk$ is balanced.
\item $\bk\barH_\lambda\in(n\bZ)^{\barV_\lambda}$.
\item There exists a vector $\mathbf{c}\in\bZ^{\barV_\lambda}$ such that $\bk=\mathbf{c}\barK_\lambda$.
\end{enumerate}
The same results hold for the non-reduced case, i.e., when $\barV_\lambda, \barH_\lambda, \barK_\lambda$ are replaced with ${V}_\lambda, \sfH_\lambda, \sfK_\lambda$ respectively.
\end{prop}

When we regard $\sfC$ as a ($\obVlast, U)\times (\obVlast,W, U)$-matrix, 
we have the following form; 
$$\sfC=\begin{pmatrix}
I&C_1&O\\
O&-I&I
\end{pmatrix}.$$

\begin{lem}[{\cite[Lemma 5.3]{KW24}}]
\label{lemKQ}
    We have $\barK_{\lambda} \barQ_{\lambda} =
    \begin{pmatrix}
-2nI &  P' \\
O    &  P \\
\end{pmatrix}$,
where the rows and columns are divided in 
$(\obV_{\lambda}, \overline{V}_{\lambda}\setminus \obV_{\lambda})$.
\end{lem}

\begin{lem}\label{lem-matrixPP}
    Suppose that $a\in \obV_{\lambda}$ is contained in the ideal triangle $\tau'$ whose coordinate is $(i'j'k')$ and 
    $b_{i}\in \overline{V}_{\lambda}\setminus \obV_{\lambda}$ is contained in the ideal triangle $\tau$ whose coordinate is $(0,i,n-i)$; see Figure \ref{Fig:triangles}. 
    Let $S=\{s\in\{1,2,3\}\mid v_s'=v_2\}$.  
    Then 
    $$P'(a,b_i)=n\sum_{s\in S}\delta_{{\bf pr}_{s}(i'j'k'),i}.$$
\end{lem}
\begin{proof}
    We have 
    $$P'(a,b_{i})=\sum_{u\in \barV_\lambda}
    \overline{\mathsf{K}}_\lambda(a,u)
    \overline{\mathsf{Q}}_\lambda(u,b_{i}).$$

    When $2\leq i\leq n-2$, we have 
    \begin{align}\label{eq-P-K-ab}
        P'(a,b_{i})= K_\lambda(a,b_{i-1})- K_\lambda(a,b_{i+1})+2 K_\lambda(a,c_{i+1})- 2K_\lambda(a,c_{i}).
    \end{align}
\cite[Lemmas 5.5 and 5.6]{KW24} imply that 
\begin{align}
\label{eq-K-f}
    K_{\tau}(f_j,b_{i-1})- K_{\tau}(f_j,b_{i+1})+2 K_{\tau}(f_j,c_{i+1})- 2K_{\tau}(f_j,c_{i})&=0\\
    \label{eq-K-g}
     K_{\tau}(g_j,b_{i-1})-  K_{\tau}(g_j,b_{i+1})+2 K_{\tau}(g_j,c_{i+1})- 2K_{\tau}(g_j,c_{i})&=n\delta_{j,i}\\
     \label{eq-K-a}
      K_{\tau}(a',b_{i-1})-  K_{\tau}(a',b_{i+1})+2  K_{\tau'}(a',c_{i+1})- 2 K_{\tau'}(a',c_{i})&=n\delta_{j',i},
\end{align}
where $a'=(i'j'k')$ in $\tau$.

When $\tau\neq \tau'$, Equation  \eqref{eq-P-K-ab} is the sum of Equations \eqref{eq-K-f} and 
\eqref{eq-K-g} depending on the position of the vertices of $\tau'$. 
Then equations \eqref{eq-P-K-ab}, \eqref{eq-K-f}, and \eqref{eq-K-g} imply that 
\begin{align*}
     P'(a,b_{i})
     =&\sum_{s\in S}
     K_{\tau}(g_{{\bf pr}_{s}(a)},b_{i-1})- K_{\tau}(g_{{\bf pr}_{s}(a)},b_{i+1})+2 K_{\tau}(g_{{\bf pr}_{s}(a)},c_{i+1})- 2K_{\tau}(g_{{\bf pr}_{s}(a)},c_{i})\\
     =&n\sum_{s\in S}\delta_{{\bf pr}_{s}(i'j'k'),i}.
\end{align*}

When $\tau=\tau'$, Equation  \eqref{eq-P-K-ab} is the sum of Equations \eqref{eq-K-f}, 
\eqref{eq-K-g}, and \eqref{eq-K-a}.
Note that $a'$ in \eqref{eq-K-a} equals $a$.
Then Equations \eqref{eq-P-K-ab}, \eqref{eq-K-f}, \eqref{eq-K-g}, and \eqref{eq-K-a} imply that 
\begin{align*}
     P'(a,b_{i'})
     =&\sum_{s\in S\cap\{1,2\}}
     K_{\tau}(g_{{\bf pr}_{s}(a)},b_{i-1})- K_{\tau}(g_{{\bf pr}_{s}(a)},b_{i+1})+2 K_{\tau}(g_{{\bf pr}_{s}(a)},c_{i+1})- 2K_{\tau}(g_{{\bf pr}_{s}(a)},c_{i})\\
     &+ K_{\tau}(a',b_{i-1})-  K_{\tau}(a',b_{i+1})+2  K_{\tau}(a',c_{i+1})- 2 K_{\tau}(a',c_{i})\\
     =&n\sum_{s\in S}\delta_{{\bf pr}_{s}(i'j'k'),i}.
\end{align*} 
\end{proof}

\begin{figure}[h]
    \centering
    \includegraphics[width=320pt]{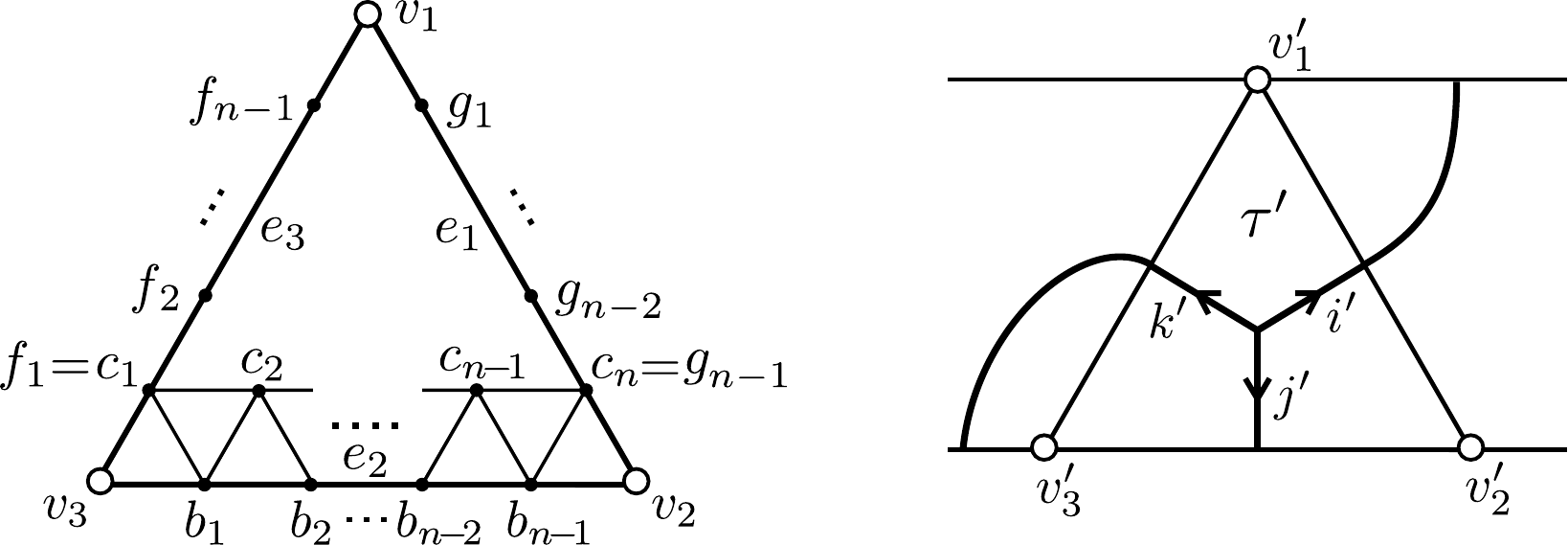}
    \caption{Left: A labeling of small vertices in $\tau$, Right: A skeleton of $a$.}\label{Fig:triangles}
\end{figure}

From Lemma \ref{lemKQ}, $\barKl\barQl$ has the following form; 
\begin{align}
\barKl\barQl=
\begin{pmatrix}
-2nI & D & \ast \\
O    & A & \ast \\
O    & B & \ast \\
\end{pmatrix},   \label{eq:matrix_barKQast}
\end{align}
where the rows and columns are divided in $(\obVlast, W, U)$. 

We have the $(\obVlast,U)\times(\obVlast,W, U)$-matrix
\begin{equation}\label{CKQ1}
\sfC\barKl\barQl=
\begin{pmatrix}
-2nI & D+C_1A & \ast\\
O    & B-A & \ast
\end{pmatrix}.
\end{equation}

Suppose 
\begin{align}
\barKl = \begin{pmatrix}
K_{11} & K_{12} & K_{13} \\
K_{21}    & K_{22} & K_{23} \\
K_{31}    & K_{32} & K_{33} \\
\end{pmatrix},\label{eq:matrix_K_docomp3}    
\end{align}
where the rows and columns are divided in $(\obVlast, W, U)$.

With the same procedure in Section 5.2 of \cite{KW24}, 
we have
\begin{equation}\label{eq_K}
\sfK_{\lambda}=(\sfC\barKl)|_{V_\lambda'\times V_\lambda} = \begin{pmatrix}
K_{11}+C_1 K_{21} & K_{12}+C_1 K_{22} \\
K_{31} - K_{21}    & K_{32} - K_{22}
\end{pmatrix} 
\end{equation}
and 
\begin{equation}\label{KQ}
\sfK_{\lambda}\sfQ_{\lambda} = 
\begin{pmatrix}
-2nI & D+C_1 A\\
   O & B-A
\end{pmatrix}.
\end{equation}

Recall some lemmas in \cite{KW24}, which will be used later.

\begin{lem}[{\cite[Lemma 5.5]{KW24}}]\label{matrixA}
    We have $A = -n I$. 
\end{lem}

\begin{lem}[{\cite[Lemma 5.6]{KW24}}]\label{matrixB}
    Suppose $\overline{\Sigma}$ has boundary components $\partial_1,\cdots,\partial_b$, and $\partial_i$ contains $r_i$ punctures for each $1\leq i\leq b$.
    We have $B = diag\{B_{1},B_2,\cdots,B_b\}$, where $B_i$ is the matrix associated to $\partial_i$ for $1\leq i\leq b$. Furthermore, for each $1\leq i\leq b$, we have 
    \begin{equation}\label{Bi}
    B_i = \begin{pmatrix}
     O & O &  \cdots &O & nI \\
     nI & O & \cdots &O & O \\
     O  & nI& \cdots & O & O\\
     \vdots & \vdots &  & \vdots & \vdots \\
     O & O & \cdots  & nI &  O \\
    \end{pmatrix} \text{\ if $r_i>1$, and\ } 
    B_i = nI \text{\ if $r_i=1$}, 
    \end{equation}
    where $B_i$ is of size $r_i(n-1)$ and  every block matrix is of size $n-1$.
\end{lem}

\begin{lem}[{\cite[Lemma 5.9]{KW24}}]\label{matrixK}
    Suppose $\overline{\Sigma}$ has boundary components $\partial_1,\cdots,\partial_b$, and $\partial_i$ contains $r_i$ punctures for each $1\leq i\leq b$.
    We have $K_{32}-K_{22} = diag\{L_{1},L_2,\cdots,L_b\}$, where $L_i$ is the matrix associated to $\partial_i$ for each $1\leq i\leq b$.
    Furthermore, for each $1\leq i\leq b$, we have 
\begin{equation}\label{eq_L}
L_i = \begin{pmatrix}
-G & O & O & \cdots &O& G \\
G & -G & O & \cdots &O& O \\
\vdots & \vdots & \vdots & \; &\vdots& \vdots \\
O & O & O & \cdots &-G& O \\
O & O & O & \cdots &G& -G \\
\end{pmatrix}\text{\ if $r_i>1$, and }
L_i=O\text{ if $r_i=1$},
\end{equation}
where $G$ is the square matrix of size $(n-1)$ with
\begin{align}\label{eq-matrix-G-def}
    G_{ij} = \begin{cases}i(n-j) & i\leq j,\\
j(n-i) & i>j.
\end{cases}
\end{align}
\end{lem}

Let $E$ and $F$ be square matrices of size $(n-1)$ defined by 
\begin{align}
E_{ij}=\begin{cases}
i-j+1 & \text{if $i\geq j$}\\
0 & \text{if $i<j$}
\end{cases},\qquad 
F_{ij}=\begin{cases}
n-j & \text{if $i=1$}\\
-n &\text{if $i=j+1$}\\
0 & \text{otherwise}
\end{cases}.\label{matrixEF}
\end{align}

\begin{lem}[{\cite[Lemma 5.10]{KW24}}]\label{matrixG}
    We have $EF=G$.
\end{lem}

\subsection{The center of $\A$ when $\hat q$ is a root of unity}\label{sub_center}
In the rest of this section, we suppose Condition~$(\ast)$ in Section~\ref{notation}.

For $\Sigma^\ast$ and $u_i=(i,0,n-i)$ in the attached triangle with the attaching edge $e$, we use $\gaa_{u_i}^e$ to denote $\gaa_{u_i}$ defined in Section~\ref{sec:quantum_trace}.

Note that the following lemma holds for all roots of unity.
\begin{lem}[{\cite[Lemma 5.15]{KW24}}]\label{boundary_center}
Suppose $\overline{\Sigma}$ has a boundary component $\partial$ such that  the number of the connected components of $\partial\cap \Sigma$ is even and the connected components are labeled as $e_1,e_2,\cdots,e_r$ with respect to the orientation of $\partial$.
For every $0\leq k\leq m'$ and $1\leq i\leq n-1$, the element
\begin{align}
(\gaa_{u_{i}}^{e_1})^k (\gaa_{u_{i}}^{e_2})^{m'-k}\cdots (\gaa_{u_{i}}^{e_{r-1}})^k (\gaa_{u_{i}}^{e_r})^{m'-k}\label{central}
\end{align}
is central in $\cS_n(\Sigma)$.
\end{lem}

The following lemma will be used to define subgroups in \eqref{def-X-m-prime}--\eqref{Xsharp}, which will be used to formulate the center of $\A$ in Theorem~\ref{center_torus}.

\begin{lem}\label{lem-k2-zero}
The vector
$(\frac{1}{n}D-C_1) \bk_2^T={\bf 0}$ in $\mathbb Z_2$ if and only if we have the following.\\
(1) When $n$ is odd, we have ${\bf k}_2={\bf 0}$ in $\mathbb Z_2$.\\ 
(2) When $n$ is even, for any attached triangle $\tau$ and the small vertices $w_1,\dots,w_{n-1}\in \tau\cap W$, 
we have $$({\bf k}_2(w_1),\cdots,{\bf k}_2(w_{n-1}))=
(1,0,1,0,\cdots,0,1)\text{ or }{\bf 0}\text{ in }\mathbb Z_2.$$
In particular, we have $$({\bf k}_2(w_1),\cdots,{\bf k}_2(w_{n-1}))
=({\bf k}_2(w_1'),\cdots,{\bf k}_2(w_{n-1}'))\text{ in }\mathbb Z_2$$
for any two different attached triangles $\tau$ and $\tau'$.
\end{lem}
\begin{proof}

Note that $D$ is a part of $P'$ when we regard $\lambda^{*}$ as $\lambda$. Hence, Lemma~\ref{lem-matrixPP} is applicable. 

\begin{figure}[h]  
\centering\includegraphics[width=5.5cm]{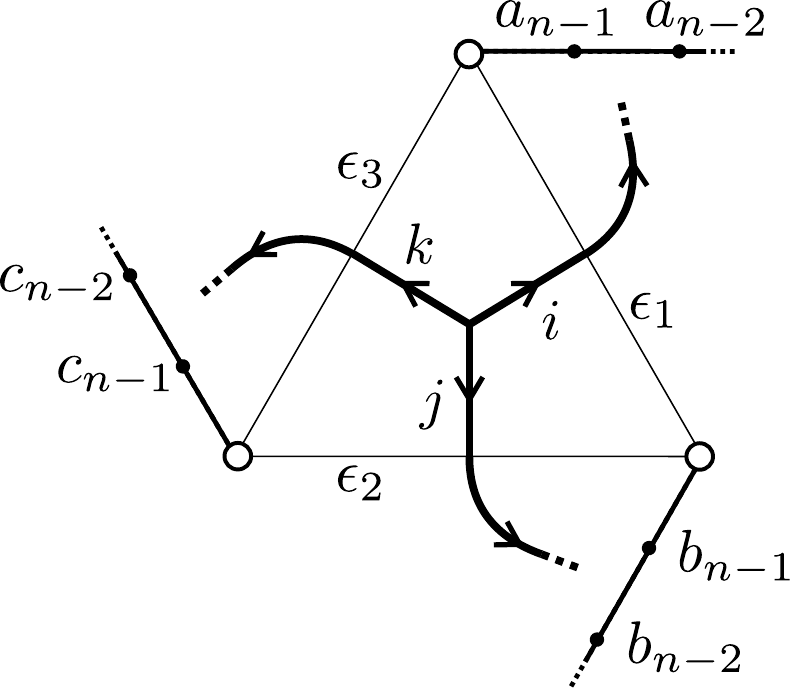}
	\caption{A labeling of small vertices on boundary edges}
    \label{fig-labeling-a-b-c}
\end{figure}

Suppose that $v=(ijk)\in \mathring{V}_\lambda$ is a small vertex contained in a triangle $\tau''$. 
We use the labelings in Figure \ref{Fig;coord_ijk} for the vertices and edges of $\tau''$ and we replace each labeling $e_i$ with $\epsilon_i$, see Figure \ref{fig-labeling-a-b-c} (the middle triangle is $\tau''$).
When $\tau''$ is not an attached triangle, we have $(C_1{\bf k}_2)(v) = 0$ from the definition of $\mathsf{C}$ in Section~\ref{sec;A_tori}.
Lemma \ref{lem-matrixPP} implies that
\begin{align}\label{eq-abc-vertex-D}
    (\frac{1}{n}D{\bf k}_2^T)(v)
    =\sum_{u}\frac{1}{n}D(v,u){\bf k}_2(u)={\bf k}_2(a_i) + {\bf k}_2(b_j) + {\bf k}_2(c_k),
\end{align}
where $a_i,b_j,c_k$ are labeled as in Figure \ref{fig-labeling-a-b-c}. 
If we have (1) or (2), Equation \eqref{eq-abc-vertex-D} implies that $\frac{1}{n}D{\bf k}_2^T={\bf 0}$ in $\mathbb Z_2$ from $i+j+k=n$. The opposite is obvious.

\begin{figure}[h]
    \centering
    \includegraphics[width=160pt]{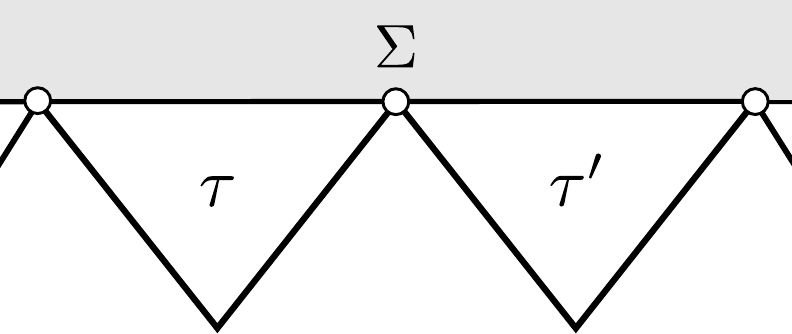}
    \caption{A local picture of $\Sigma^\ast$}\label{Fig;tau_tau'}
\end{figure}

Consider the case when $\tau''$ is an attached triangle. 
In this proof, we use $\tau$ and $\tau''=\tau'$ depicted in Figure~\ref{Fig;tau_tau'} and $w_i=(0,i,n-i)$ and $w_i'=(0,i,n-i)$ $(i=1,\dots ,n-1)$ to denote the small vertices in $\tau$ and $\tau'$ respectively.

Suppose that $v=(ijk)$ is a vertex in the attached triangle $\tau'$ (note that $j>0$). Then Lemma \ref{lem-matrixPP} implies that
$$\frac{1}{n}\sum_{w\in W}D(v,w){\bf k}_2(w)=
    \begin{cases}
        {\bf k}_2(w_i) + {\bf k}_2(w_j') & i>0,\\
        {\bf k}_2(w_j') & i=0.
    \end{cases}$$
Also, we have $$\sum_{w\in W}C_1(v,w){\bf k}_2(w)=\begin{cases}
        -{\bf k}_2(w_{n-k}') & k>0,\\
        0 & k=0.
    \end{cases}
    $$
Set ${\bf t}= (\frac{1}{n}D-C_1) \bk_2^T$.
Then ${\bf t}=\mathbf{0}$ in $\mathbb Z_2$ 
implies that 
\begin{align}\label{eq-k2-zero-key}
{\bf t}(v)=
\begin{cases}
        {\bf k}_2(w_i) + {\bf k}_2(w_j') + {\bf k}_2(w_{n-k}')=0\text{ in }\mathbb Z_2 & k>0,\\
        {\bf k}_2(w_i) + {\bf k}_2(w_j')=0\text{ in }\mathbb Z_2 & k=0.
    \end{cases}
    \end{align}
Equation~\eqref{eq-k2-zero-key} implies that 
${\bf k}_2(w_{n-i}') + {\bf k}_2(w_j') + {\bf k}_2(w_{i+j}') = 0\text{ in }\mathbb Z_2$ when $i+j< n$.
When $n=2i+j$ (i.e., $n-i = i+j$) and $i>0$, we have 
\begin{align}\label{w-j-zero}
    {\bf k}_2(w_{n-i}') + {\bf k}_2(w_j') + {\bf k}_2(w_{i+j}')={\bf k}_2(w_j') + 2{\bf k}_2(w_{i+j}')  ={\bf k}_2(w_j')= 0\text{ in }\mathbb Z_2.
\end{align}

When $n$ is odd, Equation \eqref{w-j-zero} implies that ${\bf k}_2 (w_{2t-1}') =0\text{ in }\mathbb Z_2$ for 
$1\leq t\leq \frac{n-1}{2}$. 
Then the second equation in \eqref{eq-k2-zero-key} implies that ${\bf k}_2 (w_{2t}) =0\text{ in }\mathbb Z_2$ for 
$1\leq t\leq \frac{n-1}{2}$. From the generality of $\tau$ and $\tau'$, we have 
${\bf k}_2 (w_{t}') = 0\text{ in }\mathbb Z_2$ for 
$1\leq t\leq n-1$.  This shows that 
${\bf k}_2 = {\bf 0}\text{ in }\mathbb Z_2$.

When $n$ is even, Equation \eqref{w-j-zero} implies that ${\bf k}_2 (w_{2t}') =0\text{ in }\mathbb Z_2$ for 
$1\leq t\leq \frac{n-2}{2}$. 
For any two odd integers $1\leq i,j\leq n-1$, the first equation in \eqref{eq-k2-zero-key} 
implies that ${\bf k}_2(w_i) + {\bf k}_2(w_j')=0\text{ in }\mathbb Z_2$. This 
implies that 
$$({\bf k}_2(w_1),\cdots,{\bf k}_2(w_{n-1}))
=({\bf k}_2(w_1'),\cdots,{\bf k}_2(w_{n-1}'))=
(1,0,1,0,\cdots,0,1)\text{ or }{\bf 0}\text{ in }\mathbb Z_2.$$
It follows from \eqref{eq-k2-zero-key} that the opposite is obvious. 
\end{proof}

From now we try to formulate the center of $\mathcal A_{\hat q}(\Si,\lambda)$ under the root of unity condition in Section \ref{notation}.
This is equivalent to compute the center of $\mathcal X^{\rm bl}_{\hat q}(\Si,\lambda)$ because of Theorem \ref{thm-transition-LY}.
It follows from Lemma \ref{quantum} and Proposition \ref{prop:LY23_11.10} that it suffices to check the equation 
$$\text{$\mathsf{K}_\lambda \mathsf{Q}_\lambda{\bf k}={\bf 0}\text{ in }\mathbb Z_{m''}$ for ${\bf k}\in\Lambda_\lambda.$}$$
\begin{rem}\label{rem-center-equation}
Let ${\bf k}\in\Lambda_\lambda$ such that $\mathsf{K}_\lambda \mathsf{Q}_\lambda{\bf k}={\bf 0}\text{ in }\mathbb Z_{m''}$.
We regard $\mathbf{k}_0 = (
        \tfk_1,\tfk_2
    )\in \bZ^{V_{\lambda}}$ with $\tfk_1\in \bZ^{\obVlast}$ and $\tfk_2\in \bZ^{W}$.  From Equation \eqref{KQ}, we have 
    \begin{equation}\label{eq-original-key}
    \begin{cases}
    -2n\bk_1^T+ (D+C_1A) \bk_2^T=\bm{0},\\
    (B-A)\bk_2^T=\bm{0},
    \end{cases}\text{ in $\mathbb Z_{m''}$}.
    \end{equation}
From Lemmas~\ref{lem-matrixPP} and \ref{matrixA}, the above equation is equivalent to the following one
\begin{equation}\label{eq_key}
\begin{cases}
-2\bk_1^T+ (\frac{1}{n}D-C_1) \bk_2^T=\bm{0},\\
\frac{1}{n}(B-A)\bk_2^T=\bm{0},
\end{cases}\text{ in $\mathbb Z_{m'}$}.
\end{equation}
\end{rem}

We label the components of $\partial\overline{\Si}$ by $\partial_1,\cdots,\partial_b$.
Suppose that $\partial_i$ contains $r_i$ boundary components of $\Si$.
Suppose $\tfk_2 =  (
\tfk_2|_{\partial_1},\cdots,\tfk_2|_{\partial_b})$ where $\tfk_2|_{\partial_i}\in\mathbb Z^{r_i(n-1)}$ is the row vector associated to $\partial_i$ for each $1\leq i\leq b$.
Suppose  $\tfk_2|_{\partial_i}= (
\tfb_{i1},\tfb_{i2},\cdots,\tfb_{ir_i})$, where $\tfb_{ij}\in \bZ^{n-1}$ for any $j$, and the order of $\tfb_{i1},\tfb_{i2},\cdots,\tfb_{ir_i}$ is compatible with the positive orientation of $\Sigma$. 
Consider the following conditions: 
\begin{enumerate}
\item[(X1)] $\bk_1^T= (\frac{1}{n}D-C_1) \frac{1}{2}\bk_2^T \text{ in }\mathbb Z_{m^\ast}$,
\item[(X2)] ${\bf b}_{ij} = (-1)^{j-1} {\bf b}_{i1}\text{ in }\mathbb Z_{m'}$\text{ for any $1\leq j\leq r_i$},
\item[(X3)] for any attached triangle $\tau,$ $$({\bf k}_2(w_1),\cdots,{\bf k}_2(w_{n-1}))=(1,0,1,0,\cdots,0,1)\text{ or }{\bf 0}\text{ in }\mathbb Z_2,$$
\item[(X4)] for any attached triangles $\tau,\tau'$, $$({\bf k}_2(w_1),\cdots,{\bf k}_2(w_{n-1}))=({\bf k}_2(w_1'),\cdots,{\bf k}_2(w_{n-1}'))\text{ in }\mathbb Z_2,$$
\end{enumerate}
where $w_i\in \tau$ are labeled as in Figure \ref{Fig;coord_uvw} and so are $w_i'\in\tau'$ in the same manner.
Lemma~\ref{lem-k2-zero} and (X3) ensure that 
${\bf k}_1$ is an integral vector.
 
Define 
\begin{align}
    X_{m'}&= \{\mathbf{k}=({\bf k}_1,{\bf k}_2)\in\Lambda_\lambda \mid {\bf k}_1 =\bm{0} \text{ in }\mathbb Z_{m^\ast}\text{ and }{\bf k}_2 =\bm{0} \text{ in }\mathbb Z_{m'}\},\label{def-X-m-prime}\\
    {X}_{m'}^\ast&= \{\mathbf{k}=({\bf k}_1,{\bf k}_2)\in\Lambda_\lambda \mid \mathbf{k}\text{ satisfies (X1)--(X4) and }{\bf k}_2 ={\bf 0} \text{ in }\bZ_{m^\ast}\},\\
    \overline{X}_{m'}^{\ast}&= \{\mathbf{k}=({\bf k}_1,{\bf k}_2)\in\Lambda_\lambda \mid \mathbf{k}\text{ satisfies (X1)--(X4) and }{\bf k}_2|_{\partial_i} =
    {\bf 0} \text{ in }\bZ_{m^\ast}\text{ if $r_i$ is odd}\}, \\
    \overline X_{m'}&= \{\mathbf{k}=({\bf k}_1,{\bf k}_2)\in \overline{X}_{m'}^{\ast} \mid {\bf k}_2 = {\bf 0} \text{ in }\mathbb Z_2\},\\
    X_{m'}^{\sharp}&=
    \begin{cases}
        \{{\bf k}\in X_{m'}^{\ast}\mid
    {\bf k}|_{\partial_i}={\bf 0}\text{ in $\mathbb Z_{m'}$ when $r_i$ is even}\} & \text{$n$ is odd,}\\
    X_{m'}^{\ast} & \text{$n$ is even,}
    \end{cases} \label{Xsharp}
\end{align}
where ${\bf k}|_{\partial_i}$ is the restriction of ${\bf k}$ to the boundary component $\partial_i$ of $\overline \Si$.
Set 
\begin{align}\label{def-tilde-m}
    \text{$m=2^k \bar m$\quad and\quad }\tilde{m}=d^{\ast}\bar m,
\end{align}
where $\bar m$ is odd and $k$ is a nonnegative integer.

Define
\begin{align}
\overline{X}_{m'}^{\sharp}&=
    \begin{cases}
       X_{m'}^{\sharp} & \text{$n$ is odd,}\\
    \{{\bf k}\in\overline{X}_{m'}^{\ast}\mid {\bf k}|_{\partial_i}={\bf 0}\text{ in }\mathbb Z_{\tilde{m}} \text{ if $r_i$ is even}\}& \text{$n$ is even.}
    \end{cases}\label{Xsharpbar}
\end{align}

Define \begin{equation}
\begin{split}
    \Gamma_{m'}= \begin{cases}
        \{\mathbf{k}\in\mathbb Z^{V_{\lambda}'} \mid 
\mathbf{k}\sfK_{\lambda} \in X_{m'}\} & m^{\ast}\text{ is odd and } n\text{ is odd,}\\
\{\mathbf{k}\in\mathbb Z^{V_{\lambda}'} \mid 
\mathbf{k}\sfK_{\lambda} \in X_{m'}^\ast\}
& m^{\ast}\text{ is odd and } n\text{ is even,}\\
        \{\mathbf{k}\in\mathbb Z^{V_{\lambda}'} \mid 
\mathbf{k}\sfK_{\lambda} \in\overline X_{m'}\} & m^{\ast}\text{ is even and } n\text{ is odd,}\\
\{\mathbf{k}\in\mathbb Z^{V_{\lambda}'} \mid 
\mathbf{k}\sfK_{\lambda} \in  \overline{X}_{m'}^{\ast}\}
& m^{\ast}\text{ is even and } n\text{ is even.}
    \end{cases}
    \label{eq:Lambda}
    \end{split}
\end{equation}

Let $\mathsf{B}$ denote the set of the central elements in Lemma \ref{boundary_center}. 
Note that ${X}_{m'}^\ast$ (resp. $\overline{X}_{m'}^\ast$) is bigger than $X_{m'}$ (resp. $\overline X_{m'}$). 
Note that, in the next theorem, $\tra(\mathsf{B})$ is contained in $\{a^{\mathbf{k}}\mid \mathbf{k} \in \Gamma_{m'} \}$ when $m^\ast$ is even.

\begin{thm}\label{center_torus}
Let $\Sigma$ be a triangulable essentially
bordered pb surface without interior punctures, and $\lambda$ be a triangulation of $\Sigma$.
Assume that $m'$ is even, $\hat{q}^2$ is a primitive $m''$-th root of unity, where $m'=m''/\gcd(m'',n)$ and $m^\ast=m'/2$.
Then, the center of $\A$ is generated by 
$$\begin{cases}
    \text{$\tra(\mathsf{B})$ and $\{a^{\mathbf{k}}\mid \mathbf{k} \in \Gamma_{m'} \}$} & m^{\ast} \text{ is odd,}\\
    \{a^{\mathbf{k}}\mid \mathbf{k} \in \Gamma_{m'} \} & m^{\ast} \text{ is even.}
\end{cases}$$
\end{thm}
Because of a lengthy proof, we will give a proof of Theorem \ref{center_torus} in Section \ref{sec;pf}.

Suppose $\overline{\Sigma}$ has boundary components $\partial_1,\cdots,\partial_b$, and each $\partial_i$ contains $r_i$ punctures.
We will define a subgroup $\Lambda_{\partial}\subset \bZ^{V_{\lambda}'}$.
Define $\Lp=\emptyset$ if all $r_i$ ($1\leq i\leq b$) are odd.

Suppose at least one of $r_i$ is even. 
For each $i$ such that $r_i$ is even, we label the boundary components of $\Sigma$ contained in $\partial_i$ by $e_1,\cdots,e_{r_i}$ consecutively from
$1$ to $r_i$ following the positive orientation of $\partial_i$. For $1\leq j\leq n-1$, define $\mathbf{k}_{j,\partial_i}\in \mathbb Z^{V_{\lambda}'}$ such that 
$$\mathbf{k}_{j,\partial_i}(v) =\begin{cases}
        (-1)^{k-1}& \text{if } v= u_j^{e_k}\text{ for }1\leq k\leq r_i, \\
        0 &  \text{ otherwise},\\
        \end{cases}$$
where $u_j^{e_k}$ denotes the small vertex $u_j$ in the attached triangle with $e_k\subset \partial_i$ as the attached edge.
Then, define $\Lambda_{\partial}$ as the subgroup of $\bZ^{V_{\lambda}'}$ generated by
\begin{align}\label{eq-balanced-boundary-A-version}
     \{\mathbf{k}_{j,\partial_i}\mid r_i \text{ is even},\; 1\leq j\leq n-1\},
\end{align}
and
\begin{align}\label{eq-definition-boundary-X}
    \Lambda_\partial^X:= \Lambda_{\partial}\mathsf{K}_\lambda.
\end{align}
 
Theorem~\ref{traceA} (a) implies that the subalgebra of $\A$ generated by $\tra(\mathsf{B})$ is $$\mathbb C\text{-span}
\{a^{\bf k}\mid {\bf k}\in\Lambda_\partial\}.$$

In the rest of this subsection, we state Lemmas~\ref{lem-overlineX}--\ref{lem:X/X}, which will be used in Section~\ref{sec-Unicity-Theorem} to compute the rank of $\A$ over its center.

\begin{lem}\label{lem-overlineX}
Suppose $m^\ast$ is even.  
    We have 
    $\overline X_{m'} $ is the subgroup of $\Lambda_\lambda$ generated by
$\overline{X}_{m'}^{\sharp}$
    and $\Lambda_\partial^X$.
\end{lem}

Before we prove Lemma \ref{lem-overlineX}, let us recall the following lemma. 
\begin{lem}[{\cite[Lemma 5.20]{KW24}}]\label{balance}
    Suppose $\mathbf{k}\in\mathbb Z^{V_{\lambda}}$ is balanced, $e$ is a boundary component of $\Sigma$, and $w_1,\cdots,w_{n-1}$ are illustrated in Figure \ref{Fig;coord_uvw} (here $e$ is the attached edge). Then there exists $l\in\mathbb Z$ such that $(\mathbf{k}(w_1),\mathbf{k}(w_2),\cdots,\mathbf{k}(w_{n-1})) - l(1,2,\cdots,n-1) = \bm{0}$ in $\mathbb Z_n$.
\end{lem}

Note that, while there is ambiguity on the choice of $l$ in $\bZ$, $l$ is uniquely determined in $\bZ_n$.

\begin{proof}[Proof of Lemma \ref{lem-overlineX}]
Suppose $\overline{\Sigma}$ has boundary components $\partial_1,\cdots,\partial_b$, and each boundary component $\partial _i$ contains $r_i$ punctures.

It is obvious that 
$ {X}_{m'}^\sharp \subset \overline X_{m'}$. Hence, 
We will show that $\Lambda_\partial^X\subset \overline{X}_{m'}$.
Let ${\bf k}=({\bf k}_1,{\bf k}_2)\in \Lambda_\partial^X,$ with $\tfk_1\in \bZ^{\obVlast}$ and $\tfk_2\in \bZ^{W}$. 
Then ${\bf k} = ({\bf 0}, {\bf a}) \mathsf{K}_\lambda$, where ${\bf a}\in\mathbb Z^U$ with
$${\bf a}|_{\partial_i} = 
\begin{cases}
    {\bf 0} & r_i\text{ is odd},\\
    (-{\bf a}_i,{\bf a}_i,\cdots, -{\bf a}_i,{\bf a}_i)\text{ for ${\bf a}_i\in\mathbb Z^{n-1}$} & r_i\text{ is even}.
\end{cases}$$
 Lemma \ref{matrixK} implies that $${\bf k}|_{\partial_i}=(\mathbf{a}|_{\partial_i})L_i = (2\mathbf{a}_{i}G,-2\mathbf{a}_{i}G,\cdots,2\mathbf{a}_{i}G,-2\mathbf{a}_{i}G)={\bf 0}\text{ in }
     \mathbb Z_{2}$$ 
when $r_i$ is even and 
${\bf k}|_{\partial_i}={\bf 0}$
when $r_i$ is odd.
Lemma \ref{boundary_center} and 
Remark \ref{rem-center-equation} imply that $\bk_1^T= (\frac{1}{n}D-C_1) \frac{1}{2}\bk_2^T \text{ in }\mathbb Z_{m^\ast},$ which shows ${\bf k}\in \overline X_{m'}$.

Let ${\bf k} = ({\bf k_1},{\bf k}_2)\in \overline X_{m'}$ with $\tfk_1\in \bZ^{\obVlast}$ and $\tfk_2\in \bZ^{W}$.
Suppose $\tfk_2 =  (
\tfk_{\partial_1},\tfk_{\partial_2},\cdots,\tfk_{\partial_b}
)$ where $\tfk_{\partial_i}\in\mathbb Z^{r_i(n-1)}$ is the row vector associated to $\partial_i$ for each $1\leq i\leq b$.
Suppose  $\tfk_{\partial_i}= (
\tfb_{i1},\tfb_{i2},\cdots,\tfb_{ir_i})$, where $\tfb_{ij}\in \bZ^{n-1}$ for any $j$.

When $r_i$ is even, we have
$${\bf b}_{ij} = (-1)^{j-1} {\bf b}_{i1}\text{ in }\mathbb Z_{m'}.$$
Since $\mathbf{b}_{i1}$ is balanced, Lemma \ref{balance} implies that there exists $l_i\in\mathbb Z$ such that 
\begin{align}
\mathbf{b}_{i1} - l_i(1,2,\cdots,n-1) = \bm{0}\text{\quad in } \mathbb Z_n.\label{eq:bi111}
\end{align}

When $n$ is odd, then there exists $l_i'\in\mathbb Z$ such that 
\begin{align}
{\bf b}_{i1} = l_i(1,2,\cdots,n-1)=2l_i'(1,2,\cdots,n-1)\text{ in }\mathbb Z_n.\label{eq:bij_Zn2222}
\end{align}
Since $E$ is invertible, there exists ${\bf d}_i\in\mathbb Z^{n-1}$ such that ${\bf d}_i E =(-l_i',0,\cdots,0)$.
Then we have 
$$2{\bf d}_i G = 2{\bf d}_i EF= -2(l_i',0,\cdots,0) F
=2l_i'(1,2,\cdots,n-1)={\bf b}_{i1}\text{ in }\mathbb Z_n.$$
We also have $2{\bf d}_i G ={\bf b}_{i1}$ in $\mathbb Z_{2}.$
Then $2{\bf d}_i G ={\bf b}_{i1}$ in $\mathbb Z_{2n}.$
Note that $2n$ is a multiple of $d$, i.e., 
$2{\bf d}_i G ={\bf b}_{i1}$ in $\mathbb Z_{d}.$

When $n$ is even, we have that $l_i$ is even because of \eqref{eq:bi111} and ${\bf b}_{i1}={\bf 0}$ in $\mathbb Z_2$.
Similarly, there exists $l_i'\in\mathbb Z$ such that we have \eqref{eq:bij_Zn2222}.
As above, there exists ${\bf d}_i\in\mathbb Z^{n-1}$ such that 
$2{\bf d}_i G ={\bf b}_{i1}$ in $\mathbb Z_{n}$.  Then $2{\bf d}_i G ={\bf b}_{i1}$ in $\mathbb Z_{d^{\ast}}$ since $n$ is a multiple of $d^{\ast}$.

Define $\mathbf{d}= (\mathbf{d}_{\partial_1},\mathbf{d}_{\partial_2},\cdots,\mathbf{d}_{\partial_b})\in \mathbb Z^{U}$, 
where $\mathbf{d}_{\partial_i}\in\mathbb Z^{r_i(n-1)}$ is the vector associated to $\partial_i$,  such that 
\begin{equation}\label{zero}
    \mathbf{d}_{\partial_i} =\begin{cases}
        {\bf 0} & r_i\text{ is odd},\\
        (-\mathbf{d}_i,\mathbf{d}_i,\cdots,-\mathbf{d}_i,\mathbf{d}_i) & r_i \text{ is even}.
    \end{cases}
    \end{equation}
    Set $\mathbf{d}' := (\mathbf{0},\mathbf{d})\in\mathbb Z^{V_{\lambda}'}$ and $\mathbf{f} := \mathbf{d}'\sfK_{\lambda}\in\Lambda_\partial^X$.
    We regard $\mathbf{f} = (
        \mathbf{f}_1,\mathbf{f}_2
    )\in \mathbb Z^{V_{\lambda}}$ with $\mathbf{f}_1\in \bZ^{\obVlast}$ and $\mathbf{f}_2\in \bZ^{W}$. 
From \eqref{eq_K} and Lemma \ref{matrixK}, we have 
    $\mathbf{f}_2 = \mathbf{d}(K_{32}-K_{22}) = (\mathbf{d}_{\partial_1}L_1,\mathbf{d}_{\partial_2}L_2,\cdots,\mathbf{d}_{\partial_b}L_b)$.

     Lemma \ref{matrixK} implies $$\mathbf{d}_{\partial_i}L_i = (2\mathbf{d}_{i}G,-2\mathbf{d}_{i}G,\cdots,2\mathbf{d}_{i}G,-2\mathbf{d}_{i}G)=
     \begin{cases}
      {\bf k}_{\partial _i}   \text{ in }
     \mathbb Z_{d} & \text{$n$ is odd}\\
    {\bf k}_{\partial _i} \text{ in }
     \mathbb Z_{d^{\ast}} & \text{$n$ is even}
     \end{cases}.$$

Set $\mathbf{h} := \mathbf{f} - \mathbf{k}$. Then $\mathbf{h}  =
(\mathbf{h} _1,\mathbf{h} _2)$, where $\mathbf{h}_1 = \mathbf{f}_1 - \mathbf{k}_1$ and $\mathbf{h}_2 = \mathbf{f}_2 - \mathbf{k}_2$. Note that $\mathbf{h}_1$ and $\mathbf{h}_2$ satisfy Equation \eqref{eq_key}.
We regard $\mathbf{h}_2 =  (\tff_{\partial_1},\tff_{\partial_2},\cdots,\tff_{\partial_b})$ where $\tff_{\partial_i}\in\mathbb Z^{r_i(n-1)}$ is the row vector associated to $\partial_i$ for each $1\leq i\leq b$.
From the definition of ${\bf h}$, we know $\tff_{\partial_i} = \bm{0}$ in $\mathbb Z_{m^{\ast}}$ if $r_i$ is odd. Also, we have 
$\tff_{\partial_i} = (\tff'_{i},-\tff'_{i},\cdots,\tff'_{i},-\tff'_{i})$ in $\mathbb Z_{m'}$ if $r_i$ is even, where 
$\tff'_{i} =(h_{i,1},h_{i,1},\cdots,h_{i,n-1}) \in\mathbb Z^{n-1}$.

When $n$ is odd, we have
 $\tff_2 =\bm{0}$ in $\mathbb Z_{d}$, each $h_{i,j}$ is a multiple of $d$ for $1\leq j \leq n-1$. Consider the equation 
$(x_1,x_2,\cdots,x_{n-1})F= \frac{1}{2}\tff'_{i}$ in $\mathbb Z_{m^\ast}$  and we will find a solution. 
The equation implies
\begin{equation}\label{equation11}
\begin{cases}
(n-1)x_1 - nx_2 =\frac{1}{2}h_{i,1},\\
(n-2)x_1 - nx_3 =\frac{1}{2}h_{i,2},\\
\;\vdots\\
2x_1 - nx_{n-1} =\frac{1}{2}h_{i,n-2},\\
x_1 =\frac{1}{2} h_{i,n-1},\\
\end{cases}\text{ in $\mathbb Z_{m^\ast}$}.
\end{equation}
Equation \eqref{equation11} is equivalent to the following equations:
\begin{equation}\label{equation222} 
\begin{cases}
\dfrac{2n}{d} x_2 =(n-1)\dfrac{h_{i,n-1}}{d}-\dfrac{h_{i,1}}{d} \text{ in $\mathbb Z_{m}$},\medskip\\
\dfrac{2n}{d}x_3 =(n-2)\dfrac{h_{i,n-1}}{d}-\dfrac{h_{i,2}}{d} \text{ in $\mathbb Z_{m}$},\\
\;\vdots\\
\dfrac{2n}{d}x_{n-1} =2\dfrac{h_{i,n-1}}{d}-\dfrac{h_{i,n-2}}{d} \text{ in $\mathbb Z_{m}$},\\
x_1 =\frac{1}{2} h_{i,n-1} \text{ in $\mathbb Z_{m^\ast}$}.\\
\end{cases}
\end{equation}
Since $2n/d$ and $m$ are coprime, Equation \eqref{equation222} has a solution, and so does Equation \eqref{equation11}.
Fix a solution $\mathbf{x}_i'\in \mathbb Z^{n-1}$ of Equation \eqref{equation11}. Since $E$ is invertible, there exists 
$\mathbf{x}_i''\in \mathbb Z^{n-1}$ such that $\mathbf{x}_i ''E=\mathbf{x}_i'$ in $\mathbb Z_{m''}$. Thus we have 
$\mathbf{x}_i''G = \mathbf{x}_i''EF=\mathbf{x}_i'F=\frac{1}{2}\tff_i'$ in $\mathbb Z_{m^\ast}$. This implies 
$2\mathbf{x}_i''G = \tff_i'$ in $\mathbb Z_{m'}$.

When $n$ is even, we have ${\bf h}_2={\bf 0}$ in $\mathbb Z_{d^\ast}.$
Recall that $\bar m$ and $\tilde{m}$ are defined in \eqref{def-tilde-m}.
Consider the equation 
$(x_1,x_2,\cdots,x_{n-1})F= \frac{1}{2}\tff'_{i}$ in $\mathbb Z_{\tilde{m}/2}$ and we will find a solution. 
The equation implies
\begin{equation}\label{equation11}
\begin{cases}
(n-1)x_1 - nx_2 =\frac{1}{2}h_{i,1},\\
(n-2)x_1 - nx_3 =\frac{1}{2}h_{i,2},\\
\;\vdots\\
2x_1 - nx_{n-1} =\frac{1}{2}h_{i,n-2},\\
x_1 =\frac{1}{2} h_{i,n-1},\\
\end{cases}\text{ in $\mathbb Z_{\tilde{m}/2}$}.
\end{equation}
Equation \eqref{equation11} is equivalent to the following equations:
\begin{equation}\label{equation222} 
\begin{cases}
\dfrac{2n}{d^{\ast}} x_2 =(n-1)\dfrac{h_{i,n-1}}{d^\ast}-\dfrac{h_{i,1}}{d^\ast} \text{ in $\mathbb Z_{\bar m}$},\medskip\\
\dfrac{2n}{d^\ast}x_3 =(n-2)\dfrac{h_{i,n-1}}{d^\ast}-\dfrac{h_{i,2}}{d^\ast}
\text{ in $\mathbb Z_{\bar m}$},\\
\;\vdots\\
\dfrac{2n}{d^\ast}x_{n-1} =2\dfrac{h_{i,n-1}}{d^\ast}-\dfrac{h_{i,n-2}}{d^\ast}
\text{ in $\mathbb Z_{\bar m}$},\\
x_1 =\frac{1}{2} h_{i,n-1}
\text{ in $\mathbb Z_{\tilde{m}/2}$},\\
\end{cases}
\end{equation}
Since $2n/d^{\ast}$ and $\bar m$ are coprime, Equation \eqref{equation222} has a solution, and so does Equation \eqref{equation11}.
Fix a solution $\mathbf{x}_i'\in \mathbb Z^{n-1}$ of Equation \eqref{equation11}. Since $E$ is invertible, there exists 
$\mathbf{x}_i''\in \mathbb Z^{n-1}$ such that $\mathbf{x}_i ''E=\mathbf{x}_i'$. Thus we have 
$\mathbf{x}_i''G = \mathbf{x}_i''EF=\mathbf{x}_i'F=\frac{1}{2}\tff_i'$ in $\mathbb Z_{\tilde{m}/2}$. This implies 
$2\mathbf{x}_i''G = \tff_i'$ in $\mathbb Z_{\tilde{m}}$.

Set $\mathbf{x}_2= (\mathbf{x}_{\partial_1},\mathbf{x}_{\partial_2},\cdots,\mathbf{x}_{\partial_b})\in \mathbb Z^{U}$, where $\mathbf{x}_{\partial_i}\in\mathbb Z^{r_i(n-1)}$ is the vector associated to $\partial_i$ such that 
\begin{equation}
\mathbf{x}_{\partial_i} = 
\begin{cases}
\bm{0}& r_i \text{ is odd},\\
(-\mathbf{x}_i'',\mathbf{x}_i'',\cdots,-\mathbf{x}_i'',\mathbf{x}_i'') & r_i \text{ is even}.\\
\end{cases}
\end{equation}
Define $\mathbf{x}\in\mathbb Z^{V_{\lambda}'}$ as $\mathbf{x} = (\mathbf{0},\mathbf{x}_2)$, and set $\mathbf{y} := \mathbf{x}\sfK_{\lambda}\in\Lambda_\partial^X$.

From the above construction,  we have ${\bf y}-{\bf h}\in \overline{X}_{m'}^{\sharp}$.
Recall that $\mathbf{h} = \mathbf{f} - \mathbf{k}$ and  ${\bf y},{\bf f}\in\Lambda_\partial^X.$ Then
we have 
$${\bf k} = ({\bf y} - {\bf h}) - {\bf y} + {\bf f}.$$
\end{proof}

\begin{lem}\label{eq;bar_sharp}
    When both $m^\ast$ and $n$ are even, we have $$\left|\dfrac{\overline{X}_{m'}^{\sharp}}{X_{m'}^{\sharp}} \right|=
    \begin{cases}
        1 & \text{$m$ is odd},\\
        2^{(k-1)(n-1)t + t} & \text{$m$ is even},
    \end{cases}$$
with the decomposition $m=2^k\bar{m}$ in \eqref{def-tilde-m}.
\end{lem}
\begin{proof}
    When $m$ is odd, equivalently $k=0$, it is trivial that $\overline{X}_{m'}^{\sharp}={X}_{m'}^{\sharp}$.

    Suppose that $m$ is even.
    Define
    $$\widetilde{X}_{m'}^{\sharp}:=\{{\bf k}\in\overline{X}_{m'}^{\ast}\mid {\bf k}|_{\partial_i}={\bf 0}\text{ in }\mathbb Z_{2\tilde{m}} \text{ if $r_i$ is even}\},$$
    where $\tilde{m}$ is defined in \eqref{def-tilde-m}.
Then Lemmas \ref{lem-2t} and \ref{lem-2kt}
complete the proof.

\end{proof}

\begin{lem}\label{lem-2t}
   When both $m^\ast$ and $n$ are even, we have  \begin{align}       \left|\dfrac{\overline{X}_{m'}^{\sharp}}{\widetilde{X}_{m'}^{\sharp}} \right| = 2^{t}.
    \end{align}
\end{lem}
\begin{proof}
    Let $W_{ev}$ denote the set of even boundary components of $\overline\Sigma$, and $W_{\mathbb Z}$ denote the set of maps from $W_{ev}$ to $\mathbb Z_{2}^{n-1}$.
Obviously, $W_{\mathbb Z}$ has an abelian group structure. 
We will define a group homomorphism
$\theta\colon \overline{X}_{m'}^{\sharp}\rightarrow W_{\mathbb Z}$, where $\overline X_{m'}^{\sharp}$ is defined in \eqref{Xsharpbar}.
Let ${\bf k}=({\bf k}_1,{\bf k}_2)\in \overline X_{m'}^{\sharp}$,
where $\mathbf{k} = (
        \tfk_1,\tfk_2
    )\in \bZ^{V_{\lambda}}$ with $\tfk_1\in \bZ^{\obVlast}$ and $\tfk_2\in \bZ^{W}$.
    For any $\partial_i\in W_{ev}$,
we use $(\mathbf{b}_{i1}, \mathbf{b}_{i2},\dots , \mathbf{b}_{ir_i})$ to denote the restriction $\mathbf{k}_2|_{\partial_i}$ to the $i$-th boundary component $\partial_i$ of $\overline{\Sigma}$. 
From the definition of $\overline{X}_{m'}^{\sharp}$, we know 
${\bf b}_{ij}={\bf 0}$ in $\mathbb Z_{\widetilde{m}}$.
Define 
$$\theta_{{\bf k}}(\partial_i) = \frac{1}{\tilde{m}} {\bf b}_{i1}\in\mathbb Z_{2}^{n-1}.$$
Set
$$\theta\colon \overline X_{m'}^{\sharp}\rightarrow W_{\mathbb Z},\quad
{\bf k}\mapsto \theta_{\bf k}.$$
It is easy to see that $\theta$ is a well-defined group homomorphism.

For any attached triangle $\tau$, we can suppose that 
${\bk}|\tau=\lambda_2{\bf pr}_2$ in $\mathbb Z_n$ since ${\bf k}\in \Lambda_\lambda$. 
Since ${\bf k}_2={\bf 0}$ in $\mathbb Z_{d^{\ast}}$, then $\lambda_2=0\in\mathbb Z_{d_{*}}$.
This shows that ${\bk}|\tau={\bf 0}$ in $\bZ_{d^\ast}$ for any attached triangle $\tau$.

Note that the first equation in \eqref{eq-k2-zero-key} actually holds in $\bZ$ (not only in $\bZ_2$). 
Equation~\eqref{eq_key} implies that we can replace ${\bf t}$ with $2{\bf k}_1$ in $\bZ_{m'}$.  
Then, for any vertex $v$ in $\tau\setminus W$, we have the equation $$\bm{0}=2\bk_1^T(v)= (\frac{1}{n}D-C_1) \bk_2^T(v)=\bk_2(w_i)+\bk_2(w_j')+\bk_2(w_{n-k}')\text{ in $\bZ_{m'}$},$$
especially in $\bZ_{d}$. 
By dividing by $d^\ast$, we have 
$$\dfrac{\bk_2(w_i)}{d^\ast}+\dfrac{\bk_2(w_j')}{d^\ast}+\dfrac{\bk_2(w_{n-k}')}{d^\ast}=\bm{0}\text{ in $\bZ_{2}$.}$$
Then the same argument in the proof of Lemma~\ref{lem-k2-zero} shows that 
\begin{align*}
    \frac{1}{d^\ast}({\bf k}(w_1),\cdots, {\bf k}(w_1))= r(1,2,\cdots, n-1)\text{ in }\mathbb Z_2
\end{align*}
for some integer $r$.

Note that $\tilde{m}=d^\ast \bar m$, where $\bar m$ is odd. Then we have
    \begin{align}\label{eq-theta-restriction}
        \theta_{{\bf k}}(\partial_i)=\frac{1}{\tilde{m}} {\bf b}_{i1}=\frac{1}{d^\ast} {\bf b}_{i1} = r_i(1,2,\cdots, n-1)\in\mathbb Z_{2}^{n-1} \text{ for some $r_i\in\mathbb Z$}.
    \end{align}

Obviously, we have the following short exact sequence 
$$0\rightarrow \widetilde{X}_{m'}^{\sharp}\xrightarrow{L}  \overline{X}_{m'}^{\sharp}\xrightarrow{\theta}\im\theta\rightarrow 0,$$
where $L$ is the embedding.

Let $f\in W_{\mathbb Z}$ such that 
$f(\partial_i) = f_i (1,2,\cdots, n-1)\in\mathbb Z_{2}^{n-1} \text{ for some $f_i\in\mathbb Z$}$ for each $\partial_i \in W_{ev}$.
Since $\det(E)=1$, there exists a vector 
${\bf d}_i\in \mathbb Z^{n-1}$ such that 
${\bf d}_i E = (f_i\frac{\tilde{m}}{2},0,\cdots, 0)$.
Define $\mathbf{d}= (\mathbf{d}_{\partial_1},\mathbf{d}_{\partial_2},\cdots,\mathbf{d}_{\partial_b})\in \mathbb Z^{U}$, 
where $\mathbf{d}_{\partial_i}\in\mathbb Z^{r_i(n-1)}$ is the vector associated to $\partial_i$,  such that 
\begin{equation}\label{zero}
    \mathbf{d}_{\partial_i} =\begin{cases}
        {\bf 0} & r_i\text{ is odd},\\
        (-\mathbf{d}_i,\mathbf{d}_i,\cdots,-\mathbf{d}_i,\mathbf{d}_i) & r_i \text{ is even}.
    \end{cases}
    \end{equation}
    Set $\mathbf{d}' := (\mathbf{0},\mathbf{d})\in\mathbb Z^{V_{\lambda}'}$ and $\mathbf{f} := \mathbf{d}'\sfK_{\lambda}$. Note that $\mathbf{f}|_W=(2\mathbf{d}_i G, \dots ,2\mathbf{d}_i G)$ and this implies that $\mathbf{f}|_W$ is a multiple of $\tilde{m}$ with $G=EF$. 
    From this observation, it is easy to check that 
    ${\bf f}\in \overline{X}_{m'}^{\sharp}$
    and $\theta({\bf f}) = f$. Thus    $$\left|\dfrac{\overline{X}_{m'}^{\sharp}}{\widetilde{X}_{m'}^{\sharp}} \right|=|\im \theta|=2^t.$$

\end{proof}

\begin{lem}\label{lem-2kt}
When $m$ is even (i.e., $k>0$),
    we have 
    \begin{align*}       \left|\dfrac{\widetilde{X}_{m'}^{\sharp}}{X_{m'}^{\sharp}} \right| = 2^{(k-1)(n-1)t}.
    \end{align*}
\end{lem}
\begin{proof}
\def\bH{\beta}

The condition that $m$ is even  implies that $n'=n/d^\ast$ is odd  and $d^{\ast} = \gcd(n,2\tilde{m})$. There exist integers $\mu$ and $\nu$ such that $\mu n +\nu (2\tilde{m}) = d^{\ast}$.
Then we have $\mu n' +\nu (2\bar m) = 1$.

We use $W_{ev}$ to denote the set of even boundary components of $\overline\Sigma$.
We use $W_{\mathbb Z, k}$ to denote the set of maps from $W_{ev}$ to $\mathbb Z_{2^{k-1}}^{n-1}$.
Obviously, $W_{\mathbb Z, k}$ has an abelian group structure. 
We will define a group homomorphism
$\bH \colon \widetilde{X}_{m'}^{\sharp}\rightarrow W_{\mathbb Z,k}$, where $\widetilde{X}_{m'}^{\sharp}$ is defined in \eqref{Xsharp}.
Let ${\bf k}=({\bf k}_1,{\bf k}_2)\in \widetilde X_{m'}^{\sharp}$,
where $\mathbf{k} = (
        \tfk_1,\tfk_2
    )\in \bZ^{V_{\lambda}}$ with $\tfk_1\in \bZ^{\obVlast}$ and $\tfk_2\in \bZ^{W}$.
    For any $\partial_i\in W_{ev}$,
we use $(\mathbf{b}_{i1}, \mathbf{b}_{i2},\dots , \mathbf{b}_{ir_i})$ to denote the restriction $\mathbf{k}_2|_{\partial_i}$ to the $i$-th boundary component $\partial_i$ of $\overline{\Sigma}$. 
From the definition of $\widetilde{X}_{m'}^{\sharp}$, we know 
${\bf b}_{ij}={\bf 0}$ in $\mathbb Z_{2\tilde{m}}$.
Define 
$$\bH_{{\bf k}}(\partial_i) = \frac{1}{2\tilde{m}} {\bf b}_{i1}\in\mathbb Z_{2^{k-1}}^{n-1}.$$
Set
$$\bH\colon \widetilde X_{m'}^{\sharp}\rightarrow W_{\mathbb Z},\quad
{\bf k}\mapsto \bH_{\bf k}.$$
It is easy to see that $\bH$ is a well-defined group homomorphism.

Obviously, we have the following short exact sequence 
$$0\rightarrow {X}_{m'}^{\sharp}\xrightarrow{L}  \widetilde{X}_{m'}^{\sharp}\xrightarrow{\bH}\im\bH\rightarrow 0,$$
where $L$ is the embedding.

Let ${\bf k}=({\bf k}_1,{\bf k}_2)\in  \widetilde X_{m'}^{\sharp}$,
where $\mathbf{k} = 
(\tfk_1,\tfk_2)\in \bZ^{V_{\lambda}}$ with $\tfk_1\in \bZ^{\obVlast}$ and $\tfk_2\in \bZ^{W}$. 
For any $\partial_i\in W_{ev}$,
we use $(\mathbf{b}_{i1}, \mathbf{b}_{i2},\dots, \mathbf{b}_{ir_i})$ to denote the restriction $\mathbf{k}_2|_{\partial_i}$ to the $i$-th boundary component $\partial_i$ of $\overline{\Sigma}$. 
Since ${\bf k}$ is balanced, we have 
${\bf b}_{i1} = \lambda_i(1,2,\cdots,n-1) + n{\bf c}_i$ for some ${\bf c}_i$ from Lemma \ref{balance}.
Suppose ${\bf c}_i = (c_{i,1},\cdots,c_{i,n-1})$.
Then we have ${\bf b}_{i1} = (\lambda_i+nc_{i,1},\cdots,
\lambda_i(n-1) + n c_{i,n-1})$.
From the definition of $\widetilde X_{m'}^{\sharp}$, we know ${\bf b}_{i1}={\bf 0}$ in $\mathbb Z_{\tilde{m}}$. This shows $d^{\ast}|\lambda_i$.
Define $\lambda_i'=\lambda_i/d^{\ast}$.
Note that, for $1\leq j\leq n-1$, we have 
\begin{eqnarray*}
j\lambda_i+nc_{i,j} =0\text{ in }\mathbb Z_{\tilde{m}}&\Leftrightarrow&
j\lambda_i'+n'c_{i,j} =0\text{ in }\mathbb Z_{2\bar m}\\
&\Leftrightarrow&
 c_{i,j} = -\mu j\lambda_i' + (2 \bar m) l_{i,j}\text{ for some integers $\mu$ and $l_{i,j}$}
\end{eqnarray*}
since $n'$ and $2 \bar m$ are coprime.
Then, for each $1\leq j\leq n-1$, we have 
$$\frac{j\lambda_i+nc_{i,j}}{2\tilde{m}}
=\frac{j\lambda_i'+n'c_{i,j}}{2 \bar m} = \nu j\lambda_i' + n' l_{i,j}$$
by substituting $ c_{i,j} = -\mu j\lambda_i' + 2 \bar m l_{i,j}$ to the middle. 
This shows
\begin{align*}
    \alpha_{\bf k}(\partial_i)
    = (\nu \lambda_i' + n' l_{i,1}, 2\nu \lambda_i' + n' l_{i,2},\cdots, (n-1)\nu \lambda_i' + n' l_{i,n-1})\in\mathbb Z_{2^{k-1}}^{n-1}.
\end{align*}

We will show $W_{\mathbb Z, k}\subset\im\bH$. 
Suppose $f\in W_{\mathbb Z, k}$. 
For any $\partial_i\in W_{ev}$, suppose 
$f(\partial_i) = (f_{i,1},\cdots,f_{i,n-1})\in\mathbb Z_{2^{k-1}}^{n-1}$.
Since $n'$ is invertible in $\mathbb Z_{2^{k-1}}$, there exists an integer $s$ such that $sn'=1\in \mathbb Z_{2^{k-1}}$.
Set 
$\lambda_i:=n$,
$l_{i,j} := s(f_{i,j} - \nu jn')$, and  $c_{i,j}
:=-\mu j n' + ml_{i,j}$. Set ${\bf b}_i
:= \lambda_i (1,2,\cdots,n-1) + n(c_{i,1},\cdots,c_{i,n-1})$. 
From the construction, we have ${\bf b}_i={\bf 0}$ in $\mathbb Z_{2\tilde{m}}$ and $\frac{1}{2\tilde{m}}{\bf b}_i = f(\partial_i)\in\mathbb Z_{2^{k-1}}^{n-1}$.  From the definition of ${\bf b}_i$, we have ${\bf b}_i={\bf 0}$ in $\mathbb Z_n$. Since $2\bar m$ is even, then we have ${\bf b}_i={\bf 0}$ in $\mathbb Z_{2n}$. 
Define \(\mathbf{k}_2 \in \mathbb{Z}^{W}\) by setting  
\(
\mathbf{k}_2(v) = 0 \text{ if } v \text{ does not belong to any boundary component } \partial \in W_{ev}.
\)  
For each boundary component \(\partial_i \in W_{ev}\), we define \(\mathbf{k}_2\) on \(\partial_i\) as  
\(
\mathbf{k}_2|_{\partial_i} = (\mathbf{b}_1, \mathbf{b}_1, \dots, \mathbf{b}_1).
\)

Define ${\bf k}_1:=(\frac{1}{n} D- C_1)\frac{{\bf k}_2}{2}$. We have ${\bf k}_1={\bf 0}$ in $\mathbb Z_n$ since ${\bf k}_2={\bf 0}$
in $\mathbb Z_{2n}$. Set
${\bf k} = ({\bf k}_1,{\bf k}_2)\in  \widetilde X_{m'}^{\sharp}$. 
From the definition of ${\bf k}$, we have
${\bf k} = ({\bf k}_1,{\bf k}_2)\in  X_{m'}^{\sharp}$
and $\bH({\bf k}) = f.$
Thus    $$\left|\dfrac{\widetilde{X}_{m'}^{\sharp}}{{X}_{m'}^{\sharp}} \right|=|\im \bH|=|W_{\mathbb Z, k}|=2^{(k-1)(n-1)}.$$
\end{proof}

\begin{lem}\label{lem-parity-dif}
When $n$ is even and $m^{\ast}$ is odd,
    we have 
    $$
    \left|\dfrac{X_{m'}^\ast}{X_{m'}}\right|
    =    \left|\dfrac{\overline{X}_{m'}^{\ast}}{\overline X_{m'}}\right| =\begin{cases}
        1 & \text{if $\Si$ contains odd number of boundary components,}\\
        2 & \text{if $\Si$ contains even number of boundary components.}\\
    \end{cases}
    $$
\end{lem}
\begin{proof}
For any component $E$ of $\partial\Sigma$,
we use $w_t^{E}$ to denote the small vertex $w_t\in W$ in the attached triangle with $E$ as the attached edge, see Figure \ref{Fig;coord_uvw}. 

It follows from definitions that $\left|\dfrac{X_{m'}^\ast}{X_{m'}}\right|,\left|\dfrac{\overline{X}_{m'}^{\ast}}{\overline X_{m'}}\right|=1\text{ or }2.$
Namely, we have 
$\left|\dfrac{X_{m'}^\ast}{X_{m'}}\right| = 1$ if and only if there does not exist
${\bf k}\in X_{m'}^\ast$ such that ${\bf k}(w_i^E) = i\text{ in }\mathbb Z_2$, where $1\leq i\leq n-1$.
The similar statement holds for $\left|\dfrac{\overline{X}_{m'}^{\ast}}{\overline X_{m'}}\right|$.
We give a detailed proof only for $\left|\dfrac{X_{m'}^\ast}{X_{m'}}\right|$ since the similar argument works for $\left|\dfrac{\overline{X}_{m'}^{\ast}}{\overline X_{m'}}\right|$.

For any $v\in \overline{V}_{\lambda}$, suppose that 
$\text{tr}_\lambda^X(\gaa_v) = x^{{\bf k}}$. 
Recall that we defined a weighted directed graph $\widetilde{Y}_v$ (see Figure~\ref{Fig;skeleton}).
Suppose that the edges of $\widetilde{Y}_v$ labeled by $i,j,k$
connect the boundary components 
$E_1,E_2,E_3$ of $\Si$ respectively.

Note that 
$$\mathbbm{d}_{E_1}(\gaa_v)=\varpi_i,\quad 
\mathbbm{d}_{E_2}(\gaa_v)=\varpi_j,\quad
\mathbbm{d}_{E_3}(\gaa_v)=\varpi_k,$$
where $\mathbbm{d}$ and $\varpi_i$ are defined in \eqref{eq-def-degree-d} and \eqref{eq:varpi} respectively. 
 Then Corollary \ref{cor:LY12} and \eqref{eq:varpi_varpi} imply that 
    \begin{align*}
       {\bf k}(w_t^{E_1})=n\langle \mathbbm{d}_{E_1}(\gaa_v),\mathsf \varpi_t\rangle &=n\big( \text{min}\{i,t\}-\frac{it}{n}\big),\\
         {\bf k}(w_t^{E_2})=n\langle \mathbbm{d}_{E_2}(\gaa_v),\mathsf \varpi_t\rangle &= n\big(\text{min}\{j,t\}-\frac{jt}{n}\big),\\
        {\bf k}(w_t^{E_3})= n\langle \mathbbm{d}_{E_3}(\gaa_v),\mathsf \varpi_t\rangle &= n\big(\text{min}\{k,t\}-\frac{kt}{n}\big).
    \end{align*}
Since $i+j+k=n$, we have 
\begin{align}\label{eq-V-overline}
    n\left(\langle \mathbbm{d}_{E_1}(\gaa_v),\mathsf \varpi_t\rangle + \langle \mathbbm{d}_{E_2}(\gaa_v),\mathsf \varpi_t\rangle + \langle \mathbbm{d}_{E_3}(\gaa_v),\mathsf \varpi_t\rangle\right)=0\text{ in }\mathbb Z_n.
\end{align}

For any $v\in V_{\lambda}'\setminus \overline{V}_{\lambda}$, suppose that 
$\text{tr}_\lambda^X(\gaa_v) = x^{{\bf k}}$. Suppose that $E_1$ is the attached edge of the attached triangle containing $v$ and $E_2$ is another boundary edge of $\Sigma$
encounters with $\widetilde{Y}_v$ (it is possible that $E_2=E_1$). 
Note that $\mathbbm{d}_{E_1}(\gaa_v)=\varpi_i$ and $\mathbbm{d}_{E_2}(\gaa_v)=\varpi_{i+k}-\varpi_{k}$. 
Then, 
    \begin{align}\label{eq-boundary-parity}
       {\bf k}(w_t^{E_1})=n\langle \mathbbm{d}_{E_1}(\gaa_v),\mathsf \varpi_t\rangle &= n\big(\text{min}\{i,t\}-\frac{it}{n}\big),\\
         {\bf k}(w_t^{E_2})=n\langle \mathbbm{d}_{E_2}(\gaa_v),\mathsf \varpi_t\rangle &= n\big(\text{min}\{i+k,t\}-\frac{(i+k)t}{n}-\text{min}\{k,t\}+\frac{kt}{n}\big).
    \end{align}
So we have 
\begin{align}\label{eq-V-boundary}
    n(\langle \mathbbm{d}_{E_1}(\gaa_v),\mathsf \varpi_t\rangle + \langle \mathbbm{d}_{E_2}(\gaa_v),\mathsf \varpi_t\rangle)=0\text{ in }\mathbb Z_n.
\end{align}

Equations \eqref{eq-V-overline}, \eqref{eq-V-boundary}, and Corollary \ref{cor:LY12} imply that 
\begin{align}\label{eq-parity-key}
    \sum_{E:\text{ a boundary edge}}
{\bf k}(w_t^E) =0\text{ in }\mathbb Z_n
\text{ for $1\leq t\leq n-1$.}
\end{align}
Then Corollary \ref{cor-quantum-frame} implies that Equation \eqref{eq-parity-key} holds for any ${\bf k}\in\Lambda_\lambda$. 

Suppose that the number of components of $\partial \Sigma$ is odd. 
For any ${\bf k}\in X_{m'}^\ast$ or $\overline{X}_{m'}^{\ast}$, we have Equation \eqref{eq-parity-key}.
Assume $E$ is a component of $\partial\Sigma$ and ${\bf k}(w_i^E) = i\text{ in }\mathbb Z_2$ for any $1\leq i\leq n-1$. Then we have 
\begin{align}
    \sum_{E:\text{ a boundary edge}}
{\bf k}(w_i^E) = i \text{ in }\mathbb Z_2.
\end{align}
When $i$ is odd, this contradicts with Equation \eqref{eq-parity-key}, i.e., there is no $\mathbf{k}$ satisfying ${\bf k}(w_i^E) = i\text{ in }\mathbb Z_2$ for any $1\leq i\leq n-1$.

In the following, we will construct central elements explicitly. Let $r$ denote the number of components of $\partial \Sigma$, and suppose $r$ is even. 
There exists an $n$-web diagram  $\al$ in $\Sigma$ consisting of $r/2$ arcs $\alpha_i$ such that (1) each arc connects two distinct components of $\partial\Sigma$, (2) no boundary edge of $\Sigma$ which does not intersect with $\alpha$, and (3) there is no crossing for $\al$. 
We state this $n$-web such that the starting (resp. ending) point of each arc is stated with $n$ (resp. $1$). 
Then we can regard $\al=\cup_{1\leq i\leq \frac{r}{2}}\al_i$ as a stated $n$-web diagram, each $\al_i$ is the stated arc in $\Sigma$ considered as above. 
For any positive integer $k$, let $\alpha_i^{(k)}$ denote the stated $n$-web obtained from $\alpha_i$ by taking $k$ parallel copies with respect to the vertical framing. 
Define $\al^{(k)} = \cup_{1\leq i\leq \frac{r}{2}}\al_i^{(k)}$.
Note that $(q^{\frac{1}{n}})^{m^{\ast}} =-1$ and $q^{m^{\ast}}=1$ since $n$ is even.
Then \cite[Lemmas 7.2, 7.3, 7.6, and 7.7]{Wan23} imply that
$\al^{(m^{\ast})}\in \mathcal Z(\cS_n(\Sigma))$. 
Note that $\al$ is a part of some saturated system defined in \cite[Section 8.2]{LS21}.
Using the isomorphism in  \cite[Theorem 8.8 (2)]{LS21}, we have 
$0\neq\al^{(m^{\ast})}\in \mathcal Z(\cS_n(\Sigma))$.
Suppose
$$\text{tr}_\lambda^X(\al^{(m^{\ast})})
=\sum_{1\leq i\leq N} c_i x^{{\bf k}_i},$$
where $N$ is a positive integer, $c_i\in\mathbb C^{\ast}$, and ${\bf k}_i\in\Lambda_\lambda$ are distinct with each other. 
We have $$\sum_{1\leq i\leq N} c_i x^{{\bf k}_i}\in\mathcal Z (\cX^{{\rm bl}}_{\hat{q}}(\Sigma,\lambda))$$
by combining Theorem \ref{thm-transition-LY}, Theorem \ref{traceA} (a), and $\al^{(m^{\ast})}\in \mathcal Z(\cS_n(\Sigma))$. 
Theorem \ref{center_torus} and
Lemma \ref{PI} imply that each ${\bf k}_i\in  X_{m'}^{\ast}+ \Lambda^X_\partial$ for $1\leq i\leq N$, where $\Lambda^X_\partial$ is defined in \eqref{eq-definition-boundary-X}. 
In particular, we have ${\bf k}_1
={\bf a} + {\bf b}$, where ${\bf a}\in X_{m'}^{\ast}$ and ${\bf b}\in \Lambda^X_\partial$. 
Equation \eqref{eq-boundary-parity} implies
${\bf b}(w^E_t) = 0\text{ in }\mathbb Z_2$, where
$1\leq t\leq n-1$ and $E$ is a component of $\partial\Sigma$.
From the definition of $\al$ and Corollary \ref{cor:LY12}, we have 
\begin{align}\label{eq-a-w-e-t}
    {\bf a}(w^E_t) = {\bf k}_1(w^E_t) = m^{\ast}n-m^{\ast}t.
\end{align}
Since $n$ is even and $m^{\ast}$ is odd, then ${\bf a}(w^E_t)=t\in\mathbb Z_2$.
Thus ${\bf a}\in X_{m'}^{\ast}$ and ${\bf a}\notin X_{m'}$.

\end{proof}

Let us recall the following lemma to use in the proof of Lemma~\ref{lem:X/X}. 
\begin{lem}[\cite{Wan23}]\label{lem-height-exchange}
In $\cS_n(\Sigma)$, we have
    $$
\raisebox{-.30in}{
\begin{tikzpicture}
\tikzset{->-/.style=
{decoration={markings,mark=at position #1 with
{\arrow{latex}}},postaction={decorate}}}
\draw [line width =1.5pt,decoration={markings, mark=at position 0.91 with {\arrow{>}}},postaction={decorate}] (0,1)--(0,-1);
\draw [color = black, line width =1pt](-1,0.5) --(-0,-0.5);
\draw [color = black, line width =1pt](-1,-0.5) --(-0.6,-0.1);
\draw [color = black, line width =1pt](-0.4,0.1) --(0,0.5);
\draw [color = black, line width =1pt](-1.5,-0.5) --(-1,-0.5);
\draw [color = black, line width =1pt](-1.5,0.5) --(-1.4,0.5);
\node [right]at(0,0.5) {\small $i$};
\node [right]at(0,-0.5) {\small $j$};
\node at(-1.2,0.5) {\small $m$};
\draw[color=black] (-1.2,0.5) circle (0.2);
\filldraw[fill=white,line width =0.8pt]  (-0.2,0.3) circle (0.085);
\filldraw[fill=white,line width =0.8pt]  (-0.2,-0.3) circle (0.085);
\end{tikzpicture}}= 
     (q^{-\frac{1}{n}+\delta_{ij}})^{m}
\raisebox{-.30in}{
\begin{tikzpicture}
\tikzset{->-/.style=
{decoration={markings,mark=at position #1 with
{\arrow{latex}}},postaction={decorate}}}
\draw [line width =1.5pt,decoration={markings, mark=at position 0.91 with {\arrow{>}}},postaction={decorate}] (0,1)--(0,-1);
\draw [color = black, line width =1pt](-1,-0.5) --(0,-0.5);
\draw [color = black, line width =1pt](-1,0.5) --(-0,0.5);
\draw [color = black, line width =1pt](-0.3,0.5) --(0,0.5);
\node [right]at(0,0.5) {\small $j$};
\node [right]at(0,-0.5) {\small $i$};
\filldraw[fill=white] (-0.6,0.5) circle (0.2);
\filldraw[fill=white,line width =0.8pt]  (-0.2,0.5) circle (0.085);
\filldraw[fill=white,line width =0.8pt]  (-0.2,-0.5) circle (0.085);
\node at(-0.6,0.5) {\small $m$};
\end{tikzpicture}},  \qquad  
\raisebox{-.30in}{
\begin{tikzpicture}
\tikzset{->-/.style=
{decoration={markings,mark=at position #1 with
{\arrow{latex}}},postaction={decorate}}}
\draw [line width =1.5pt,decoration={markings, mark=at position 0.91 with {\arrow{>}}},postaction={decorate}] (0,1)--(0,-1);
\draw [color = black, line width =1pt](-1,0.5) --(-0,-0.5);
\draw [color = black, line width =1pt](-1,-0.5) --(-0.6,-0.1);
\draw [color = black, line width =1pt](-0.4,0.1) --(0,0.5);
\draw [color = black, line width =1pt](-1.5,-0.5) --(-1,-0.5);
\draw [color = black, line width =1pt](-1.5,0.5) --(-1,0.5);
\node [right]at(0,0.5) {\small $i$};
\node [right]at(0,-0.5) {\small $j$};
\draw[color=black,fill=white] (-1.2,-0.5) circle (0.2);
\node at(-1.2,-0.5) {\small $m$};
\filldraw[fill=black,line width =0.8pt]  (-0.2,0.3) circle (0.085);
\filldraw[fill=white,line width =0.8pt]  (-0.2,-0.3) circle (0.085);
\end{tikzpicture}}= 
     (q^{\frac{1}{n}-\delta_{i\bar{j}}})^{m}
\raisebox{-.30in}{
\begin{tikzpicture}
\tikzset{->-/.style=
{decoration={markings,mark=at position #1 with
{\arrow{latex}}},postaction={decorate}}}
\draw [line width =1.5pt,decoration={markings, mark=at position 0.91 with {\arrow{>}}},postaction={decorate}] (0,1)--(0,-1);
\draw [color = black, line width =1pt](-1,-0.5) --(0,-0.5);
\draw [color = black, line width =1pt](-1,0.5) --(-0,0.5);
\draw [color = black, line width =1pt](-0.3,0.5) --(0,0.5);
\node [right]at(0,0.5) {\small $j$};
\node [right]at(0,-0.5) {\small $i$};
\filldraw[fill=white,line width =0.8pt]  (-0.2,0.5) circle (0.085);
\filldraw[fill=black,line width =0.8pt]  (-0.2,-0.5) circle (0.085);
\draw[color=black,fill=white] (-0.6,-0.5) circle (0.2);
\node at(-0.6,-0.5) {\small $m$};
\end{tikzpicture}},   
%
$$
where on the left-hand side of the equality 
$\begin{tikzpicture}
\tikzset{->-/.style=
{decoration={markings,mark=at position #1 with
{\arrow{latex}}},postaction={decorate}}}
\draw [color = black, line width =1pt](0.8,0) --(1.2,0);
\draw [color = black, line width =1pt](0,0) --(0.4,0);
%
\node at(0.6,0) {\small $m$};
\draw[color=black] (0.6,0) circle(0.2);
\end{tikzpicture}$ is a part of  $m$ parallel copies of a stated framed oriented arc and the other single arc (the one stated by $i$) is not a part of these $m$ parallel copies.
\end{lem}

\begin{lem}\label{lem:X/X}
    When $m^{\ast}$ and $n$ are even, we have the following
$$  \left|\dfrac{\overline{X}_{m'}^{\ast}}{\overline X_{m'}}\right| =\begin{cases}
       2 &\text{$b=t$ and $n'$ is even,}\\
       1 & \text{otherwise.}
   \end{cases}
    $$
\end{lem}
\begin{proof}
{\bf Case 1 ($b\neq t$)}:   Suppose ${\bf k}=({\bf k}_1,{\bf k}_2)\in \overline{X}_{m'}^{\ast}$.
When $b\neq t$, i.e.,  $\overline{\Sigma}$ contains  at least one  odd boundary component, it is obvious since ${\bf k}|_{\partial _i}={\bf 0}\in\mathbb Z_{m^{\ast}}$ when $r_i$ is odd. 
From the definition of $X_{m'}^\ast$ and $\overline{X}_{m'}^{\ast}$, it follows that
${\bf k}_2={\bf 0}$ in $\mathbb Z_2$.

\noindent {\bf Case 2 ($b=t$ and $n'$ is odd)}:  
Suppose ${\bf k}=({\bf k}_1,{\bf k}_2)\in \overline{X}_{m'}^\ast$.
It suffices to show ${\bf k}_2={\bf 0}$ in $\mathbb Z_2$.

Let $\tau$ and $\tau'$ be two attached triangles as shown in Figure \ref{Fig;tau_tau'}.
It follows from Equation~\eqref{eq-k2-zero-key} (note that the first equality in Equation~\eqref{eq-k2-zero-key} holds in $\mathbb Z$) that, for $v=(\frac{n}{2}, \frac{n}{2}, 0)$ in $\tau'$, $2{\bf k}_1(v) = {\bf k}_2(w_{n/2}) + {\bf k}_2(w_{n/2}')\in\mathbb Z_{m'}$.
The condition ${\bf b}_{ij} = (-1)^{j-1} {\bf b}_{i1}\text{ in }\mathbb Z_{m'}$ in the definition of $\overline{X}_{m'}^\ast$ implies that
$$\bk_2(w_{n/2})+\bk_2(w_{n/2}')=0 \text{ in $\bZ_{m'}$}.$$
Then we have $2{\bf k}_1(v)=0\text{ in }\mathbb Z_{d}$.

As we saw, we have 
${\bf k}|_{\tau'} = \lambda' {\bf pr}_2 + n{\bf c}'$, where ${\bf k}|_{\tau'}$ is the restriction of ${\bf k}$ to $\tau'$ and $\lambda'\in\mathbb Z$.
Then 
$2{\bf k}_1(v)=n\lambda' + 2n{\bf c}'(v)=0\text{ in }\mathbb Z_{d}$.
Recall that $n=n'd^{\ast}$, then we have 
$$n'\lambda'=n'\lambda' + 2n'{\bf c}'(v)=0\text{ in }\mathbb Z_{2}.$$
Since $n'$ is odd, we have 
$\lambda'=0$ in $\mathbb Z_2$. 
Hence, ${\bf k}|_{\tau'} = \lambda' {\bf pr}_2 + n{\bf c}'={\bf 0}$ in $\bZ_2$.
This implies ${\bf k}_2={\bf 0}$ in $\mathbb Z_2$.

\noindent {\bf Case 3 ($b=t$ and $n'$ is even)}: 
Obviously, we have $\left|\dfrac{\overline{X}_{m'}^{\ast}}{\overline X_{m'}}\right|=1$ or $2$. To show it is $2$, it suffices to find a non-trivial element in
$\dfrac{\overline{X}_{m'}^{\ast}}{\overline X_{m'}}$.

Let $\partial$ be an even boundary component of $\overline{\Sigma}$ and $\#\partial$ be the number of connected components of $\partial\Si$ contained in $\partial$.
Let $\alpha_\partial$ be a crossingless $n$-web diagram consisting of $\#\partial/2$ corner arcs of $\partial$ with counterclockwise orientations such that each component of $\partial\Sigma$ intersects $\alpha_\partial$ exactly once; see Figure \ref{Fig;central}.
We state the starting (resp. ending) point of each arc of $\alpha_\partial$ by $n$ (resp. $1$). 
Set $\alpha:=\bigcup_\partial \alpha_\partial= \prod_\partial \alpha_\partial$, where $\partial$ is taken over all components of $\partial\overline{\Si}$.
We say the component of $\partial\Sigma$ of the \textbf{first type} if this component contains an endpoint of $\alpha$ stated by $1$.

We will check the commutativity $\gaa_v\alpha^{(m)}=\alpha^{(m)}\gaa_v$ for any $v\in V_\lambda'$ to show that $\alpha^{(m)}$ is central. 
Let $\beta$ be any stated $n$-web diagram such that all its endpoints point toward $\partial\Sigma$. 
Let $k_1$ be the number of endpoints of $\beta$ lying the components of the first type, and let $k_2$ be the number of the rest of endpoints of $\beta$ on $\partial$. 
Lemma \ref{lem-height-exchange}  shows that
\begin{align*}
    \alpha^{(m)}\beta&=\begin{array}{c}\includegraphics[scale=0.45]{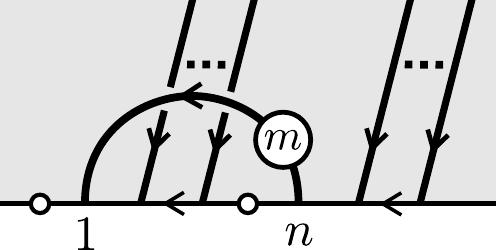}\end{array} = (q^{-\frac{m}{n}})^{k_1} \begin{array}{c}\includegraphics[scale=0.45]{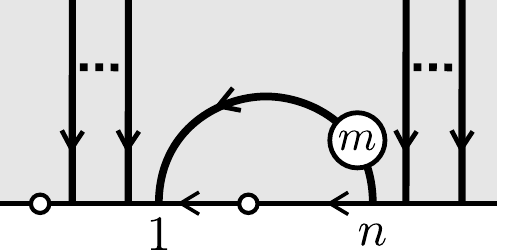}\end{array},\\
    \beta\alpha^{(m)}&=\begin{array}{c}\includegraphics[scale=0.45]{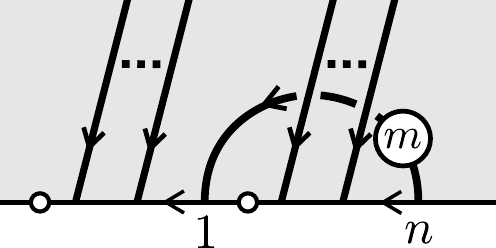}\end{array} = (q^{\frac{m}{n}})^{k_2} \begin{array}{c}\includegraphics[scale=0.45]{draws/central_m_03.pdf}\end{array}
\end{align*}
Then, we have 
\begin{align}\label{eq-beta-alpha-com}
    \beta\alpha^{(m)}=(q^{-\frac{m}{n}})^{k_1+k_2}\alpha^{(m)}\beta.
\end{align}

Since $n'$ is odd, we have that $\gcd(m',n)
=\gcd(m',2n)=d$. This implies that 
the order of $q$ is $m$.
Since $p=t$, Equation \eqref{eq-beta-alpha-com} and $q^m=1$ show that
$\gaa_v\alpha^{(m)}=\alpha^{(m)}\gaa_v$
for any $v\in V_\lambda'$.
It follows from Corollary \ref{cor-quantum-frame} that $\alpha^{(m)}$ is central.
As the last paragraph of the proof of Lemma \ref{lem-parity-dif}, one can show that $\alpha^{(m)}$ is a non-trivial element in 
$\dfrac{\overline{X}_{m'}^{\ast}}{\overline X_{m'}}$.
\end{proof}

\begin{figure}[h]
    \centering
    \includegraphics[width=70pt]{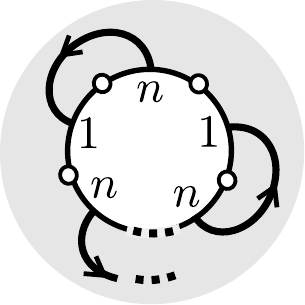}
    \caption{A central element on an even boundary component}\label{Fig;central}
\end{figure}

\begin{rem}
    In the proof of Lemma~\ref{lem:X/X}, if we change the states of $\alpha$ with arbitrary states, it is also central since the proof of $\alpha^{(m)}$ being central is independent of the choice of the state of $\alpha$.
\end{rem}

\subsection{Proof of Theorem \ref{center_torus}}\label{sec;pf}
\begin{proof}[Proof of Theorem \ref{center_torus}]

{\bf When $m^{\ast}$ is even:}
From Lemma~\ref{quantum},
it suffices to show that 
$\Gamma_{m'}$ is the solution for the equation system \eqref{eq_key}.
It is obvious that elements in $\Gamma_{m'}$ are solutions for equation system \eqref{eq_key}.
Let ${\bf k}$ be a solution for 
equation system \eqref{eq_key}.
As we will see in \eqref{ri-odd-bi-mstart} (the argument for \eqref{ri-odd-bi-mstart} does not depend on the parity of $m^\ast$),
we have
${\bf b}_{ij}={\bf 0}$ in $\mathbb Z_{m^{\ast}}$ when $r_i$ is odd.
Then ${\bf k}\in \Gamma_{m'}$ when $n$ is even.
The first equation in the equation system \eqref{eq_key} shows that
$$(\frac{1}{n}D- C_1){\bf k}_2={\bf 0}\text{ in }\mathbb Z_2.$$
Then Lemma~\ref{lem-k2-zero} implies that
${\bf k}_2={\bf 0}$ in $\mathbb Z_2$
and ${\bf k}\in \Gamma_{m'}$ when $n$ is odd.

In the rest of the proof, we give the detailed proof for the case when $m^{\ast}$ is odd. 

{\bf When $m^{\ast}$ is odd:}    Lemma \ref{quantum} implies $$\mathcal Z(\A) = \cR\text{-span}\{a^{\tft}\mid  \langle {\bf t},{\bf t}'\rangle_{\sfP_{\lambda}}=0\text{ in } \mathbb{Z}_{m''},\forall {\bf t}'\in\mathbb{Z}^{V_{\lambda}'}  \}.$$
    For ${\bf t}_0\in \mathbb Z^{V_\lambda'}$,
    Lemma~\ref{lem:invertible_KH} 
    Equation~\eqref{eq-anti-matric-P-def}
    implies 
    $\text{$\tft \sfP_{\lambda} \tft_0^{T} = 0$ {in $\mathbb Z_{m''}\iff$}}\tft \sfK_{\lambda} \sfQ_{\lambda} (\tft_0 \sfK_{\lambda})^T = 0 $ in $\mathbb Z_{m''}$ for all $\tft\in \mathbb Z^{V_{\lambda}'}.$
    Set $\tfk _0= \tft_0 \sfK_{\lambda}$. Then we have $\text{$\tft \sfP_{\lambda} \tft_0^{T} = 0$ {in $\mathbb Z_{m''}\iff$}}\tft\sfK_{\lambda} \sfQ_{\lambda} \tfk _0^{T}= \bm{0}$ in $\mathbb Z_{m''}.$

We regard $\mathbf{k}_0 = (
        \tfk_1,\tfk_2
    )\in \bZ^{V_{\lambda}}$ with $\tfk_1\in \bZ^{\obVlast}$ and $\tfk_2\in \bZ^{W}$. 
  As in Remark \ref{rem-center-equation}, we have equation \eqref{eq_key}.

The first equation in \eqref{eq_key} implies that 
$(\frac{1}{n}D-C_1) \bk_2^T={\bf 0}$ in $\mathbb Z_2$ since $m'$ is even. 
Lemma~\ref{lem-k2-zero} shows that
\begin{align}\label{old-n-zero-k2}
    {\bf k}_2={\bf 0} \text{ in }\mathbb Z_2
    \text{ when $n$ is odd.}
\end{align}

From Lemma \ref{boundary_center}, 
$\tra(\mathsf{B})$ is contained in the center of $\A$.
To show $\{a^{\mathbf{k}}\mid \mathbf{k}\in \Gamma_{m'} \}$ is contained in the center of $\A$, it suffices to show that 
$\{x^{\mathbf{k}}\mid \mathbf{k}\in X_{m'} \}$ is contained in the center of $\Xbll$.
It is obvious that each ${\bf k}=(\tfk_1,\tfk_2)\in X_{m'}$ is a solution of Equation \eqref{eq_key}, where 
  $\tfk_1\in \bZ^{\obVlast}$ and $\tfk_2\in \bZ^{W}$.
This implies that $x^{\bf k}\ ({\bf k}\in X_{m'})$ is contained 
in the center of $\Xbll$.

    Suppose $\tft_0\in\mathbb Z^{V_{\lambda}'}$ such that $\tft \sfP_{\lambda} \tft_0^{T} = 0$ {in $\mathbb Z_{m''}$} for any $\tft\in \mathbb Z^{V_{\lambda}'}$, where we regard $\tft,\tft_0$ as row vectors. Set
    $\tfk _0= \tft_0 \sfK_{\lambda} = (
        \tfk_1,\tfk_2
    )\in \bZ^{V_{\lambda}}$ with $\tfk_1\in \bZ^{\obVlast}$ and $\tfk_2\in \bZ^{W}$.

Suppose $\overline{\Sigma}$ has boundary components $\partial_1,\cdots,\partial_b$, and each boundary component $\partial _i$ contains $r_i$ punctures. Suppose $\tfk_2 =  (
\tfk_{\partial_1},\tfk_{\partial_2},\cdots,\tfk_{\partial_b}
)$ where $\tfk_{\partial_i}\in\mathbb Z^{r_i(n-1)}$ is the row vector associated to $\partial_i$ for each $1\leq i\leq b$.
From Lemmas \ref{matrixA} and \ref{matrixB}, we know $A = diag\{A_1,\cdots,A_b\}$ and $B = diag\{B_1,\cdots,B_b\}$ with $A_i=-n I$ and $B_i$ given as \eqref{Bi}. Then the second equality in Equation \eqref{eq_key} implies $\frac{1}{n}(B_i - A_i) \tfk_{\partial _i}^T = \bm{0}$ in $\mathbb Z_{m'}$ for each $1\leq i\leq b$.
Suppose  $\tfk_{\partial_i}= (
\tfb_{i1},\tfb_{i2},\cdots,\tfb_{ir_i})$, where $\tfb_{ij}\in \bZ^{n-1}$ for any $j$. 
Then $\frac{1}{n}(B_i - A_i) \tfk_{\partial _i}^T = \bm{0}$ in $\mathbb Z_{m'}$ implies 
\begin{equation}\label{eq's}
\begin{cases}
\tfb_{ir_i} + \tfb_{i1} =\bm{0},\\
\tfb_{i1} + \tfb_{i2} =\bm{0},\\
 \tfb_{i2} + \tfb_{i3} =\bm{0},\\
\;\vdots\\
\tfb_{ir_i-1} + \tfb_{ir_i} =\bm{0},\\
\end{cases}\text{ in $\mathbb Z_{m'}$}.
\end{equation}

Suppose $r_i$ is odd. Then we have ${\bf b}_{ij} = (-1)^{j-1} {\bf b}_{i1}$ in $\mathbb Z_{m'}$ and $2{\bf b}_{ij} = {\bf 0}$ in $\mathbb Z_{m'}$, i.e., 
\begin{align}\label{ri-odd-bi-mstart}
    {\bf b}_{ij} =  {\bf b}_{i1}\text{ in }\mathbb Z_{m'}\text{ and }{\bf b}_{ij} = {\bf 0}\text{ in }\mathbb Z_{m^\ast}
\end{align}
for $1\leq j\leq r_i.$ 
When both $m^{\ast}$ and $n$ are odd, we have ${\bf b}_{ij}={\bf 0}$ in $\mathbb Z_{2}$ (Equation~\eqref{old-n-zero-k2}) and ${\bf b}_{ij}={\bf 0}$ in $\mathbb Z_{m'}$ for $1\leq j\leq r_i$.

When $r_i$ is even, we have
$${\bf b}_{ij} = (-1)^{j-1} {\bf b}_{i1}\text{ in }\mathbb Z_{m'}.$$
Since $\mathbf{b}_{i1}$ is balanced, Lemma \ref{balance} implies there exists $l_i\in\mathbb Z$ such that 
\begin{align}
\mathbf{b}_{i1} - l_i(1,2,\cdots,n-1) = \bm{0}\text{\quad in } \mathbb Z_n.\label{eq:bi1}
\end{align}

We consider the case when $n$ is even.
There exists an integer $l_i'$ such that $2l_i' = l_i\in\mathbb Z_{d^{\ast}}$ since $d^{\ast}$ is odd; see Section \ref{notation}.
Since $\det(E)=1$ (Equation~\eqref{matrixEF}), there exists ${\bf d}_i\in\mathbb Z^{n-1}$ such that ${\bf d}_i E =(-l_i',0,\cdots,0)$. Then 
$${\bf d}_i G = {\bf d}_i EF= (-l_i',0,\cdots,0) F
=l_i'(1,2,\cdots,n-1)\text{ in }\mathbb Z_n.$$
Note that $n$ is a multiple of $d^{\ast}$.
This implies that 
$$2{\bf d}_i G 
=2l_i'(1,2,\cdots,n-1)=l_i(1,2,\cdots,n-1)={\bf b}_{i1}\text{ in }\mathbb Z_{d^{\ast}}.$$

If $n$ is odd, then there exists $l_i'\in\mathbb Z$ such that 
\begin{align}
{\bf b}_{i1} = l_i(1,2,\cdots,n-1)=2l_i'(1,2,\cdots,n-1)\text{ in }\mathbb Z_n.\label{eq:bij_Zn}
\end{align}
Similarly as above, there exists ${\bf d}_i\in\mathbb Z^{n-1}$ such that ${\bf d}_i E =(-l_i',0,\cdots,0)$ and 
$2{\bf d}_i G ={\bf b}_{i1}$ in $\mathbb Z_n$.
Also, we have $2{\bf d}_i G ={\bf b}_{i1}$ in $\mathbb Z_2$. Since $n$ is odd, then we have 
$2{\bf d}_i G ={\bf b}_{i1}$ in $\mathbb Z_{2n}.$
Note that $2n$ is a multiple of $d$.

For both parity of $n$, 
define $\mathbf{d}= (\mathbf{d}_{\partial_1},\mathbf{d}_{\partial_2},\cdots,\mathbf{d}_{\partial_b})\in \mathbb Z^{U}$, 
where $\mathbf{d}_{\partial_i}\in\mathbb Z^{r_i(n-1)}$ is the vector associated to $\partial_i$,  such that 
\begin{equation}\label{zero}
    \mathbf{d}_{\partial_i} =\begin{cases}
        {\bf 0} & r_i\text{ is odd},\\
        (-\mathbf{d}_i,\mathbf{d}_i,\cdots,-\mathbf{d}_i,\mathbf{d}_i) & r_i \text{ is even}.
    \end{cases}
    \end{equation}
    Set $\mathbf{d}' := (\mathbf{0},\mathbf{d})\in\mathbb Z^{V_{\lambda}'}$ and $\mathbf{f} := \mathbf{d}'\sfK_{\lambda}$.
    From the definition of $\mathbf{d}$, we know $\sfK_{\lambda}\sfQ_{\lambda}\mathbf{f}^{T} = \bm{0}$ in $\mathbb Z_{m''}$.
    We regard $\mathbf{f} = (
        \mathbf{f}_1,\mathbf{f}_2
    )\in \mathbb Z^{V_{\lambda}}$ with $\mathbf{f}_1\in \bZ^{\obVlast}$ and $\mathbf{f}_2\in \bZ^{W}$. 
    By replacing $\mathbf{k}_i$ with $\mathbf{f}_i\ (i=1,2)$, $\mathbf{f}_1$ and $\mathbf{f}_2$ satisfy Equation \eqref{eq_key}, i.e., $\mathbf{f}$ is a concrete solution. From \eqref{eq_K} and Lemma \ref{matrixK}, we have 
    $\mathbf{f}_2 = \mathbf{d}(K_{32}-K_{22}) = (\mathbf{d}_{\partial_1}L_1,\mathbf{d}_{\partial_2}L_2,\cdots,\mathbf{d}_{\partial_b}L_b)$.

     Lemma \ref{matrixK} implies 
     \begin{align}\label{eq-cal-dl}
         \mathbf{d}_{\partial_i}L_i = (2\mathbf{d}_{i}G,-2\mathbf{d}_{i}G,\cdots,2\mathbf{d}_{i}G,-2\mathbf{d}_{i}G)={\bf k}_{\partial _i}\text{ in }
     \begin{cases}
         \mathbb Z_d & n\text{ is odd},\\
         \mathbb Z_{d^{\ast}} & n\text{ is even}.
     \end{cases}
     \end{align}

Set $\mathbf{h} := \mathbf{f} - \mathbf{k}_0$. Then, we have $\mathbf{h}  =
(\mathbf{h} _1,\mathbf{h} _2)$, where $\mathbf{h}_1 = \mathbf{f}_1 - \mathbf{k}_1$ and $\mathbf{h}_2 = \mathbf{f}_2 - \mathbf{k}_2$. Note that $\mathbf{h}_1$ and $\mathbf{h}_2$ satisfy Equation \eqref{eq_key}.
We regard $\mathbf{h}_2 =  (\tff_{\partial_1},\tff_{\partial_2},\cdots,\tff_{\partial_b})$ where $\tff_{\partial_i}\in\mathbb Z^{r_i(n-1)}$ is the row vector associated to $\partial_i$ for each $1\leq i\leq b$.
We have 
$\tff_{\partial_i} = (\tff'_{i},-\tff'_{i},\cdots,\tff'_{i},-\tff'_{i})$ in $\mathbb Z_{m'}$ if $r_i$ is even, where 
$\tff'_{i} =(h_{i,1},h_{i,1},\cdots,h_{i,n-1}) \in\mathbb Z^{n-1}$.
    
Consider the case when $n$ is odd.
From the above discussion, we know 
${\bf h}_2={\bf 0}$ in $\mathbb Z_d$.
Since $\tff_2 =\bm{0}$ in $\mathbb Z_{d}$, each $h_{i,j}$ is a multiple of $d$ for $1\leq j \leq n-1$. The equation 
$(x_1,x_2,\cdots,x_{n-1})F= \frac{1}{2}\tff'_{i}$ in $\mathbb Z_{m^\ast}$ implies
\begin{equation}\label{equation}
\begin{cases}
(n-1)x_1 - nx_2 =\frac{1}{2}h_{i,1},\\
(n-2)x_1 - nx_3 =\frac{1}{2}h_{i,2},\\
\;\vdots\\
2x_1 - nx_{n-1} =\frac{1}{2}h_{i,n-2},\\
x_1 =\frac{1}{2} h_{i,n-1},\\
\end{cases}\text{ in $\mathbb Z_{m^\ast}$}.
\end{equation}
Equation \eqref{equation} is equivalent to the following equations:
\begin{equation}\label{equation2} 
\begin{cases}
\dfrac{2n}{d} x_2 =(n-1)\dfrac{h_{i,n-1}}{d}-\dfrac{h_{i,1}}{d} \text{ in $\mathbb Z_{m}$},\medskip\\
\dfrac{2n}{d}x_3 =(n-2)\dfrac{h_{i,n-1}}{d}-\dfrac{h_{i,2}}{d} \text{ in $\mathbb Z_{m}$},\\
\;\vdots\\
\dfrac{2n}{d}x_{n-1} =2\dfrac{h_{i,n-1}}{d}-\dfrac{h_{i,n-2}}{d} \text{ in $\mathbb Z_{m}$},\\
x_1 =\frac{1}{2} h_{i,n-1} \text{ in $\mathbb Z_{m^\ast}$}.\\
\end{cases}
\end{equation}
Since $2n/d$ and $m$ are coprime, Equation \eqref{equation2} has a solution, and so does Equation \eqref{equation}.
Fix a solution $\mathbf{x}_i'\in \mathbb Z^{n-1}$ of Equation \eqref{equation}. Since the inverse matrix $E^{-1}$ is also an integer matrix, we set  
$\mathbf{x}_i'':=\mathbf{x}_i' E^{-1}\in \mathbb Z^{n-1}$. 
Thus we have 
$\mathbf{x}_i''G = \mathbf{x}_i''EF=\mathbf{x}_i'F=\frac{1}{2}\tff_i'$ in $\mathbb Z_{m^\ast}$. This implies 
$2\mathbf{x}_i''G = \tff_i'$ in $\mathbb Z_{m'}$. 

Set $\mathbf{x}_2= (\mathbf{x}_{\partial_1},\mathbf{x}_{\partial_2},\cdots,\mathbf{x}_{\partial_b})\in \mathbb Z^{U}$, where $\mathbf{x}_{\partial_i}\in\mathbb Z^{r_i(n-1)}$ is the vector associated to $\partial_i$ such that 
\begin{equation}
\mathbf{x}_{\partial_i} = 
\begin{cases}
\bm{0}& r_i \text{ is odd},\\
(-\mathbf{x}_i'',\mathbf{x}_i'',\cdots,-\mathbf{x}_i'',\mathbf{x}_i'') & r_i \text{ is even}.\\
\end{cases}
\end{equation}
Define $\mathbf{x}\in\mathbb Z^{V_{\lambda}'}$ as $\mathbf{x} = (\mathbf{0},\mathbf{x}_2)$, and set $\mathbf{y} := \mathbf{x}\sfK_{\lambda}$.

Consider the case when $n$ is even.
From the above discussion, we know 
${\bf h}_2={\bf 0}$ in $\mathbb Z_{d^{\ast}}$.
Since $\tff_2 =\bm{0}$ in $\mathbb Z_{d^{\ast}}$, each $h_{i,j}$ is a multiple of $d^{\ast}$ for $1\leq j \leq n-1$. The equation 
$(x_1,x_2,\cdots,x_{n-1})F= \tff'_{i}$ in $\mathbb Z_{m^\ast}$ implies
\begin{equation}\label{equation3}
\begin{cases}
(n-1)x_1 - nx_2 =h_{i,1},\\
(n-2)x_1 - nx_3 =h_{i,2},\\
\;\vdots\\
2x_1 - nx_{n-1} =h_{i,n-2},\\
x_1 =h_{i,n-1},\\
\end{cases}\text{ in $\mathbb Z_{m^\ast}$}.
\end{equation}
Equation \eqref{equation3} is equivalent to the following equations:
\begin{equation}\label{equation4} 
\begin{cases}
\dfrac{2n}{d} x_2 =(n-1)\dfrac{h_{i,n-1}}{d^{\ast}}-\dfrac{h_{i,1}}{d^{\ast}} \text{ in $\mathbb Z_{m}$},\medskip\\
\dfrac{2n}{d}x_3 =(n-2)\dfrac{h_{i,n-1}}{d^{\ast}}-\dfrac{h_{i,2}}{d^{\ast}} \text{ in $\mathbb Z_{m}$},\\
\;\vdots\\
\dfrac{2n}{d}x_{n-1} =2\dfrac{h_{i,n-1}}{d^{\ast}}-\dfrac{h_{i,n-2}}{d^{\ast}} \text{ in $\mathbb Z_{m}$},\\
x_1 = h_{i,n-1} \text{ in $\mathbb Z_{m^\ast}$},\\
\end{cases}
\end{equation}
Since $2n/d$ and $m$ are coprime, Equation \eqref{equation2} has a solution, and so does Equation \eqref{equation}.
Fix a solution $\mathbf{x}_i'\in \mathbb Z^{n-1}$ of Equation \eqref{equation}. Since $2E$ is invertible in $\mathbb Z_{m^{\ast}}$, there exists 
$\mathbf{x}_i''\in \mathbb Z^{n-1}$ such that $2\mathbf{x}_i ''E=\mathbf{x}_i'$ in $\mathbb Z_{m^{\ast}}$. Thus we have 
$2\mathbf{x}_i''G = 2\mathbf{x}_i''EF=\mathbf{x}_i'F=\tff_i'$ in $\mathbb Z_{m^\ast}$. 

For any $n\geq 2$, 
set $\mathbf{x}_2= (\mathbf{x}_{\partial_1},\mathbf{x}_{\partial_2},\cdots,\mathbf{x}_{\partial_b})\in \mathbb Z^{U}$, where $\mathbf{x}_{\partial_i}\in\mathbb Z^{r_i(n-1)}$ is the vector associated to $\partial_i$ such that 
\begin{equation}
\mathbf{x}_{\partial_i} = 
\begin{cases}
\bm{0}& r_i \text{ is odd},\\
(-\mathbf{x}_i'',\mathbf{x}_i'',\cdots,-\mathbf{x}_i'',\mathbf{x}_i'') & r_i \text{ is even}.\\
\end{cases}
\end{equation}
Define $\mathbf{x}\in\mathbb Z^{V_{\lambda}'}$ as $\mathbf{x} = (\mathbf{0},\mathbf{x}_2)$, and set $\mathbf{y} := \mathbf{x}\sfK_{\lambda}$.

We regard $\mathbf{y} = (\mathbf{y}_1,\mathbf{y}_2)\in \mathbb Z^{V_{\lambda}}$ with $\mathbf{y}_1\in \bZ^{\obVlast}$ and $\mathbf{y}_2\in \bZ^{W}$.
By the definitions of $\mathbf{x}$, $\mathbf{y}_1$ and $\mathbf{y}_2$ satisfy Equation \eqref{eq_key}, and 
$$
\mathbf{y}_2-\mathbf{h}_2 = \bm{0} \text{\ in} 
\begin{cases}
    \mathbb Z_{m'} & \text{$n$ is odd,}\\
    \mathbb Z_{m^{\ast}}& \text{$n$ even}.
\end{cases}
$$
Thus $\mathbf{y} - \mathbf{h} = (\mathbf{y}_1 - \mathbf{h}_1,\mathbf{y}_2 - \mathbf{h}_2)$ satisfies Equation \eqref{eq_key}, especially $-2(\mathbf{y}_1 - \mathbf{h}_1)^T+ (\frac{1}{n}D-C_1) (\mathbf{y}_2 - \mathbf{h}_2)^T=\bm{0}\text{ in }\mathbb Z_{m'}$.
Then we have $\mathbf{y}_1 - \mathbf{h}_1 = \bm{0}$ in $\mathbb Z_{m^\ast}$ when $n$ is odd. 
This implies $\mathbf{y} - \mathbf{h} =\bm{0}$ is in $X_{m'}$ when $n$ is odd.
On the other hand, we have ${\bf y}_2 - {\bf h}_2={\bf 0}$ in $\mathbb Z_{m^{\ast}}$ when $n$ is even. Since $\mathbf{y} - \mathbf{h} = (\mathbf{y}_1 - \mathbf{h}_1,\mathbf{y}_2 - \mathbf{h}_2)$ satisfies Equation \eqref{eq_key}, then $\mathbf{y} - \mathbf{h} =\bm{0}$ is in $X_{m'}^{\ast}$ when $n$ is even.
Since both of $\mathbf{y}$ and $\mathbf{h}$ are balanced, then 
$\mathbf{y} - \mathbf{h} = \mathbf{z}\sfK_{\lambda}$ for some 
$\mathbf{z}\in \mathbb Z^{V_{\lambda}'}$. Then we have $\mathbf{z}\in\Gamma_{m'}$ for any $n$.

We have $\mathbf{h} = \mathbf{f}-\mathbf{k}_0 = \mathbf{d}'\sfK_{\lambda}- \mathbf{t}_0\sfK_{\lambda}$, and 
$\mathbf{h} = \mathbf{y}-\mathbf{z}\sfK_{\lambda} = \mathbf{x}\sfK_{\lambda}- \mathbf{z}\sfK_{\lambda}$. We have $\mathbf{t}_0=
\mathbf{d}'-\mathbf{x}+\mathbf{z}$ since $\sfK_{\lambda}$ is invertible.
Then $a^{\mathbf{t}_0} = a^{\mathbf{d}'}a^{-\mathbf{x}}a^{\mathbf{z}}$ is contained in the subalgebra of $\A$ generated by $\text{$\tra(\mathsf{B})$ and $\{a^{\mathbf{k}}\mid \mathbf{k} \in \Gamma_{m'} \}$}$ from the definitions of $\mathbf{d}',\mathbf{x}$ and $\mathbf{z}\in \Gamma_{m'}$.
\end{proof}

\begin{rem}\label{rem-m-double-prime-prime}
    As noted in 
    Remark \ref{rem-center-equation},
    the equation system \eqref{eq-original-key} is equivalent to the equation system \eqref{eq_key}. 
    Then the proof of \cite[Theorem 5.21]{KW24}
    works when $m'$ is odd (we required $m''$ is odd in \cite[Theorem 5.21]{KW24}).
    Actually, all the arguments in \cite{KW24} can be generalized from the restriction $m''$ being odd to the restriction $m'$ being odd.
\end{rem}

\section{The Unicity Theorem and the PI-degree of the stated $\SL(n)$-skein algebra}\label{sec-Unicity-Theorem}
In this section, we suppose that $\cR=\bC$ and Condition~$(\ast)$ as Section~\ref{sub_center} when we consider a root of unity case.
Let $\Si$ be a triangulable pb surface without interior punctures, and let $\lambda$ be a triangulation of $\Si$.

This section contributes to the heart of this paper. 
We compute the rank of $\A$ over its center when $m'$ is even, and every boundary component of $\overline{\Si}$ contains even punctures if $n$ is also even (Theorem~\ref{thm:rank}). 
It was shown in \cite{KW24} that this rank equals the rank of $\cS_n(\Si)$ over its center, which furthermore equals the square of the highest dimension among its irreducible representations.
We use some techniques in \cite{KW24}, which deals with the case when $m'$ is odd. However, we prove many novel results to handle the even case, such as Lemmas~\ref{lem-image-J}, \ref{lem-exact-L-J},  Proposition~\ref{prop5.4}, and etc.
These results are stated for the general case, which could be used to solve the remaining case
(i.e., both $m'$ and $n$ are even, and $\overline{\Si}$ contains at least one boundary component containing odd number of punctures).

\subsection{Almost Azumaya algebras and Unicity theorem}\label{sub-Almost-Azumaya-algebra}
In this subsection, we recall some basics on representations of (almost) Azumaya algebras following \cite{FKBL19, BG02}. 
\begin{dfn}\label{def:almost_Azumaya}
A $\bC$-algebra $A$ is \textbf{almost Azumaya} if $A$ satisfies the following conditions; 
\begin{enumerate}
    \item $A$ is finitely generated as a $\bC$-algebra, 
    \item $A$ has no zero-divisors, 
    \item $A$ is finitely generated as a module over its center. 
\end{enumerate}
\end{dfn}
While the above conditions are sufficient conditions to be almost Azumaya, let us use these conditions as a definition, which fits well with skein algebras.
Note that the original second condition is that $A$ is prime, which is a weaker condition than $A$ has no zero-divisors. In the paper, we use the different definition suitable for skein algebras.

Suppose $A$ is  almost Azumaya. We use $\text{Frac}( {\mathcal Z(A)})$ to denote the fractional field of $\mathcal Z(A)$, and use 
$\widetilde A$ to denote $A\otimes_{\cZ(A)} \Frac(\cZ(A))$.
Define $\rankZ A$ to be the dimension of $\widetilde{A}$ over $\Frac(\cZ(A))$.

\begin{rem}
If $A$ is free over its center $\cZ(A)$ then $\rankZ A$ is equal to the usual rank. 
\end{rem}

For an almost Azumaya algebra $A$, let $V(A)$ denote the maximal spectrum $\Specm(\cZ(A))$. 
Wedderburn’s theorem implies that $\widetilde{A}$ is a division algebra, and has dimension $D^2$ over $\Frac(\cZ(A))$ with a positive integer $D$, called the \textbf{PI-degree} of $A$.
A point $\mathfrak{m} \in  V(A)$ is \textbf{Azumaya} if $A/\mathfrak{m}A \cong M_D(\bC)$, the algebra of $D \times D$-matrices over $\bC$. Let $\Azumaya(A)\subset V(A)$ denote the subset of all the Azumaya points, called the \textbf{Azumaya locus}. 
It is known that $D^2 = \rankZ A$. The Unicity theorem in \cite{FKBL19} claims that the Azumaya locus is a Zariski open dense subset of $V(A)$, and 
any point of the Azumaya locus gives a unique finite dimensional irreducible representation (up to equivalence) of $A$ with the highest dimension $D$ among all the finite dimensional irreducible representations.

\subsection{Dimension of quantum $A$-torus over its center}\label{subsec:dim_torus_stated}
In the rest of this section, we suppose that $\Sigma$ is a connected triangulable essentially bordered pb surface without interior punctures with a triangulation $\lambda$ of $\Sigma$. Suppose $\overline{\Sigma}$ has connected boundary components $\partial_1,\cdots,\partial_b$, and each $\partial_i$ contains $r_i$ punctures.

Recall that we defined a 
subgroup $\Lambda_\partial\subset \bZ^{V_{\lambda}'}$ in 
\eqref{eq-balanced-boundary-A-version}.

Define $\Lambda_z$ to be the subgroup of $\mathbb Z^{V_{\lambda}'}$ generated by  $\Lambda_{\partial}$ and $\Gamma_{m'}$, defined as \eqref{eq:Lambda}, when $m^\ast$ is odd. Define $\Lambda_z:=\Gamma_{m'}$ when $m^\ast$ is even.

\begin{lem}[{Lemma 6.5 in ~\cite{KW24}}]\label{lem5.1}
Suppose we have Condition $(\ast)$, then we have the following:\\
(a) The center of $\A$ is $\mathbb C\text{-span}\{a^{\mathbf{k}}\mid \mathbf{k}\in \Lambda_z\}$.\\
(b) We have $\rankZ\A=\left|\dfrac{\mathbb Z^{V_{\lambda}'}}{\Lambda_z}\right|$. 
\end{lem}

\begin{prop}[{\cite[Proposition 6.6]{KW24}}]\label{rank_eq}
Suppose $\hat{q}$ is a root of unity.
Let $\Sigma$ be a triangulable pb surface without interior punctures, and let $\lambda$ be a triangulation of $\Sigma$. Then
we have $\rankZ\A = \rankZ\cS_n(\Sigma)$.
\end{prop}

For any abelian group $G$,
we take the CW cochain complex of $\overline\Sigma$ over $G$
$$0\rightarrow C^0(\overline \Sigma,G)\xrightarrow{\delta_0} C^1(\overline \Sigma,G)\xrightarrow{\delta_1} C^2(\overline \Sigma,G)\rightarrow 0.$$
Let $Z^{i}(\overline\Sigma,G),\ B^{i}(\overline\Sigma,G)$, and $H^{i}(\overline\Sigma,G)$ denote the $i$-th cocycle, the $i$-th coboundary, and the $i$-th cohomology group of $\overline\Sigma$ with coefficient $G$ respectively.
For any positive integer $l$, define
 \begin{align}\label{def-eq-Zl-Cl}
     Z^1(\overline \Sigma,\mathbb Z_n)_l=l(C^1(\overline\Sigma,\mathbb Z_n))\cap Z^1(\overline \Sigma,\mathbb Z_n).
 \end{align}

Recall $d=\gcd(m',2n) = 2\gcd(m^{\ast},n)$ is even and $d^\ast = d/2$, defined in Section~\ref{notation}, since $m'$ is even.

\begin{lem}[{\cite[Proposition 6.7]{KW24}}]\label{lem5.3}
 Set $N = nm^{\ast}/d^{\ast}$.
   Then there is a short exact sequence 
$$0\rightarrow N\mathbb Z^{V_{\lambda}}\xrightarrow{L}\Lambda_{\lambda}\cap m^{\ast}\mathbb Z^{V_{\lambda}} \xrightarrow{J} Z^1(\overline \Sigma,\mathbb Z_n)_{d^\ast} \rightarrow 0,$$
where $L$ is the natural embedding and $\Lambda_\lambda$ is the balanced part.
\end{lem}

The triangulation $\lambda$ gives a cell decomposition of $\overline\Sigma$. We orient the $1$-simplices so that the orientations of the $1$-simplices contained in $\partial \overline \Sigma$ match with that of $\overline\Sigma$. The orientation of $\overline\Sigma$ gives those of the $2$-simplices. 
For every $1$-simplex $e$, we label the vertices of the $n$-triangulation in $e$ as $v_1^{e},v_2^{e},\cdots,v_{n-1}^{e}$  consecutively using the orientation of $e$ such that the orientation of $e$ is given from $v_1^{e}$ to $v_{n-1}^e$. 

Let us review the definition of $J$ given in Section 6.3 of \cite{KW24}. 
For any $\mathbf{k}\in \Lambda_{\lambda}\cap 
m^\ast\mathbb Z^{V_{\lambda}}$, there exists $s^{\mathbf{k}}_e\in\mathbb Z$ such that
$s^{\mathbf{k}}_e(1,2,\cdots,n-1) = (\mathbf{k}(v_1^e),\cdots, \mathbf{k}(v_{n-1}^e)) \in \mathbb Z_n$ (note that $s^{\mathbf{k}}_e$ is unique as an element in $\mathbb Z_n$).
Let $s^{\mathbf{k}} \in C^1(\overline\Sigma,\mathbb Z_n)$ be an element such that every 1-simplex $e$ is assigned with $s^{\mathbf{k}}_e\in d^\ast\mathbb Z_n$.
 Then define
$$J\colon\Lambda_{\lambda}\cap m^{\ast}\mathbb Z^{V_{\lambda}} \rightarrow Z^1(\overline \Sigma,\mathbb Z_n)_{d^\ast},\quad \mathbf{k}\mapsto s^{\mathbf{k}}.$$

Note that $X_{m'}$ is a subgroup of 
$\Lambda_{\lambda}\cap m^{\ast}\mathbb Z^{V_{\lambda}}$, where $X_{m'}$ was defined as \eqref{def-X-m-prime}.
We use $J'$ to denote the restriction of $J$ to $X_{m'}$.
Define
$$\Omega_{m'}=\{{\bf k}=({\bf k}_1,{\bf k}_2)\in m^\ast\mathbb Z^{V_\lambda}\mid {\bf k}_2={\bf 0}\text{ in }\mathbb Z_{m'}\},$$
where $\tfk_1\in \bZ^{\obVlast}$
and ${\bf k}_2\in \bZ^{W}$.
It is easy to see that 
$X_{m'} = \Omega_{m'}\cap m^{\ast}\mathbb Z^{V_\lambda}\cap \Lambda_\lambda.$
From $N\mathbb Z^{V_\lambda} \subset \Lambda_\lambda \cap m^{\ast}\mathbb Z^{V_\lambda}\ (N = nm^{\ast}/d^\ast)$ and $\Omega_{m'}\subset m^{\ast}\mathbb Z^{V_\lambda}$, we have $$\ker J' = N\mathbb Z^{V_\lambda}  \cap \Omega_{m'}\cap \Lambda_\lambda \cap m^{\ast}\mathbb Z^{V_\lambda} = N\mathbb Z^{V_\lambda}  \cap \Omega_{m'}.$$

Set
$$C^1_{\partial,d}(\overline{\Si},\mathbb Z_n)=\{s\in C^1(\overline{\Si},\mathbb Z_n)\mid s(e)\text{ is a multiple of $d$ in $\mathbb{Z}_{n}$ for any $e\in \lambda$ such that $e\subset \partial\overline{\Si}$}\}.$$

Recall that $n'=\frac{2n}{d}$
and $m=\frac{m'}{d}$, see Section \ref{notation}.

\begin{lem}\label{lem-image-J}
    We have 
    $$\im J'=\begin{cases}
        \im J & n'\text{ is odd},\\
        Z^1(\overline\Si,\mathbb Z_n)_{d^\ast}
        \cap C^1_{\partial,d}(\overline{\Si},\mathbb Z_n) & n' \text{ is even},
    \end{cases}$$
    where $Z^1(\overline\Si,\mathbb Z_n)_{d^\ast}$ is defined in \eqref{def-eq-Zl-Cl}.
\end{lem}

Note that $n'$ is even if and only if $n$ is a multiple of $d$.

\begin{proof}

Obviously, we have $\im J'\subset \im J= Z^1(\overline\Si,\mathbb Z_n)_{d^\ast}$ from Lemma~\ref{lem5.3}.
Then we shall show that 
$\im J'\subset Z^1(\overline\Si,\mathbb Z_n)_{d^\ast}
        \cap C^1_{\partial,d}(\overline{\Si},\mathbb Z_n)$ when $n$ is a multiple of $d$.
    For a $1$-simplex $e\subset\partial\overline{\Si}$, let $\tau$ be the triangle attached to $e$ as shown in Figure \ref{Fig;coord_uvw}.
    Since the orientation of $e$ is given by the orientation of $\overline{\Si}$, then 
    $v_i^e=v_i$ for $1\leq i\leq n-1$. 
    For any ${\bf k}\in X_{m'}$, we have 
    ${\bf k}|_{\tau} = \lambda_1{\bf pr}_1
    +\lambda_2{\bf pr}_2$ in $\mathbb Z_n$ since ${\bf pr}_1+{\bf pr}_2+{\bf pr}_3=\mathbf{0}$ in $\bZ_{n}$.
    Then ${\bf k}(u_i) = \lambda_1i\in\mathbb Z_n$.
    Since ${\bf k}\in \mathbb Z^{V_\lambda}$, then we have ${\bf k}(u_i)=0$. When we take the zero-extension of ${\bf k}\in \mathbb Z^{V_\lambda}$ to $\mathbb{Z}^{\overline{V}_{\lambda^\ast}}$, we have ${\bf k}(u_i)=0$.
    This implies that 
    $\lambda_1=0\text{ in }\mathbb Z_n$ and ${\bf k}|_{\tau} =\lambda_2{\bf pr}_2$ in $\mathbb Z_n$.
    Then we have $({\bf k}(v_1),\cdots,{\bf k}(v_{n-1}))
    =\lambda_2(1,\cdots,n-1)$ in $\mathbb Z_n$ and $s^{{\bf k}}_e = \lambda_2$.
    Since $n$ is a multiple of $d$, then we have  ${\bf k}|_{\tau} =\lambda_2{\bf pr}_2$ in $\mathbb Z_d$.
    We also have ${\bf k}|_{\tau\cap W} = {\bf 0}$ in $\mathbb Z_{m'}$ since ${\bf k}\in X_{m'}$. These show that ${\bf 0}={\bf k}|_{\tau\cap W} =\lambda_2{\bf pr}_2|_{\tau\cap W}$ in $\mathbb Z_d$, i.e., $d|\lambda_2$.
    Hence, $J'({\bf k})=s^{{\bf k}}\in  C^1_{\partial,d}(\overline{\Si},\mathbb Z_n)$. Hence, we have 
    $\im J'\subset Z^1(\overline\Si,\mathbb Z_n)_{d^\ast}
        \cap C^1_{\partial,d}(\overline{\Si},\mathbb Z_n)$.

Note that every element in $C^1(\overline\Sigma,\mathbb Z_n)$ is represented by a map from the set of all the 1-simplices to $\mathbb Z_n$. Suppose $c\in Z^1(\overline\Sigma,\mathbb Z_n)_{d^\ast}$. For any 1-simplex $e$, we choose $t_e\in \mathbb Z$ such that $c(e) = t_e\in\mathbb Z_n$. Since $c\in d^\ast C^1(\overline\Sigma,\mathbb Z_n)$, we have $t_e$ is a multiple of $d^\ast$ for each 1-simplex $e$. For each 2-simplex $\tau$, suppose $\tau$, $e_1,\;e_2,\;e_3$ look like in the left picture in Figure \ref{Fig;coord_ijk}. Assume that the orientation of $e_2$ is the one induced from $\tau$ and the orientations of $e_1$ and $e_3$ are the ones opposite to the orientations induced from $\tau$. 
    Since $c\in Z^1(\overline\Sigma,\mathbb Z_n)$, we have $-t_{e_1}+t_{e_2}-t_{e_3} =0\text{ in } \mathbb Z_n$. 
    Set $y_1 = -t_{e_1}$, $y_2 = 0$, $y_3 = -t_{e_2}$. 
    Then we have the following equations in $\mathbb Z_n$;
    \begin{align}\label{eqy}
        y_2-y_1 = t_{e_1},\quad y_2-y_3 = t_{e_2},\quad y_1-y_3=t_{e_3}.
    \end{align}
Note that each $y_i$ is a multiple of $d^\ast$. Since $\gcd(2n/d,m'/d) = 1$, then equation $\frac{2y_1}{d}\mathbf{pr}_1+\frac{2y_2}{d}\mathbf{pr}_2+\frac{2y_3}{d}\mathbf{pr}_3+\frac{2n}{d}\mathbf{x} = \mathbf{0}$ 
has a unique solution $\mathbf{x}=\mathbf{a}_{\tau}\in\mathbb Z^{\overline V_{\tau}}$ in $\mathbb Z_{\frac{m'}{d}}$.  
Then we have  $y_1\mathbf{pr}_1+y_2\mathbf{pr}_2+y_3\mathbf{pr}_3+n\mathbf{a}_{\tau} = \mathbf{0}$ in $\mathbb Z_{m^{\ast}}$. For each 2-simplex $\tau$, define 
    $\mathbf{k}_{\tau} = y_1\mathbf{pr}_1+y_2\mathbf{pr}_2+y_3\mathbf{pr}_3+n\mathbf{a}_{\tau}\in m^\ast\mathbb Z^{\overline{V}_{\tau}}$. Then Equation \eqref{eqy} implies we have the following equations in $\mathbb Z_n$;
\begin{align}\label{eq_edge}
    \mathbf{k}_{\tau}(v_i^{e_j}) = (-1)^j(y_j-y_{j+1})i=c(e_j)i\quad (j=1,2,3),
\end{align}
where the indices are $\bZ_3$-cyclic. 

For each attached triangle $\tau$, we label some small vertices in $\tau$ as Figure \ref{Fig;coord_uvw}. 
Let us recall the coordinates for these vertices
$$u_i=(i,0,n-i),\quad v_i=(n-i,i,0),\quad w_i=(0,i,n-i).$$
Suppose $\tau$ is an attached triangle and $e$ is the attaching edge of $\tau$.

When $n$ is not a multiple of $d$, then $n'=\frac{2n}{d}$ is odd. We have $\gcd(n,m')=\gcd(n,m^{*})=d^\ast$ and 
$\gcd(\frac{2n}{d},2m)=1$. 
Then there exist $r,s\in\mathbb Z$ such that $r\frac{2n}{d}+2sm = 1$.
Define $\mathbf{f} = -r\frac{2t_e}{d}\mathbf{pr}_2$. 
Then 
$\frac{2t_e}{d}\mathbf{pr}_2+\frac{2n}{d} \mathbf{f} = \mathbf{0}$ in $\mathbb{Z}_{2m}$. Set $\mathbf{k}_{\tau} = t_e\mathbf{pr}_2+ n\mathbf{f}$. Then $\mathbf{k}_{\tau} =\bm{0}$ in $\mathbb Z_{m'}$,  $\mathbf{k}_{\tau}(u_i) = 0$ and $\mathbf{k}_{\tau}(v_i) = \mathbf{k}_{\tau}(v_i^e) = t_e i = c(e) i$ for 
    $1\leq i\leq n-1$.

When $n$ is a multiple of $d$, then $m$ is odd.
 We have $\gcd(\frac{4n}{d},m)=1$. 
Then there exist $r',s'\in\mathbb Z$ such that $r'\frac{4n}{d}+s'm = 1$.
Define $\mathbf{f} = -r'\frac{2t_e}{d}\mathbf{pr}_2$. 
Then 
$\frac{2t_e}{d}\mathbf{pr}_2+\frac{4n}{d} \mathbf{f} = \mathbf{0}$ in $\mathbb{Z}_{m}$.
Since $\frac{2t_e}{d}\mathbf{pr}_2+\frac{4n}{d} \mathbf{f} = \mathbf{0}$ in $\mathbb{Z}_{2}$ and $m$ is odd, then we have $\frac{2t_e}{d}\mathbf{pr}_2+\frac{4n}{d} \mathbf{f} = \mathbf{0}$ in $\mathbb{Z}_{2m}$.
Set $\mathbf{k}_{\tau} = t_e\mathbf{pr}_2+ 2n\mathbf{f}$. Then $\mathbf{k}_{\tau} =\bm{0}$ in $\mathbb Z_{m'}$,  $\mathbf{k}_{\tau}(u_i) = 0$ and $\mathbf{k}_{\tau}(v_i) = \mathbf{k}_{\tau}(v_i^e) = t_e i = c(e) i$ for 
    $1\leq i\leq n-1$.

    We will construct $\mathbf{h}\in \mathbb Z^{\overline{V}_{\lambda}}$ whose restriction on each triangle $\tau$ is $\mathbf{k}_{\tau}$. For each 1-simplex $e$, suppose the triangle $\tau$ (resp. $\tau'$) is on your right (resp. left) if you walk along $e$ (we consider the triangulation $\lambda^\ast$ of $\Sigma^\ast$ when $e$ is contained in $\partial\overline\Sigma$).
    Note that $\tau'$ can not be an attached triangle.
   Equation \eqref{eq_edge} implies 
    $$(\mathbf{k}_{\tau}(v_1^{e}),\cdots,\mathbf{k}_{\tau}(v_{n-1}^{e}))=(\mathbf{k}_{\tau'}(v_1^{e}),\cdots,\mathbf{k}_{\tau'}(v_{n-1}^{e}))\text{ in }\mathbb Z_n.$$ We also have 
    $$(\mathbf{k}_{\tau}(v_1^{e}),\cdots,\mathbf{k}_{\tau}(v_{n-1}^{e}))=(\mathbf{k}_{\tau'}(v_1^{e}),\cdots,\mathbf{k}_{\tau'}(v_{n-1}^{e}))=\bm{0}\text{ in }\mathbb Z_{m^\ast}.$$
    Then we have 
    $$(\mathbf{k}_{\tau}(v_1^{e}),\cdots,\mathbf{k}_{\tau}(v_{n-1}^{e}))=(\mathbf{k}_{\tau'}(v_1^{e}),\cdots,\mathbf{k}_{\tau'}(v_{n-1}^{e}))\text{ in }\mathbb Z_N.$$
    Thus there exists $\mathbf{b}_e\in N\mathbb Z^{n-1}$ such that 
    $$(\mathbf{k}_{\tau}(v_1^{e}),\cdots,\mathbf{k}_{\tau}(v_{n-1}^{e}))=\mathbf{b}_e+(\mathbf{k}_{\tau'}(v_1^{e}),\cdots,\mathbf{k}_{\tau'}(v_{n-1}^{e})).$$
    Define $\mathbf{b}_{\tau'}\in \mathbb Z^{\overline{V}_{\tau'}}$ as $$\mathbf{b}_{\tau'}(v) = \begin{cases}
        \mathbf{b}_e(v_i^e)& v= v_i^e\\
        \bm{0} &\text{otherwise}.
    \end{cases}$$
    Then we define a new $\mathbf{k}_{\tau'}$ as $\mathbf{k}_{\tau'}+\mathbf{b}_{\tau'}$.
    We apply the same procedure consecutively to every  1-simplex. Then for each  1-simplex $e$, we have $\mathbf{k}_{\tau}$ and $\mathbf{k}_{\tau'}$ agree with each on $e$, where $\tau$ and $\tau'$ are the two triangles containing $e$. 
    Thus there exists $\mathbf{h}\in \mathbb Z^{\overline{V}_{\lambda}}$ such that $\mathbf{h}|_{\tau} = \mathbf{k}_{\tau}$ for every triangle $\tau$ of $\lambda^{\ast}$ in $\mathbb Z_{m^\ast}$, especially $\mathbf{h}|_{\tau} = \mathbf{k}_{\tau}$ in $\mathbb Z_{m^\ast}$ if $\tau$ is an attached triangle. 
    From the construction of $\mathbf{h}$, we have $\mathbf{h}\in \Omega_{m'}\cap m^{\ast}\mathbb Z^{V_\lambda}\cap \Lambda_\lambda$ and $J'(\mathbf{h}) = c$. 
    This completes the proof. 
\end{proof}

By taking the intersection with $\Omega_{m'}$, we have the following. 
\begin{lem}\label{lem-exact-L-J}
    We have the short exact sequence
    $$0\rightarrow N\mathbb Z^{V_{\lambda}}\cap \Omega_{m'}\rightarrow X_{m'} \xrightarrow{J'} \im J'\rightarrow 0,$$
    where $\im J'$ is given in Lemma \ref{lem-image-J}.
\end{lem}

\begin{lem}\label{lem-iso-ZC-ZCp}
Suppose $n$ is a multiple of $d$. Then, we have $$Z^1(\overline{\Si},\mathbb{Z}_n)_{d^\ast}\cap C^1_{\partial,d}(\overline{\Si},\bZ_n)\cong Z^1(\overline{\Si},\mathbb{Z}_{n'})\cap C^1_{\partial,2}(\overline{\Si},\bZ_{n'}).$$
\end{lem}
\begin{proof}
Consider the map $$\varphi\colon Z^1(\overline{\Si},\mathbb{Z}_n)_{d^\ast}\cap C^1_{\partial,d}(\overline{\Si},\bZ_n) \to Z^1(\overline{\Si},\mathbb{Z}_{n'})\cap C^1_{\partial,2}(\overline{\Si},\bZ_{n'}),\quad s\mapsto \frac{s}{d^\ast}.$$ 
Obviously it is a well-defined group homomorphism. 
We also consider the group homomorphism
$$\psi\colon
Z^1(\overline{\Si},\mathbb{Z}_{n'})\cap C^1_{\partial,2}(\overline{\Si},\bZ_{n'})\to 
Z^1(\overline{\Si},\mathbb{Z}_n)_{d^\ast}\cap C^1_{\partial,d}(\overline{\Si},\bZ_n),\quad s\mapsto \frac{s}{d^\ast}.$$
Then both of $\psi\circ \varphi$ and $\varphi \circ \psi$ are the identity maps. 
Hence, $\varphi$ is an isomorphism. 
\end{proof}

When $n'$ is a multiple of $2$ (equivalently $n$ is a multiple of $d$), there is a projection 
$g\colon \mathbb Z_{n'}\rightarrow \mathbb Z_2$. Then $g$ induces a group homomorphism $g_{*}\colon Z^1(\overline\Sigma,\mathbb Z_{n'})\rightarrow Z^1(\overline\Sigma,\mathbb Z_2)$.

\begin{lem}\cite[A special case of Lemma 6.9. by replacing $n$ and $d$ with $n'$ and $2$, respectively]{KW24}\label{lem_short}
Suppose $n$ is a multiple of $d$. 
Then, we have the following short exact sequence
    \begin{align*}
    0\rightarrow Z^1(\overline\Sigma,\mathbb Z_{n'})_2\xrightarrow{L} Z^1(\overline\Sigma,\mathbb Z_{n'})\xrightarrow{g_*} Z^1(\overline\Sigma,\mathbb Z_2)\rightarrow 0,
\end{align*}
where $L$ is the natural embedding. 
\end{lem}

It is easy to check that the group homomorphism
$g_{*}$ in Lemma \ref{lem_short} restricts to the group homomorphism 
$$g'\colon
Z^1(\overline \Si,\mathbb Z_{n'})\cap C^1_{\partial ,2}(\overline{\Si},\mathbb Z_{n'}) \rightarrow Z^1(\overline{\Si},\mathbb Z_2)\cap 
    C^1_{\partial,2}(\overline{\Si},\mathbb Z_2).$$

\begin{lem}\label{lem-ism-ab-short}
(a)   
When $n$ is a multiple of $d$, there is a short exact sequence
    $$0\rightarrow Z^1(\overline{\Si},\mathbb Z_{n'})_2\xrightarrow{L}Z^1(\overline \Si,\mathbb Z_{n'})\cap C^1_{\partial ,2}(\overline{\Si},\mathbb Z_{n'}) \xrightarrow{g'}Z^1(\overline{\Si},\mathbb Z_2)\cap 
    C^1_{\partial,2}(\overline{\Si},\mathbb Z_2)\rightarrow 0.$$

(b) We have $Z^1(\overline{\Si},\mathbb Z_2)\cap 
    C^1_{\partial,2}(\overline{\Si},\mathbb Z_2)\simeq H_1(\overline{\Si},\mathbb Z_2)$.
\end{lem}
\begin{proof}
        (a) For any element $s\in Z^1(\overline{\Si},\mathbb Z_2)\cap 
    C^1_{\partial,2}(\overline{\Si},\mathbb Z_2)$, there exists $t\in Z^1(\overline \Si,\mathbb Z_{n'})\cap C^1_{\partial ,2}(\overline{\Si},\mathbb Z_{n'})$ such that $g_*(t)=s$ since $g_*$ is surjective. Then $s\in C^1_{\partial,2}(\overline{\Si},\mathbb Z_2)$ implies $t\in C^1_{\partial ,2}(\overline{\Si},\mathbb Z_{n'})$. This shows $g'(t) = s$ and $g'$ is surjective. 
    Lemma \ref{lem_short} implies that $\ker g_* = Z^1(\overline{\Si},\mathbb Z_{n'})_2=Z^1(\overline{\Si},\mathbb Z_{n'})\cap 2 C^1(\overline{\Si},\mathbb Z_{n'})$.
    We have 
    $$\ker g' = \ker g_*\cap Z^1(\overline \Si,\mathbb Z_{n'})\cap C^1_{\partial ,2}(\overline{\Si},\mathbb Z_{n'}) =\ker g_*$$
    since $2 C^1(\overline{\Si},\mathbb Z_{n'})\subset C^1_{\partial ,2}(\overline{\Si},\mathbb Z_{n'})$.

    (b) Any element $\gamma\in H_1(\overline\Si,\mathbb Z_2)$ can be realized by a curve on $\overline{\Sigma}$. Then $\gamma$ gives a map
    $s_{\gamma}$ from $\lambda$ to $\mathbb Z_2$ using the intersection number between $\gamma$ and the edges in $\lambda$.
    It is easy to check that the map $H_1(\overline\Si,\mathbb Z_2)\rightarrow Z^1(\overline{\Si},\mathbb Z_2)\cap 
    C^1_{\partial,2}(\overline{\Si},\mathbb Z_2)$ defined by $\gamma\mapsto s_{\gamma}$ is a well-defined group homomorphism. 

    For any element in $Z^1(\overline{\Si},\mathbb Z_2)\cap 
    C^1_{\partial,2}(\overline{\Si},\mathbb Z_2)$, we regard it as a geometric intersection number for some curve on $\overline{\Sigma}$. In particular, the curve is closed. Hence, this gives a well defined group homomorphism  $Z^1(\overline{\Si},\mathbb Z_2)\cap 
    C^1_{\partial,2}(\overline{\Si},\mathbb Z_2)\to H_1(\overline\Si,\mathbb Z_2)$. 
\end{proof}

Recall 
\begin{align}\label{eq-def-r-sigma}
    r(\Sigma) := \# (\partial \Sigma) - \chi(\Sigma)
\end{align}
defined in Lemma~\ref{lem:cardinarity}, where $\chi(\Sigma)$ denotes the Euler characteristic of $\Sigma$.

\begin{lem}\label{lem:number_ZCp}
    When $n$ is a multiple of $d$, we have   $$|Z^1(\overline{\Si},\mathbb{Z}_n)_{d^\ast}\cap C^1_{\partial,d}(\overline{\Si},\bZ_n)| = \Big(\frac{n'}{2}\Big)^{r(\Sigma)} 2^{2g+b-1},$$
    where $b$ is the number of boundary components of $\overline\Si$.
\end{lem}
\begin{proof}
    From Lemmas \ref{lem-iso-ZC-ZCp} and \ref{lem-ism-ab-short}, we have    $$|Z^1(\overline{\Si},\mathbb{Z}_n)_{d^\ast}\cap C^1_{\partial,d}(\overline{\Si},\bZ_n)| = |Z^1(\overline{\Si},\mathbb Z_{n'})_2|
    |H_1(\overline\Si,\mathbb Z_2)|.$$

    Lemma \ref{lem_short} implies that 
    $$|Z^1(\overline{\Si},\mathbb Z_{n'})_2|=\left|\dfrac{Z^1(\overline{\Si},\mathbb Z_{n'})}{Z^1(\overline{\Si},\mathbb Z_{2})}\right| =\Big(\frac{n'}{2}\Big)^{r(\Sigma)},$$
    where the last equality comes from a similar formula to Equation (72)  in \cite{Yu23}.
    It is well-known that $|H_1(\overline\Si,\mathbb Z_2)|=2^{2g+b-1}.$
\end{proof}
\def\ev{\text{od}}

\begin{prop}\label{prop5.4}
    We have $$\displaystyle \left|\frac{\Lambda_{\lambda}}{X_{m'}}\right|=
\begin{cases}
2^{|W|-r(\Sigma)}m^{|V_\lambda|} d^{r(\Sigma)}
    & n' \text{ is odd},\medskip\\ 
    2^{-2g-b+1} m^{|V_\lambda|}d^{r(\Sigma)}
    & n' \text{ is even},
    \end{cases}$$
    where $b$ is the number of boundary components of $\overline\Si$ and $r(\Si)$ is defined in \eqref{eq-def-r-sigma}.
\end{prop}
\begin{proof}
Recall that
$$\Omega_{m'}=\{{\bf k}=({\bf k}_1,{\bf k}_2)\in m^\ast\mathbb Z^{V_\lambda}\mid {\bf k}_2={\bf 0}\text{ in }\mathbb Z_{m'}\},$$
where $\tfk_1\in \bZ^{\obVlast}$
and ${\bf k}_2\in \bZ^{W}$.
We have 
\begin{align}
    \left|\frac{\Lambda_\lambda}{N\mathbb Z^{V_\lambda}\cap \Omega_{m'}}\right| = \left|\frac{\Lambda_\lambda}{X_{m'}}\right|
\left|\frac{X_{m'}}{N\mathbb Z^{V_\lambda}\cap \Omega_{m'}}\right|.
\end{align}
From Lemma \ref{lem-exact-L-J}, we have 
\begin{align}\label{eq-cr-odd-1}
    \left|\frac{\Lambda_{\lambda}}{X_{m'}}\right|=
    |\im J'|^{-1}\left|\frac{\Lambda_\lambda}{N\mathbb Z^{V_\lambda}\cap \Omega_{m'}}\right|.
\end{align}
Also, we have 
\begin{align}\label{eq-cr-odd-2}
\left|\frac{\Lambda_\lambda}{N\mathbb Z^{V_\lambda}\cap \Omega_{m'}}\right|=\left|\frac{\Lambda_\lambda}{n\mathbb Z^{V_\lambda}\cap \Omega_{2}}\right|
\left|\frac{n\mathbb Z^{V_\lambda}\cap \Omega_{2}}{N\mathbb Z^{V_\lambda}\cap \Omega_{m'}}\right|.
\end{align}

Note that 
$$\left|\frac{\Lambda_\lambda}{n\mathbb Z^{V_\lambda}\cap \Omega_{2}}\right| = 
\left|\frac{\Lambda_\lambda}{n\mathbb Z^{V_\lambda}}\right|\left|\frac{n\mathbb Z^{V_\lambda}}{n\mathbb Z^{V_\lambda}\cap \Omega_{2}}\right|.$$
Lemma \ref{lem5.3} implies that 
$\left|\dfrac{\Lambda_\lambda}{n\mathbb Z^{V_\lambda}}\right|=
|Z^{1}(\overline{\Si},\mathbb Z_n)|$ by setting $m'=2$, i.e., $m^\ast=1$.
It is easy to see that 
$$
n\mathbb Z^{V_\lambda}\cap \Omega_{2}=\begin{cases}
   (n\mathbb Z)^{\oplus|\mathring{V}_\lambda|}\oplus (2n\mathbb Z)^{\oplus|W|}
    & n \text{ is odd},\\
     n\mathbb Z^{V_\lambda} & n \text{ is even}.
\end{cases}
$$
Then we have 
$$
\left|\frac{n\mathbb Z^{V_\lambda}}{n\mathbb Z^{V_\lambda}\cap \Omega_{2}}\right|=\begin{cases}
    2^{|W|} & n \text{ is odd}, \\
    0 & n \text{ is even}. 
\end{cases}
$$
We have 
$$
N\mathbb Z^{V_\lambda}\cap \Omega_{m'}
=\begin{cases}
    (N\mathbb Z)^{\oplus|\mathring{V}_\lambda|}\oplus (2N\mathbb Z)^{\oplus|W|} & n \text{ is not a multiple of $d$}, \\
    N\mathbb Z^{V_\lambda} &  n\text{ is a multiple of $d$}. 
\end{cases}
$$
Then we have 
$$
\left|\frac{N\mathbb Z^{V_\lambda}}{N\mathbb Z^{V_\lambda}\cap \Omega_{m'}}\right|=\begin{cases}
    2^{|W|} & n \text{ is not a multiple of $d$}, \\
    0 &  n\text{ is a multiple of $d$}. 
\end{cases}
$$

When $n$ is even, we have 
\begin{align*}
    &\left|\frac{n\mathbb Z^{V_\lambda}\cap \Omega_{2}}{N\mathbb Z^{V_\lambda}\cap \Omega_{m'}}\right|
=  \left|\dfrac{n\mathbb Z^{V_\lambda}}{N\mathbb Z^{V_\lambda}}\right|
  \left|\dfrac{N\mathbb Z^{V_\lambda}}{N\mathbb Z^{V_\lambda}\cap \Omega_{m'}}\right|
  \\=&m^{|V_\lambda|} \left|\dfrac{N\mathbb Z^{V_\lambda}}{N\mathbb Z^{V_\lambda}\cap \Omega_{m'}}\right|
  =
  \begin{cases}
    2^{|W|}m^{|V_\lambda|} & n \text{ is not a multiple of $d$}, \medskip\\
m^{|V_\lambda|} &  n\text{ is a multiple of $d$}. 
\end{cases}
\end{align*}
When $n$ is odd, we also have 
$$\left|\frac{n\mathbb Z^{V_\lambda}\cap \Omega_{2}}{N\mathbb Z^{V_\lambda}\cap \Omega_{m'}}\right|
=
\left|\frac{n\mathbb Z^{\oplus |\mathring{V}_\lambda|}\oplus (2n \bZ)^{\oplus|W|}}{(N\mathbb Z)^{\oplus|\mathring{V}_\lambda|}\oplus (2N \bZ)^{\oplus|W|}}\right|=\Big(\frac{N}{n}\Big)^{|V_\lambda|}=m^{|V_\lambda|}.$$

From \eqref{eq-cr-odd-2}, we have
\begin{align}\label{eq-cr-odd-3}
\left|\frac{\Lambda_\lambda}{N\mathbb Z^{V_\lambda}\cap \Omega_{m'}}\right|=\begin{cases}
    2^{|W|} m^{|V_\lambda|}|Z^{1}(\overline{\Si},\mathbb Z_n)| & \text{$n$ is not a multiple of $d$},\\
    m^{|V_\lambda|}|Z^{1}(\overline{\Si},\mathbb Z_n)| & \text{$n$ is a multiple of $d$}.
    \end{cases}
\end{align}

Equations \eqref{eq-cr-odd-1} and \eqref{eq-cr-odd-3}
imply that 
\begin{align*}
    \left|\frac{\Lambda_{\lambda}}{X_{m'}}\right|=
\begin{cases}
2^{|W|} m^{|V_\lambda|}\dfrac{|Z^{1}(\overline{\Si},\mathbb Z_n)|}{|\im J'|}
    & n \text{ is not a multiple of $d$},\medskip\\
    m^{|V_\lambda|} \dfrac{|Z^{1}(\overline{\Si},\mathbb Z_n)|}
    {|\im J'|}& n \text{ is a multiple of $d$}.
    \end{cases}
\end{align*}

From Lemma~\ref{lem:number_ZCp},  
\cite[Lemma 6.9]{KW24}, and $|Z^1(\overline\Sigma,\mathbb Z_k)|= k^{r(\Sigma)}$, we have
$$
|\im J'|=\begin{cases}
|Z^1(\overline{\Sigma},\bZ_n)_{\frac{d}{2}}|=(n')^{r(\Sigma)}    & n \text{ is not a multiple of $d$},\\
|Z^1(\overline{\Si},\mathbb{Z}_n)_{\frac{d}{2}}\cap C^1_{\partial,d}(\overline{\Si},\bZ_n)| = \Big(\dfrac{n'}{2}\Big)^{r(\Sigma)} 2^{2g+b-1}    & n \text{ is a multiple of $d$}, 
\end{cases}$$
Then, we have the claim. 
\end{proof}

Let $W_{n}$ denote the set of odd boundary components of $\overline\Sigma$ when $n$ is odd, and to denote the set of all the boundary components of $\overline\Sigma$ when $n$ is even.
Let $W_{\mathbb Z_2}$ denote the set of maps from $W_{n}$ to $\mathbb Z_2^{n-1}$.
Obviously, $W_{\mathbb Z_2}$ has an abelian group structure. 
We will define a group homomorphism
$\alpha\colon X_{m'}^{\sharp}\rightarrow W_{\mathbb Z_2}$, where $ X_{m'}^{\sharp}$ is defined in \eqref{Xsharp}.
Let ${\bf k}=({\bf k}_1,{\bf k}_2)\in  X_{m'}^{\sharp}$,
where $\mathbf{k} = (
        \tfk_1,\tfk_2
    )\in \bZ^{V_{\lambda}}$ with $\tfk_1\in \bZ^{\obVlast}$ and $\tfk_2\in \bZ^{W}$.
    For any $\partial_i\in W_{n}$,
let $(\mathbf{b}_{i1}, \mathbf{b}_{i2},\dots , \mathbf{b}_{ir_i})$ denote the restriction $\mathbf{k}_2|_{\partial_i}$ to the $i$-th boundary component $\partial_i$ of $\overline{\Sigma}$. 
From the definition of ${X}_{m'}^{\sharp}$, we know 
${\bf b}_{ij}={\bf 0}$ in $\mathbb Z_{m^{\ast}}$ and 
${\bf b}_{ij}=(-1)^{j-1}{\bf b}_{i1}$ in $\mathbb Z_{m'}$ for 
$1\leq j\leq r_i$.
Define 
$$\alpha_{{\bf k}}(\partial_i) = \frac{1}{m^{\ast}} {\bf b}_{i1}\in\mathbb Z_2^{n-1}.$$
Set
$$\alp\colon X_{m'}^{\sharp}\rightarrow W_{\mathbb Z_2},\quad
{\bf k}\mapsto \alpha_{\bf k}.$$
It is easy to see that $\alp$ is a well-defined group homomorphism.

\begin{lem}\label{lem-image-al}
Suppose $m^{\ast}$ is even.
For $n'=2n/d$, 
    we have 
 $$\left|\im\alp\right|=\begin{cases}
        |W_{\mathbb Z_2}|=2^{(b-t)(n-1)} & n\text{ is odd,}\\
        2^{(n-1)b} & n\text{ is even and $m$ is even,}\\
        2^{t} & \text{otherwise},
    \end{cases}$$
where for the third case (i.e., ``the otherwise case") we require $b=t$.
\end{lem}

\begin{proof}
When $b=t$ and $n$ is odd, it is obvious that $\im\alp=\{0\}$.

In the rest of the proof, we assume that
$b>t$ when $n$ is odd.
First, we introduce some notations. We use $d^{\ast}$ to denote $\frac{d}{2}$. Then $d^{\ast} = \gcd(n,m^{\ast})$. There exist integers $\mu$ and $\nu$ such that $\mu n +\nu m^{\ast} = d^{\ast}$.
Then we have $\mu n' +\nu m = 1$.

Let ${\bf k}=({\bf k}_1,{\bf k}_2)\in  X_{m'}^{\sharp}$,
where $\mathbf{k} = 
(\tfk_1,\tfk_2)\in \bZ^{V_{\lambda}}$ with $\tfk_1\in \bZ^{\obVlast}$ and $\tfk_2\in \bZ^{W}$. 
For any $\partial_i\in W_{n}$,
we use $(\mathbf{b}_{i1}, \mathbf{b}_{i2},\dots, \mathbf{b}_{ir_i})$ to denote the restriction $\mathbf{k}_2|_{\partial_i}$ to the $i$-th boundary component $\partial_i$ of $\overline{\Sigma}$. 
Since ${\bf k}$ is balanced, we have 
${\bf b}_{i1} = \lambda_i(1,2,\cdots,n-1) + n{\bf c}_i$ for some ${\bf c}_i$ from Lemma \ref{balance}.
Suppose ${\bf c}_i = (c_{i,1},\cdots,c_{i,n-1})$.
Then we have ${\bf b}_{i1} = (\lambda_i+nc_{i,1},\cdots,
\lambda_i(n-1) + n c_{i,n-1})$.
From the definition of $ X_{m'}^{\sharp}$, we know ${\bf b}_{i1}={\bf 0}$ in $\mathbb Z_{m^{\ast}}$. This shows $d^{\ast}|\lambda_i$.
Define $\lambda_i'=\lambda_i/d^{\ast}$.
Note that, for $1\leq j\leq n-1$, we have 
\begin{equation}
    \begin{split}&j\lambda_i+nc_{i,j} =0\text{ in }\mathbb Z_{m^{\ast}}\\\Leftrightarrow
        &j\lambda_i'+n'c_{i,j} =0\text{ in }\mathbb Z_m\\\Leftrightarrow
       & c_{i,j} = -\mu j\lambda_i' + m l_{i,j}\text{ for some integer $l_{i,j}$}
    \end{split}
\end{equation}
since $\mu n' +\nu m = 1$.
Then, for each $1\leq j\leq n-1$, we have 
$$\frac{j\lambda_i+nc_{i,j}}{m^{\ast}}
=\frac{j\lambda_i'+n'c_{i,j}}{m} = \nu j\lambda_i' + n' l_{i,j}$$
by substituting $ c_{i,j} = -\mu j\lambda_i' + m l_{i,j}$ to the middle. 
This shows
\begin{align}\label{eq-alpha-k-partial-i}
    \alpha_{\bf k}(\partial_i)
    = (\nu \lambda_i' + n' l_{i,1}, 2\nu \lambda_i' + n' l_{i,2},\cdots, (n-1)\nu \lambda_i' + n' l_{i,n-1})\in\mathbb Z_2^{n-1}.
\end{align}

Set $l=b-t$. Assume that $\{\partial_1,\cdots,\partial_l\}$ (resp. $\{\partial_{l+1},\cdots,\partial_b\}$) is the set of odd (resp. even) boundary components of $\overline{\Si}$ (by relabeling if need).

\paragraph{{\bf Case 1}} When $n'$ is even:
Suppose that $n'$ is even, then $\nu$ is odd from $\mu n' +\nu m = 1$. 
Then 
\begin{equation}\label{eq-k-definition-lambda-prime}
    \alpha_{\bf k}(\partial_i)
=\lambda_i'(1,0,1,0,\cdots,1)\in\mathbb Z_2^{n-1}.
\end{equation}
Note that we assume $b=t$ for this case.
Then using the same argument as in the last paragraph in the proof of Lemma~\ref{lem-2t}, we can show that
$|\im\alpha_{m^\ast}|=2^t$.

\paragraph{{\bf Case 2}} When $n'$ is odd and $m$ is even:
When $n'$ is odd and $m$ is even, we will show $W_{\mathbb Z_2}\subset\im\alpha$. 
Suppose $f\in W_{\mathbb Z_2}$. 
For any $\partial_i\in W_n$, suppose 
$f(\partial_i) = (f_{i,1},\cdots,f_{i,n-1})\in\mathbb Z_2^{n-1}$.
Set 
$\lambda_i:=n$, $l_{i,j} := f_{i,j} - \nu jn'$, and  $c_{i,j}
:=-\mu j n' + ml_{i,j}$. Set ${\bf b}_i
:= \lambda_i (1,2,\cdots,n-1) + n(c_{i,1},\cdots,c_{i,n-1})$. 
From the construction, we have ${\bf b}_i={\bf 0}$ in $\mathbb Z_{m^\ast}$ and $\frac{1}{m^{\ast}}{\bf b}_i = f(\partial_i)\in\mathbb Z_2^{n-1}$ since $\frac{1}{m^\ast}(nj\nu m-nj\nu mn')=0$ in $\bZ_2$. 
From the definition of ${\bf b}_i$, we have ${\bf b}_i={\bf 0}$ in $\mathbb Z_n$. Since $m$ is even, then we have ${\bf b}_i={\bf 0}$ in $\mathbb Z_{2n}$. 
Define \(\mathbf{k}_2 \in \mathbb{Z}^{W}\) by setting  
\(
\mathbf{k}_2(v) = 0 \text{ if } v \text{ does not belong to any boundary component } \partial \in W_n.
\)  
For each boundary component \(\partial_i \in W_n\), we define \(\mathbf{k}_2\) on \(\partial_i\) as  
\(
\mathbf{k}_2|_{\partial_i} = (\mathbf{b}_1, \mathbf{b}_1, \dots, \mathbf{b}_1).
\)

Define ${\bf k}_1:=(\frac{1}{n} D- C_1)\frac{{\bf k}_2}{2}$. We have ${\bf k}_1={\bf 0}$ in $\mathbb Z_n$ since ${\bf k}_2={\bf 0}$
in $\mathbb Z_{2n}$. Set
${\bf k} = ({\bf k}_1,{\bf k}_2)\in  X_{m'}^{\sharp}$. 
From the definition of ${\bf k}$, we have
${\bf k} = ({\bf k}_1,{\bf k}_2)\in  X_{m'}^{\sharp}$
and $\alp({\bf k}) = f.$

\paragraph{{\bf Case 3}} When $n$ is even, both $n'$ and $m$ are odd:
Note that we assume $b=t$ for this case.
Then the same argument as in the proof of Lemma~\ref{lem-2t} works here.
\end{proof}

It is easy to see that $X_{m'}$ is a subgroup of $ X_{m'}^{\sharp}$.
Let $\iota$ denote the embedding from $X_{m'}$ to $ X_{m'}^{\sharp}$.

\begin{cor}\label{cor-key-XX}
    We have the following exact sequence 
    $$0\rightarrow X_{m'}\xrightarrow{\iota}  X_{m'}^{\sharp}\xrightarrow{\alp}\im\alp\rightarrow 0.$$
\end{cor}
\begin{proof}
    From the definitions of $\alp$ and $\iota$, we have 
    $\im\iota=\ker\alp$.
\end{proof}

We use $\Lambda_{z}$ to denote the subgroup of $\mathbb{Z}^{V_{\lambda}'}$ generated by $\Lambda_\partial^X$ and $\Gamma_{m'}$.

Proposition~\ref{prop:LY23_11.10} implies there is a group isomorphism $\varphi\colon\mathbb Z^{V_{\lambda}'}\rightarrow \Lambda_{\lambda}$ defined by 
\begin{align}
\varphi\colon\mathbb Z^{V_{\lambda}'}\rightarrow \Lambda_{\lambda},\ \varphi(\textbf{k}) = \textbf{k}\sfK_{\lambda}\label{eq:isom_phi}
\end{align}for $\textbf{k}\in \mathbb Z^{V_{\lambda}'}$ since $K_\lambda$ is invertible (Lemma~\ref{lem:invertible_KH}).

\begin{lem}\label{lem5.10}
Suppose $\overline\Sigma$ contains $t$ even boundary components. Then we have the following:
\begin{enumerate}
    \item $$ \left| \frac{X_{m'}+\varphi(\langle\Lambda_{\partial}\rangle)}{X_{m'}}\right|
=
(m^\ast)^tm^{t(n-2)},$$
\item 
$$\left| \frac{X_{m'}^{\sharp}+\varphi(\langle\Lambda_{\partial}\rangle)}{X_{m'}^{\sharp}}\right|=
(m^\ast)^tm^{t(n-2)}\text{ when $n$ is odd,}$$
\item 
$$\left| \frac{\overline{X}_{m'}^{\sharp}+\varphi(\langle\Lambda_{\partial}\rangle)}{\overline{X}_{m'}^{\sharp}}\right|=
2^{-t}\tilde{m}^t\bar m^{t(n-2)}\text{ when both $m^\ast$ and $n$ are even}.$$
\end{enumerate}
\end{lem}
\begin{proof}

(1) We have 
    $$
   \left| \frac{X_{m'}+\varphi(\langle\Lambda_{\partial}\rangle)}{X_{m'}}\right|=\left|\{
   x+X_{m'}\mid x\in \varphi(\Lambda_\partial) \}\right|.$$
   
   Define
$$Z_{m'}=\{{\bf k}=({\bf k}_1,{\bf k}_2)\in m^\ast\mathbb Z^{V_\lambda}\mid {\bf k}_1={\bf 0}\text{ in }\mathbb Z_{m^\ast},\;{\bf k}_2={\bf 0}\text{ in }\mathbb Z_{m'}\},$$
where $\tfk_1\in \bZ^{\obVlast}$
and ${\bf k}_2\in \bZ^{W}$.
   It is easy to see that there is a bijection between $\{
   x+X_{m'}\mid x\in \varphi(\Lambda_\partial) \}$
   and 
   $\{x+Z_{m'}\mid x\in \varphi(\langle\Lambda_{\partial}\rangle)\}$, where $\{x+Z_{m'}\mid x\in \varphi(\langle\Lambda_{\partial}\rangle)\}$ is a subset of $\frac{\mathbb Z^{V_{\lambda}}}{Z_{m'}}$ (it is actually a subgroup of $\frac{\mathbb Z^{V_{\lambda}}}{Z_{m'}}$),
   using the equation $2\bk_1^T= (\frac{1}{n}D-C_1) \bk_2^T \text{ in }\mathbb Z_{m'}$.
   
Define $$\nu\colon\mathbb Z_{m'}^{n-1}\rightarrow \mathbb Z_{m'}^{n-1},\;\nu(\textbf{p}) = 2\textbf{p} G.$$
   As in the proof of \cite[Lemma 6.13]{KW24}, we can show that
   $$\left|\{
   x+X_{m'}\mid x\in \varphi(\Lambda_\partial) \}\right|=|\im\nu|^t.$$ 
   Lemma \ref{lem;Im_mu} implies that 
   $$\left| \frac{X_{m'}+\varphi(\langle\Lambda_{\partial}\rangle)}{X_{m'}}\right|=(m^\ast)^tm^{t(n-2)}.$$

(2)
When we restrict to the even boundary component, the definitions of $X_{m'}$ and $X_{m'}^{\sharp}$ are the same.
Obviously, we have 
$$\left| \frac{X_{m'}+\varphi(\langle\Lambda_{\partial}\rangle)}{X_{m'}}\right|
= \left| \frac{X_{m'}^{\sharp}+\varphi(\langle\Lambda_{\partial}\rangle)}{X_{m'}^{\sharp}}\right|.$$

(3) 
Define 
$$\nu''\colon\mathbb Z_{\tilde{m}}^{n-1}\rightarrow \mathbb Z_{\tilde{m}}^{n-1},\;\nu''(\textbf{p}) = 2\textbf{p} G.$$
As in the above discussion, we have
$$\left| \frac{\overline{X}_{m'}^{\sharp}+\varphi(\langle\Lambda_{\partial}\rangle)}{\overline{X}_{m'}^{\sharp}}\right|
=|\im\nu''|^t.$$
With $m^\ast$ (resp. $m$) replaced by $\tilde{m}$ (resp. $\bar m$), Lemma \ref{lem;Im_mu}(2) completes the proof. 
\end{proof}

\begin{lem}\label{lem;Im_mu}
\begin{enumerate}
    \item For $\nu\colon\mathbb Z_{m'}^{n-1}\rightarrow \mathbb Z_{m'}^{n-1},\ \mathbf{p}\mapsto 2\mathbf{p} G$, we have $|\im\nu| = m^{\ast}m^{n-2}$.

    \item Suppose that $m^{\ast}$ is even. For $\nu'\colon\mathbb Z_{m^{\ast}}^{n-1}\rightarrow \mathbb Z_{m^{\ast}}^{n-1},\ \mathbf{p}\mapsto 2\mathbf{p} G$, we have $$|\im\nu'| = \begin{cases}
    2^{-1}m^\ast m^{n-2}&\text{when $m$ is odd}\\
    2^{-(n-1)}m^\ast m^{n-2}&\text{when $m$ is even}
\end{cases}.$$
\end{enumerate}
\end{lem}
\begin{proof}
    Recall $G=EF$ (Lemma \ref{matrixG}). Since $\det{E}=1$, we can regard $\nu$ (and $\nu'$) as a map defined by $\textbf{p} \mapsto 2\textbf{p} F$ up to isomorphism. Suppose $\textbf{p}=(p_1,\cdots,p_{n-1})$. 
    
    (1) We have $$\nu(\textbf{p})  = 2((n-1)p_1-np_2,(n-2)p_1-np_3,\cdots,2p_1-np_{n-1},p_1) =\bm{0}\in\mathbb Z_{m'}$$ implies 
$p_1=0\text{ in }\mathbb Z_{m^{\ast}}$ and $ np_{i} =0\text{ in }\mathbb Z_{m^{\ast}}$, $2\leq i\leq n-1$. Since $\gcd(n,m^\ast)=d^\ast$, we have $|\kernel \mu| = 2d^{n-2}$. Then 
$$|\im\nu|=\dfrac{(m')^{n-1}}{|\kernel \nu|} = \frac{m'}{2}(\frac{m'}{d})^{n-2}=m^\ast m^{n-2}.$$

 (2) We have $$\nu'(\textbf{p})  = 2((n-1)p_1-np_2,(n-2)p_1-np_3,\cdots,2p_1-np_{n-1},p_1) =\bm{0}\in\mathbb Z_{m^{\ast}}$$ implies 
$p_1=0\text{ in }\mathbb Z_{m^{\ast}/2}$ and $ np_{i} =0\text{ in }\mathbb Z_{m^{\ast}/2}$, $2\leq i\leq n-1$. 

When $m$ is odd, we have $\gcd(n,m^\ast/2)=d^\ast/2$ and $|\kernel \nu'| = 2(d^\ast)^{n-2}$. Then 
$$|\im\nu'|=\dfrac{(m^\ast)^{n-1}}{|\kernel \nu'|} = \frac{m^\ast}{2}(\frac{m^\ast}{d^\ast})^{n-2}=2^{-1}m^\ast m^{n-2}.$$

When $m$ is even, we have $\gcd(n,m^\ast/2)=d^\ast$ and $|\kernel \nu'| = 2d^{n-2}$. Then 
$$|\im\nu'|=\dfrac{(m^\ast)^{n-1}}{|\kernel \nu'|} = \frac{m^\ast}{2}(\frac{m^\ast}{d})^{n-2}=2^{-(n-1)}m^\ast m^{n-2}.$$
\end{proof}

\begin{thm}\label{thm:rank}
Let $\Sigma$ be a triangulable essentially bordered pb surface without interior punctures, $\lambda$ be a triangulation of $\Sigma$, and $r(\Sigma) := \# (\partial \Sigma) - \chi(\Sigma)$, where $\chi(\Sigma)$ denotes the Euler characteristic of $\Sigma$. 
Suppose that $\overline\Sigma$ has $b$ boundary components and among them there are $t$ even boundary components.
We assume $b=t$ when both of $m^\ast$ and $n$ are even and $m$ is odd.
We have 
\begin{align*}
    &\rankZ \cS_n(\Sigma)=\rankZ \A\\
    &= \begin{cases}
   2^{|W|-r(\Sigma)+t}d^{r(\Sigma)-t}m^{(n^2-1)r(\Sigma)-t(n-1)}   &\text{$m^\ast$ is odd and $n$ is odd}\\
   2^{-2g-2\lfloor\frac{b-t}{2}\rfloor} d^{r(\Sigma)-t}m^{(n^2-1)r(\Sigma)-t(n-1)} &\text{$m^\ast$ is odd and $n$ is  even}\\
2^{|W|-r(\Sigma)+t+(b-t)(1-n)}
d^{r(\Si)-t}
   m^{(n^2-1)r(\Sigma)-t(n-1)} &\text{$m^{\ast}$ is even and }\begin{cases}
   \text{$n$ is odd}\\
    \text{$n$  and $m$ are even}
   \end{cases}\\
   2^{|W|-r(\Sigma)+t}
d^{r(\Si)-t}
   m^{(n^2-1)r(\Sigma)-t(n-1)} &\text{$m^{\ast}$ and $n$ are even, $m$ and $n'$ are odd}\\
   2^{-2g}d^{r(\Si)-t}m^{(n^2-1)r(\Sigma)-t(n-1)} &\text{$m^{\ast}$ is even, and $n'$ is even}
\end{cases}
\end{align*}
where $d$ and $m$ are defined in Section~\ref{notation}. 
\end{thm}

\begin{proof}
Note that $|V_\lambda| = (n^2-1)r(\Si)$ (see Lemma~\ref{lem:cardinarity}).

Proposition \ref{rank_eq} says $\rankZ \cS_n(\Sigma)=\rankZ \A$. 

From Lemma \ref{lem5.1}, 
we have 
$$\rankZ\A=\left|\dfrac{\mathbb Z^{V_{\lambda}'}}{\Lambda_z}\right|=\left|\dfrac{\Lambda_{\lambda}}{\varphi(\Lambda_z)}\right|,$$
where $\Lambda_z$ is defined in Section~\ref{subsec:dim_torus_stated}, 
and $
\varphi(\Lambda_z)=
\begin{cases}
    \Gamma_{m'}+\Lambda_\partial & \text{$m^\ast$ is odd}\\
    \Gamma_{m'} &  \text{$m^\ast$ is even}
\end{cases}
$ by definition.

\noindent {\bf Case 1 ($m^\ast$ is odd):} 
Note that, since $m^\ast$ is odd, the condition whether $n$ is a multiple of $d$ is equivalent to the parity of $n$. 

When $n$ is odd, 
we have 
\begin{align}
\left|\frac{\Lambda_{\lambda}}{X_{m'}}\right| = \left|\dfrac{\Lambda_{\lambda}}{\varphi(\Lambda_z)}\right| \left|\dfrac{\varphi(\Lambda_z)}{X_{m'}}\right|
\label{eq:card_odd}
\end{align}
Proposition \ref{prop5.4} and Lemma \ref{lem5.10} imply 
 \begin{align}\label{eq-rank-A-case1}
     \rankZ \A= 2^{|W|-r(\Sigma)}(m^{\ast})^{-t}d^{r(\Sigma)}
     m^{|V_\lambda|-t(n-2)}=2^{|W|-r(\Sigma)+t}d^{r(\Sigma-t)}m^{|V_\lambda|-t(n-1)} ,
 \end{align} 
where $m'=dm$ and $|V_{\lambda}|=(n^2-1)r(\Sigma)$ (see Lemma \ref{lem:cardinarity}).  

When $n$ is even, by replacing $X_{m'}$ in \eqref{eq:card_odd} with  $X_{m'}^\ast$, 
Lemmas \ref{lem-parity-dif} and \ref{lem5.10} and Proposition \ref{prop5.4} imply 
$$\rankZ \A=   2^{-2g-2\lfloor \frac{b-t}{2}\rfloor} d^{r(\Sigma)-t}m^{(n^2-1)r(\Sigma)-t(n-1)}.$$

\noindent {\bf Case 2 ($m^\ast$ is even and $n$ is odd):} 
The parity condition implies that $\varphi(\Lambda_z)=\varphi(\Gamma_{m'})=\overline{X}_{m'}$.
Lemma \ref{lem-overlineX} implies that  
$$\left|\frac{\Lambda_{\lambda}}{\varphi(\Lambda_z)}\right| = \left|\frac{\Lambda_{\lambda}}{\overline{X}_{m'}}\right|
=\left|\frac{\Lambda_{\lambda}}{X_{m'}^{\sharp}+\varphi(\Lambda_{\partial})}\right|.$$
 We have 
$$\left|\frac{\Lambda_{\lambda}}{X_{m'}^{\sharp}}\right| = \left|\dfrac{\Lambda_{\lambda}}{X_{m'}^{\sharp}+\varphi(\Lambda_{\partial})}\right| \left|\dfrac{X_{m'}^{\sharp}+\varphi(\Lambda_{\partial})}{X_{m'}^{\sharp}}\right|
=(m^{\ast}m^{n-2})^t \left|\dfrac{\Lambda_{\lambda}}{\varphi(\Lambda_z)}\right| $$
where the last equality follows from 
Lemma \ref{lem5.10} (2).
We also have
$$\left|\frac{\Lambda_{\lambda}}{X_{m'}}\right|=
\left|\frac{\Lambda_{\lambda}}{X_{m'}^{\sharp}}\right|
\left|\frac{X_{m'}^{\sharp}}{X_{m'}}\right|= 
    2^{(b-t)(n-1)} \left|\dfrac{\Lambda_{\lambda}}{X_{m'}^{\sharp}}\right|, $$
where the last equality follows from Lemma~\ref{lem-image-al} and Corollary \ref{cor-key-XX}.
Then Proposition \ref{prop5.4} and $m^\ast=dm/2$ imply that 
\begin{eqnarray*}
    \left|\dfrac{\Lambda_{\lambda}}{\varphi(\Lambda_z)}\right|
    &=& (m^{\ast}m^{n-2})^{-t}2^{-(b-t)(n-1)}\left|\dfrac{\Lambda_{\lambda}}{X_{m'}}\right|\\
    &=&
   (dm^{n-1}/2)^{-t}2^{-(b-t)(n-1)+|W|-r(\Sigma)}d^{r(\Sigma)}m^{|V_\lambda|} 
\end{eqnarray*}
From $|V_\lambda|=(n^2-1)r(\Sigma)$, we have the claim.

\noindent {\bf Case 3 ($m^\ast$ and $n$ are even and if $m$ is odd then $b=t$):} 
Lemma  \ref{lem:X/X} implies that 
$$
\left|\frac{\Lambda_{\lambda}}{\varphi(\Lambda_z)}\right|
=\begin{cases}
    \left|\dfrac{\Lambda_{\lambda}}{\overline{X}_{m'}}\right| & n'\text{ is odd,}\medskip\\
     \dfrac{1}{2}\left|\dfrac{\Lambda_{\lambda}}{\overline{X}_{m'}}\right| & n'\text{ is even.}
\end{cases}
$$
Lemmas \ref{lem-overlineX} and \ref{lem5.10} (4) imply that 
\begin{align}\label{eq-b=t:1}
    \left|\dfrac{\Lambda_{\lambda}}{\overline{X}_{m'}}\right| =\left|\dfrac{\overline{X}_{m'}^{\sharp}+\varphi(\langle \Lambda_\partial\rangle)}{\overline{X}_{m'}^{\sharp}}\right|^{-1}
    \left|\dfrac{\Lambda_{\lambda}}{\overline{X}_{m'}^{\sharp}}\right|
    = 2^t \tilde{m}^{-t} \bar{m}^{-t(n-2)}\left|\frac{\Lambda_{\lambda}}{\overline{X}_{m'}^{\sharp}}\right|. 
\end{align}
We have 
\begin{align}\label{eq-b=t:4}
\left|\dfrac{\Lambda_{\lambda}}{\overline{X}_{m'}^{\sharp}}\right|=
\left|\frac{\Lambda_{\lambda}}{X_{m'}}\right|
\left|\frac{X_{m'}}{X_{m'}^{\sharp}}\right|
\left|\frac{X_{m'}^{\sharp}}{\overline{X}_{m'}^{\sharp}}\right|.
\end{align}
Corollary \ref{cor-key-XX} shows that 
\begin{align}\label{eq-b=t:2}
    \left|\dfrac{X_{m'}^{\sharp}}{X_{m'}}\right|=|\im \alpha_{m^\ast}|
    =\begin{cases}
        2^{(n-1)b} & \text{$m$ is even,}\\
        2^{t} & \text{$m$ is odd}.
    \end{cases}
\end{align}
Lemma~\ref{eq;bar_sharp} shows that 
\begin{align}\label{eq-b=t:5}
\left|\dfrac{\overline{X}_{m'}^{\sharp}}{X_{m'}^{\sharp}} \right|=
    \begin{cases}
        2^{(k-1)(n-1)t + t} & \text{$m$ is even}, \\
        1 & \text{$m$ is odd}.
    \end{cases}
\end{align}
By applying \eqref{eq-b=t:1} and \eqref{eq-b=t:4} then  \eqref{eq-b=t:2}, \eqref{eq-b=t:5} and  Proposition \ref{prop5.4}, we have 
\begin{eqnarray*}
    \left|\dfrac{\Lambda_{\lambda}}{\varphi(\Lambda_z)}\right|
    &=&
    2^t \tilde{m}^{-t}\bar{m}^{-t(n-2)}\left|\frac{\Lambda_{\lambda}}{X_{m'}}\right|
\left|\frac{X_{m'}}{X_{m'}^{\sharp}}\right|
\left|\frac{X_{m'}^{\sharp}}{\overline{X}_{m'}^{\sharp}}\right|\medskip\\
&=&
    2^{t+(k+1)t+kt(n-2)} m^{-t(n-1)}\left|\frac{\Lambda_{\lambda}}{X_{m'}}\right|
\left|\frac{X_{m'}}{X_{m'}^{\sharp}}\right|
\left|\frac{X_{m'}^{\sharp}}{\overline{X}_{m'}^{\sharp}}\right|\medskip\\
    &=&\begin{cases}
   2^{|W|-r(\Sigma)+t+(b-t)(1-n)}d^{r(\Si)-t}m^{(n^2-1)r(\Sigma)-t(n-1)}
   &\text{$m$ is even,}\\
   2^{|W|-r(\Sigma)+t}
d^{r(\Si)-t}
   m^{(n^2-1)r(\Sigma)-t(n-1)} &\text{$m^{\ast}$ and $n$ are even, $mn'$ is odd}\\
   2^{-2g}d^{r(\Si)-t}m^{(n^2-1)r(\Sigma)-t(n-1)} &\text{$m^{\ast}$ is even, and $n'$ is even},
    \end{cases}
\end{eqnarray*}
where we used $b=t$ for the last two conditions, and the condition $b=t$ implies that $|W|$ is even.
\end{proof}

\begin{rem}\label{rem;generalize}
    To get rid of the restriction $b=t$ in Theorem~\ref{thm:rank}, it suffices to get rid of the restriction $b=t$ in Lemma~\ref{lem-image-al} because all other results used in the proof of Theorem~\ref{thm:rank} are stated for the general case. 
\end{rem}

\begin{rem}
\begin{enumerate}
    \item For $n=2$ and $d=2$, Theorem \ref{thm:rank} recovers Theorem 5.3 in \cite{Yu23} when the pb surface has no interior punctures. 
    More precisely, since $d=2$, the exponent of $2$ is 
    $$-2g-2\lfloor\frac{b-t}{2}\rfloor+r(\Sigma)-t=\#\partial \Sigma-2+(b-t)-2\lfloor\frac{b-t}{2}\rfloor=2\lfloor\frac{\#\partial \Sigma-1}{2}\rfloor,$$ where the parities of $\#\partial \Sigma$ and $b-t$ are the same, where $\#\partial \Sigma$ and $t$ in this paper correspond to $v$ and $b_2$ in \cite{Yu23} respectively.
    In addition, for $n=2$ and $d=4$, Theorem \ref{thm:rank} recovers Theorem 5.3 in \cite{Yu23} when the pb surface has no interior punctures and no odd boundary components.
\end{enumerate}
\end{rem}

\section{On reduced stated $\SL(n)$-skein algebras}
Let $\Si$ be a triangulable pb surface without interior punctures, and let $\lambda$ be a triangulation of $\Si$.

In this section, 
we will describe the center and PI-degree of $\rA$ when $\hat q$ is a root of unity of even order and $n$ is odd (Theorems~\ref{center_torus-reduced} and \ref{thm-PI-reducedA}), which indicates the center and PI-degree of $\overline{\cS}_n(\Si)$ when $\Si$ is a polygon. Although the strategy is similar to the non-reduced case, the boundary central elements require a ``symmetric property"  (Lemma \ref{reduced-boundary_center} (b)). This has us to deal with this symmetric property for the center of $\rA$.  
Although we assume that $n$ is odd in Theorems~\ref{center_torus-reduced} and \ref{thm-PI-reducedA}, some results are stated for general $n$, such as Lemmas~\ref{lem-k2-zero-reduced}, \ref{lem-image-J-reduced}, \ref{lem-reduced-mu}, and Proposition~\ref{propY}, which could be used in the future to deal with the case when $n$ is even.

\subsection{Quantum torus matrices}\label{sub:quantum-torus-reduced}
Suppose that the connected components of $\partial\overline{\Sigma}$ are labeled as $\partial_1,\partial_2,\dots,\partial_b$. 
For $\partial_i$, label boundary punctures and boundary edges of $\Sigma$ on $\partial_i$ by $p_1, p_2,\dots, p_{r_i}$ and
$e_1, e_2,\dots, e_{r_i}$ in the orientation of $\partial \overline{\Sigma}$. 
Here the endpoints of $e_k$ $(k=1,2,\dots, r_i)$ are $p_k$ and $p_{k+1}$ with indices in $\bZ_{r_i}$. 

We will define a triangulation of $\Sigma$ defined as follows. 
Suppose $r_i\geq 2$. 
Take ideal arcs $e_{2k-1,2k+1}$ $(k=1,2,\dots,\lfloor r_i/2\rfloor)$ connecting $p_{2k-1}$ and $p_{2k}$ which are relatively homotopic to a part of $\partial_i$, where $\lfloor \,\cdot\,\rfloor$ is the floor function.
Then $e_{2k-1},e_{2k},e_{2k-1,2k+1}$ forms an ideal triangle. 
We regard the obtained triangles as attached triangles with the coordinates as in Section~\ref{sec:vertex}. 

When $r_i>1$ is odd, we additionally take an ideal arc $e_{r_i-2,1}$ connecting $p_{1}$ and $p_{r_i-2}$ so that $e_{r_i-2,1}, e_{r_i-1}, e_{r_i-2,r_i-1}$ forms an ideal triangle. 
See, e.g., Figure \ref{Fig;arcs_bdr_odd}.

After the above procedure, we take additional ideal arcs to obtain a triangulation of $\Sigma$, denoted by $\mu$. 
By regarding the triangle bounded by $e_{2k-1},e_{2k},e_{2k-1,2k+1}$ as an attached triangle, let $W$ (resp. $U$) denote the set of all the small vertices $w_j$ (resp. $u_j$) on $\partial\overline{\Sigma}$ with the same indices in Figure \ref{Fig;coord_uvw}. 
When $r_i$ is odd, we label the small vertices on $e_{r_i}$ by $a_1,a_2,\dots,a_{n-1}$ as in Figures~\ref{Fig:r=odd} and \ref{Fig:r=1}.

When $r_i>1$,
we will define the sets $W_i$ and $U_i$. 
Let $W_i$ denote the set of the small vertices in $W$ which are on $\partial_i$. 
Each $w\in W_i$ is determined by $(e_j,w_k)$, where $w$ is the small vertex with the coordinate $w_k$ on the edge $e_j$ in $\partial_i$. 
We identify $(e_j,w_k)$ with $(r_i-j,n-k)$. 
Consider the lexicographic order on $W_i$ with respect to $(r_i-j,n-k)$ and suppose that $W$ is equipped with the order by regarding $W$ as the union of $W_i$'s. 
In the same manner, we define $U_i$ equipped with the order, and so is $U$. 

Only in the case when $r_i\geq 1$ is odd,
consider $V_{i}=\{a_1,a_2,\dots, a_{r_i}\}$ equipped with the order defined by $a_j>a_k\ (j<k)$.

In the rest of this subsection, assume $\Sigma$ is an essentially bordered pb surface and contains no interior punctures.

\begin{figure}[h]  
	\centering\includegraphics[width=7cm]{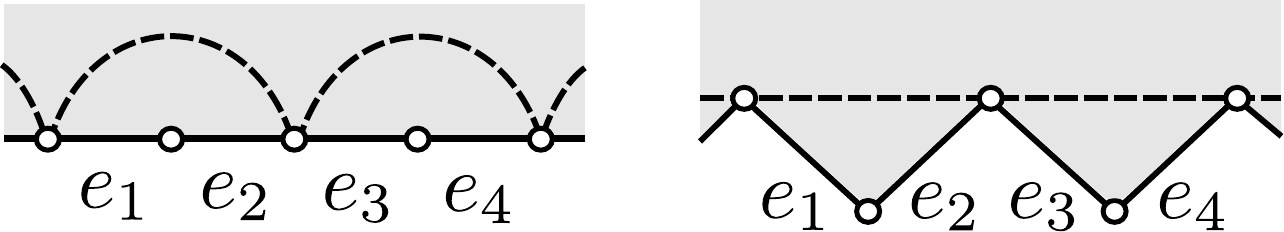}
	\caption{The case when $r_i$ is even. Left: a triangulation $\mu$,\ Right: $\mu$ as an extended triangulation $\lambda^\ast$.}  
	\label{Fig:r=even}   
\end{figure}

\begin{figure}[h]  
	\centering\includegraphics[width=13.5cm]{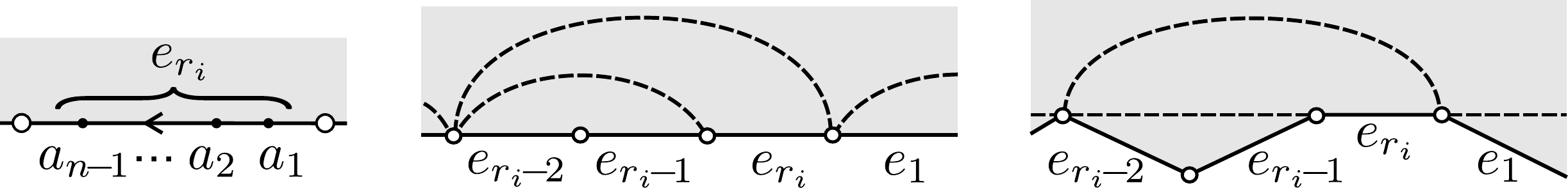}
	\caption{The case when $r_i$ is odd. Left: the small vertices $a_1,a_2,\dots, a_{n-1}$ on $e_{r_i}$,\ Middle: a triangulation $\mu$,\ Right: $\mu$ as a partially extended triangulation.}  
	\label{Fig:r=odd}   
\end{figure}

\begin{figure}[h]
    \centering
    \includegraphics[width=100pt]{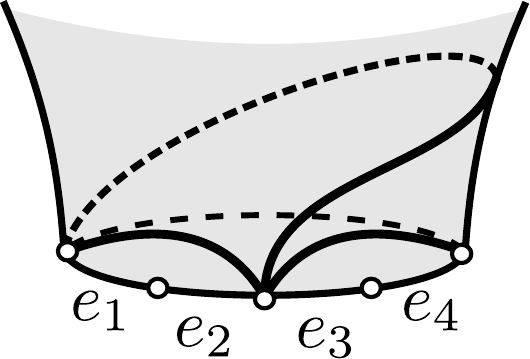}
    \caption{Case of $r_i=5$. The edges $e_{3,1}, e_5, e_{3,4}$ bound a triangle.}\label{Fig;arcs_bdr_odd}
\end{figure}

In Lemmas~\ref{lem:barKQ_even}-\ref{reduced-K}, we recall the calculations of the matrices in \cite{KW24}.

\begin{lem}[{\cite[Lemma 7.1]{KW24}}]\label{lem:barKQ_even}
Let $\Sigma$ be an essentially bordered pb surface without interior punctures such that $\overline{\Sigma}$ has only even boundary components.  
Let $\lambda$ be a triangulation whose extended triangulation $\lambda^\ast$ is $\mu$. 
Then we have 
\begin{align}
\barK_{\mu}\barQ_{\mu}=\barKl\barQl=
\begin{pmatrix}
-2nI & \ast & \ast\\
O    & A    & -A \\
O    & B    & A
\end{pmatrix}, \label{eq:barKQ_even}
\end{align}
where the columns and the rows are divided into $(\mathring{\barV}_{\lambda^\ast}, W, U)$ with respect to $\lambda^\ast$, where $W,U$ are defined in Section~\ref{sec:vertex}. 
\end{lem}

Let $I'$  denote the anti-diagonal matrix whose anti-diagonal entries are 1. 
Let $G'$ denote the matrix obtained from $G$, defined in \eqref{eq-matrix-G-def}, by reversing the rows. 
For $i\in\bZ_{\geq1}$, define the following matrices:
\begin{align*}
    (B;i)&=\begin{pmatrix}
     O & O &  \cdots &O & nI \\
     nI & O & \cdots &O & O \\
     O  & nI& \cdots & O & O\\
     \vdots & \vdots &  & \vdots & \vdots \\
     O & O & \cdots  & nI &  O \\
    \end{pmatrix},\qquad
    (B_O;i)=\begin{pmatrix}
     O & O &  \cdots &O & O \\
     nI & O & \cdots &O & O \\
     O  & nI& \cdots & O & O\\
     \vdots & \vdots &  & \vdots & \vdots \\
     O & O & \cdots  & nI &  O \\
    \end{pmatrix},\;\\
    (A;i) &= \text{diag}\{-nI,\cdots,-nI\},\quad
    (E;i) = (O,\cdots,O, nI),\quad
    (E^T; i) = (nI',O,\cdots,O)^T,\\
      (\widetilde{G};i)&=\begin{pmatrix}
     O & O &  \cdots &O & G \\
     G & O & \cdots &O & O \\
     O  & G& \cdots & O & O\\
     \vdots & \vdots &  & \vdots & \vdots \\
     O & O & \cdots  & G &  O \\
    \end{pmatrix},\qquad
    (\widetilde{G}_O;i)=\begin{pmatrix}
     O & O &  \cdots &O & O \\
     G & O & \cdots &O & O \\
     O  & G& \cdots & O & O\\
     \vdots & \vdots &  & \vdots & \vdots \\
     O & O & \cdots  & G &  O \\
    \end{pmatrix},\\
    (G;i) &= \text{diag}\{G,\cdots,G\},\qquad
    (G_O;i) = \text{diag}\{O,G,\cdots,G\},\\
    (E_G;i) &= (O,\cdots,O, G),\qquad
    (E_G^T; i) = (G',O,\cdots,O)^T,
\end{align*}
where the block matrices $O,I, G$ are of size $n-1$, and the number of the block matrices in the rows or columns is $i$ in each matrix.

From Lemma \ref{lemKQ}, we can
write the matrices 
in the following forms
\begin{align}
\barK_{\mu}\barQ_{\mu}=
\begin{pmatrix}
-2nI & P' \\
O    &  P \\
\end{pmatrix}, \quad
\barK_{\mu}=
\begin{pmatrix}
\ast & \ast \\
\ast    &  K_{\partial} \\
\end{pmatrix},\label{eq;matrix_P}
\end{align}
where the rows and the columns are divided in $(\mathring{\barV}_\mu, \barV_\mu\setminus \mathring{\barV}_\mu)$.

Suppose that the $i$-th boundary component $\partial_i$ of $\overline{\Sigma}$ contains $r_i$ punctures. 
\begin{lem}[{\cite[Lemma 7.2]{KW24}}]\label{lem-reduced-P}
The matrix $P$ has the form
    \begin{align}\label{eq-matrix-P-reduced}
        P = \text{diag}\{P_1,\cdots, P_b\},
    \end{align}
 where the matrix $P_i$ is associated to the $i$-th boundary component with the following descriptions:
$$P_i = \begin{cases}
    
    \begin{pmatrix}
       (A,{\frac{r_i}{2}}) &  -(A,{\frac{r_i}{2}})\\
       (B,{\frac{r_i}{2}})&  (A,{\frac{r_i}{2}}) \\
    \end{pmatrix}
&\text{ if } r_i\text{ is even,}\medskip\\
   
    \begin{pmatrix}
       (A,{\frac{r_i-1}{2}})&  -(A,{\frac{r_i-1}{2}}) & O \\
       (B_O,{\frac{r_i-1}{2}})  &  (A,{\frac{r_i-1}{2}}) & (E^T,{\frac{r_i-1}{2}}) \\
      (E,{\frac{r_i-1}{2}})  &  O & -nI \\
    \end{pmatrix}
&\text{ if } r_i\text{ is odd and } r_i>1,\medskip\\
     -nI+nI'
&\text{ if } r_i=1,
\end{cases}
$$
with the decomposition of rows and columns in $(W_i,U_i), (W_i,U_i,V_i), V_i$ respectively.
\end{lem}

\begin{figure}[h]  
	\centering\includegraphics[width=10cm]{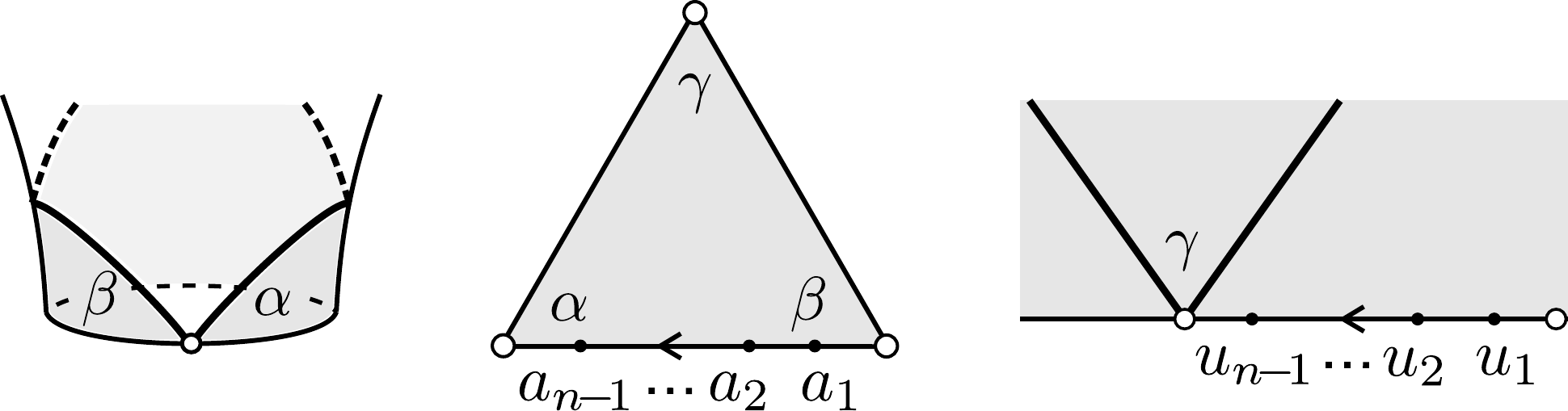}
	\caption{The case when $r_i=1$. Left: the unique triangle containing $\partial_i$,\ Middle: the same triangle with the left one with corners $\alpha,\beta,\gamma$ and the small vertices $a_1,a_2,\dots, a_{n-1}$ on $\partial_i$,\ Right: the corner $\gamma$ and the small vertices $u_1,\dots,u_{n-1}$ on $e$.}  
	\label{Fig:r=1}   
\end{figure}

\begin{lem}[{\cite[Lemma 7.4]{KW24}}]\label{reduced-K}
The matrix $K_\partial$ has the form
\begin{align}\label{eq-matrix-K-reduced}
    K_{\partial} = \text{diag}\{S_1,\cdots, S_b\},
\end{align}
     where the matrix $S_i$ is associated to the $i$-th boundary component with the following descriptions:
\begin{align*}
S_i=
\begin{cases} 
    \begin{pmatrix}
       (G;{\frac{r_i}{2}}) & (G;{\frac{r_i}{2}})\\
       (\widetilde{G};{\frac{r_i}{2}})&  (G;{\frac{r_i}{2}}) \\
    \end{pmatrix}
&\text{if  $r_i$ is even,}\medskip\\
\begin{pmatrix}
       (G;{\frac{r_i-1}{2}})&  (G;{\frac{r_i-1}{2}}) & O \\
       (\widetilde{G}_O;{\frac{r_i-1}{2}})  &  (G;{\frac{r_i-1}{2}}) & (E_G^T;{\frac{r_i-1}{2}}) \\
      (E_G;{\frac{r_i-1}{2}})  &  O & G \\
    \end{pmatrix}
&\text{if $r_i$ is odd and  $r_i>1$,}\medskip\\
    G+G'
&\text{if  $r_i=1$},
\end{cases}
\end{align*}
with rows and columns are divided in $(W_i,U_i),\ (W_i,U_i,V_i),\ V_i$ respectively. 
\end{lem}

The following will be used to prove Theorem~\ref{center_torus-reduced}.

\begin{lem}\label{lem-k2-zero-reduced}
The vector
$\frac{1}{n}P'{\bf k}_2^T={\bf 0}$ in $\mathbb Z_2$ if and only if we have the following:

(1) When $n$ is odd, we have ${\bf k}_2={\bf 0}$ in $\mathbb Z_2$.

(2) When $n$ is even, for each boundary edge $e$, we label the small vertices on $e$ consecutively in the direction of $e$ by
$v^e_1,v^e_2,\cdots,v^e_{n-1}$.
Then, we have 
\begin{align}\label{eq-diff-tri-parity}
({\bf k}_2(v^e_1),\cdots,{\bf k}_2(v^e_{n-1}))
=({\bf k}_2(v^{e'}_1),\cdots,{\bf k}_2(v^{e'}_{n-1}))={\bf 0}\text{ or }(1,0,1,0,\cdots,0,1)\text{ in }\mathbb Z_2
\end{align}
for any two boundary edges $e$ and $e'$.
\end{lem}
\begin{proof}

Suppose $v=(ijk)\in \mathring{\overline{V}}_{\lambda^\ast}$ is a small vertex contained in the triangle $\tau$.
For $v=(ijk)$ and its $\widetilde{Y}_v$ (see Figure \ref{Fig;skeleton}), there are 3 boundary edges $E_1,E_2,E_3$, where $E_1$ (resp. $E_2, E_3$) encounters with the edge of $\widetilde{Y}_v$ obtained by elongating the edge of the main segment with $i$ (resp. $j,k$); see  Figure \ref{Fig;coord_ijk}.
We use the labelings in Figure \ref{Fig;coord_ijk} for the vertices and edges of $\tau$ and we replace each labeling $e_i$ with $\epsilon_i$.
Lemma \ref{lem-matrixPP} shows that
\begin{align}\label{eq-abc-vertex}
    (\frac{1}{n}P'{\bf k}_2^T)(v)
    =\sum_{u}\frac{1}{n}P'(v,u){\bf k}_2(u)={\bf k}_2(a_i) + {\bf k}_2(b_j) + {\bf k}_2(c_k).
\end{align}

If we have (1) or (2), Equation \eqref{eq-abc-vertex} implies that $\frac{1}{n}P'{\bf k}_2^T={\bf 0}$ in $\mathbb Z_2$.

Suppose that $\frac{1}{n}P'{\bf k}_2^T={\bf 0}$ in $\mathbb Z_2$. We will show (1) and (2).
Let $\partial$ be a boundary component of $\overline\Sigma$ with $r$ punctures. 
    
{\bf Case 1}: When $r=1$. 
Suppose $\epsilon_2=\partial$, then $i\neq 0$. We have
$E_i=E_j=E_k=\partial.$
Equation \eqref{eq-abc-vertex} implies that 
$${\bf k}_2(a_i) + {\bf k}_2(a_j) + {\bf k}_2(a_k)=0\text{ in }\mathbb Z_2 \text{ and }
{\bf k}_2(a_{n-j-k}) + {\bf k}_2(a_j) + {\bf k}_2(a_k)=0\text{ in }\mathbb Z_2,$$
where $j+k<n$.

When $n=2j+k$ (i.e., $n-j = j + k$) and $j>0$, we have 
\begin{align}\label{w-j-zero-k}
    {\bf k}_2(a_{n-j-k}) + {\bf k}_2(a_j) + {\bf k}_2(a_k)= 
    2{\bf k}_2(a_j) + {\bf k}_2(a_k)={\bf k}_2(a_k)=0\text{ in }\mathbb Z_2.
\end{align}

When $n$ is odd, Equation \eqref{w-j-zero-k} implies that ${\bf k}_2 (a_{2t-1}) = 0$ in $\mathbb Z_2$ for 
$1\leq t\leq \frac{n-1}{2}$. 
When $0\leq i\leq n-1$, we have
${\bf k}_2(a_i) + {\bf k}_2(a_{n-i}) =0\text{ in }\mathbb Z_2$.
Set $i$ to be an odd number, then 
${\bf k}_2(a_i)=0\text{ in }\mathbb Z_2$ and $n-i$ is even.
This shows that 
${\bf k}_2(a_i)=0\text{ in }\mathbb Z_2$ for $1\leq i\leq n-1.$

When $n$ is even, Equation \eqref{w-j-zero-k} implies that ${\bf k}_2 (a_{2t}) = 0$ in $\mathbb Z_2$ for 
$1\leq t\leq \frac{n-2}{2}$. 
For any two odd integers $1\leq i,j\leq n-1$ such that $i+j\leq n$, Equation \eqref{w-j-zero-k} implies
 that $${\bf k}_2(a_i) + {\bf k}_2(a_j) + {\bf k}_2(a_k)={\bf k}_2(a_i) + {\bf k}_2(a_j)=0\text{ in }\mathbb Z_2.$$ This 
implies that 
$$({\bf k}_2(a_1),\cdots,{\bf k}_2(a_{n-1}))
=
(1,0,1,0,\cdots,0,1)\text{ or }{\bf 0}\text{ in }\mathbb Z_2.$$

{\bf Case 2}: When $r>1$ is odd. 
Suppose $\epsilon_2=e_{2t-1}$ for $1\leq t\leq \frac{r-1}{2}$. We have
$$E_i=e_{2t},\,E_j=e_{2t-1},\, E_k=e_{2t-2}.$$
Suppose $\epsilon_2=e_{2t}$ for $1\leq t\leq \frac{r-1}{2}$. We have
$$E_i=e_{2t-2},\,E_j=e_{2t},\, E_k=e_{2t-1}.$$
For each boundary edge $e_i$, we label the small vertices on $e_i$ consecutively in the direction of $e_i$ by
$v^i_1,v^i_2,\cdots,v^i_{n-1}$.

Equation \ref{eq-abc-vertex} implies that
\begin{align}
    {\bf k}_2(v^{2t}_i) + {\bf k}_2(v_j^{2t-1}) + {\bf k}_2(v_k^{2t-2})=0\text{ in }\mathbb Z_2, \label{eq-k_2-odd-big-1}
    \\
    {\bf k}_2(v^{2t-2}_i) + {\bf k}_2(v_j^{2t}) + {\bf k}_2(v_k^{2t-1})=0\text{ in }\mathbb Z_2. \label{eq-k_2-odd-big-2}
    \end{align}
where $1\leq t\leq \frac{r-1}{2}.$
Set $k=0$ in \eqref{eq-k_2-odd-big-1}, we have 
\begin{align}\label{k_2-i-j-1}
    {\bf k}_2(v^{2t}_i) + {\bf k}_2(v_j^{2t-1})=0\text{ in }\mathbb Z_2
\end{align}
when $i+j=n$ and
$1\leq t\leq \frac{r-1}{2}.$
Set $j=0$ in \eqref{eq-k_2-odd-big-1}, we have
\begin{align}\label{k_2-i-k-1}
    {\bf k}_2(v^{2t}_i) + {\bf k}_2(v_k^{2t-2})=0\text{ in }\mathbb Z_2
\end{align}
 when $i+k=n$ and
$1\leq t\leq \frac{r-1}{2}.$
By taking the subtract of these equations with $j=k$, we have 
\begin{align}\label{k_2-i-k-1-1}
    {\bf k}_2(v^{2t-1}_j) + {\bf k}_2(v_j^{2t-2})=0\text{ in }\mathbb Z_2.
\end{align}
We apply the same procedure to Equation 
\eqref{eq-k_2-odd-big-2}, we have 
\begin{align}\label{k_2-i-k-1-2}
    {\bf k}_2(v^{2t}_j) + {\bf k}_2(v_j^{2t-1})=0\text{ in }\mathbb Z_2.
\end{align}
Equations \eqref{k_2-i-k-1-1} and \eqref{k_2-i-k-1-2} imply 
\begin{align}\label{k_2-i-k-1-3}
    {\bf k}_2(v^{t}_j) + {\bf k}_2(v_j^{t'})=0\text{ in }\mathbb Z_2
\end{align}
for $1\leq t,t'\leq r$ and $1\leq j\leq n-1$.
Equation \eqref{eq-k_2-odd-big-1} and 
\eqref{k_2-i-k-1-3} imply that 
\begin{align}\label{k_2-i-k-1-4}
{\bf k}_2(v^{t}_i)+
    {\bf k}_2(v^{t}_j) + {\bf k}_2(v^{t}_k)=0\text{ in }\mathbb Z_2
\end{align}
when $i+j+k=n$, $i>0$, and $1\leq t\leq r$.
Then the arguments in case 1 complete the rest of proof for this case.

{\bf Case 3:} When $r$ is even. 
The proof is the same with case 2.

The above discussion shows (a).

\begin{figure}[h]  
\centering\includegraphics[width=4.5cm]{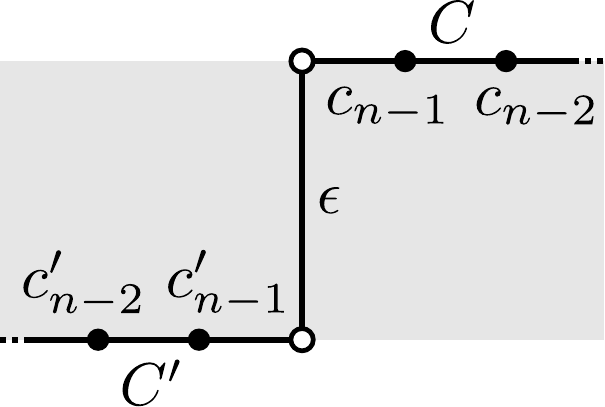}
	\caption{An interior ideal arc $\epsilon$.}
    \label{fig-64744.jpg}
\end{figure}

Suppose that $n$ is even $e$ and $e'$ are contained in the boundary components $C$ and $C'$ of $\overline\Si$.
The above discussion shows equation 
\eqref{eq-diff-tri-parity} when $C=C'$.
Suppose that $C\neq C'$ and there is an ideal arc $\epsilon$ in $\mu$ such that $e$ connects $C$ and $C'$ (see Figure \ref{fig-64744.jpg}). 
Set $v$ in equation \eqref{eq-abc-vertex} to be the small vertices contained in $\epsilon$.
We get 
$${\bf k}_2(c_i)+{\bf k}_2(c_{n-i}')=0\text{ in }\mathbb Z_2$$
for $1\leq i\leq n-1.$
Since $n$ is even,
then $({\bf k}_2(c_1),\cdots,{\bf k}_2(c_{n-1}))= ({\bf k}_2(c_1'),\cdots,{\bf k}_2(c_{n-1}'))\text{ in }\mathbb Z_2.$
Thus we have 
\begin{align*}
&({\bf k}_2(v^e_1),\cdots,{\bf k}_2(v_{n-1}^e))=
    ({\bf k}_2(c_1),\cdots,{\bf k}_2(c_{n-1}))\\
    = &({\bf k}_2(c_1'),\cdots,{\bf k}_2(c_{n-1}'))
    =({\bf k}_2(v_1^{e'}),\cdots,{\bf k}_2(v_{n-1}^{e'}))
    \text{ in }\mathbb Z_2.
\end{align*}

Then Lemma \ref{lem-connection} completes the proof.

\end{proof}

Possibly, the following lemma is written in a paper or a book. 
\begin{lem}\label{lem-connection}
For any triangulation of a connected pb surface $\Sigma$ with $\overline{\Si}$ having at least two boundary components, there is an edge connecting two different boundary components of $\overline{\Si}$. 
In particular, any two different boundary components of $\overline{\Si}$ are connected by a finite sequence of edges. 
\end{lem}
\begin{proof}
After we split $\Sigma$ along all the edges, we have a disjoint union of triangles. Here, each edge of the triangles is a copy of an interior edge of $\Sigma$ or a boundary edge of $\Sigma$. 
If there is no edge connecting different boundary edges, the resulting surface (i.e., $\Si$), obtained from the triangles by gluing two copies of interior edges, has at least two connected components. This contradicts the assumption. 

For an edge connecting two different boundary components of $\overline{\Sigma}$, by splitting the surface along the edge, the two different boundary components are merged. Then we apply the same argument as above. By applying the same procedure repeatedly, we conclude the second claim. 
\end{proof}

The following states the boundary central elements in $\overline{\cS}_n(\Si)$ established in \cite{KW24}. 

\begin{lem}[{\cite[Lemma 7.6]{KW24}}]\label{reduced-boundary_center}
Suppose $\overline{\Sigma}$ has a boundary component $\partial$ such that 
the connected components of $\partial\cap \Sigma$ are labeled as $e_1,e_2,\cdots,e_r$ with respect to the orientation of $\partial$.
For each $1\leq t\leq r$, suppose the small vertices in $e_t$ are labeled as $u_{t,1},\cdots,u_{t,n-1}$ with respect to the orientation of $\partial$. 

(a) Suppose $r$ is even. For $k\in\mathbb N$ and $1\leq i\leq n-1$, the element
\begin{align}\label{eq-reduced-central1}
\bar \gaa_{u_{1,i}}^k \bar \gaa_{u_{2,n-i}}^k\cdots \bar \gaa_{u_{r-1,i}}^k \bar \gaa_{u_{r,n-i}}^k
\end{align}
is central in $\overline \cS_n(\Sigma)$, where $\bar\gaa_v$ with a small vertex $v$ is defined in Section \ref{sub_center}.

(b) Suppose  $r$ is odd. For  $k\in\mathbb N$ and $1\leq i\leq \lfloor\frac{n}{2}\rfloor$, the element
\begin{align}\label{eq-reduced-central2}
\bar \gaa_{u_{1,i}}^k \bar \gaa_{u_{1,n-i}}^k \bar \gaa_{u_{2,i}}^k \bar \gaa_{u_{2,n-i}}^k\cdots \bar \gaa_{u_{r,i}}^k \bar \gaa_{u_{r,n-i}}^k
\end{align}
is central in $\overline \cS_n(\Sigma)$.
\end{lem}

In the rest of this section, we always assume $\Sigma$ is a connected triangulable essentially pb surface without interior punctures. 
Fix a triangulation $\lambda=\mu$ of $\Sigma$, where $\mu$ is introduced in Section~\ref{sub:quantum-torus-reduced}.

Suppose $\overline{\Sigma}$ has boundary components $\partial_1,\cdots,\partial_b$, and each $\partial_i$ contains $r_i$ punctures.
For each $i$  we label the boundary components of $\Sigma$ contained in $\partial_i$ consecutively from
$1$ to $r_i$ following the positive orientation of $\partial_i$.

For any positive integer $t$ and ${\bf k}=(k_1,\cdots,k_t)\in\mathbb Z^t$, define $\cev{{\bf k}} = (k_t,\cdots,k_1)$.

Let $\overline \Lambda_{\partial}$ be the set consisting of ${\bf k}\in \mathbb Z^{\overline V_{\lambda}}$ with the following properties with the order of the small vertices on $\partial_i$ defined by following the positive orientation of $\partial_i$:
\begin{enumerate}
    \item ${\bf k}|_{\obVl} = {\bf 0}$, 
    \item for each even $r_i$, we have ${\bf k}|_{\partial_i} = ({\bf a},\cdots,{\bf a})$, where ${\bf a}\in\mathbb Z^n$, 
    \item for each odd $r_i$, we have ${\bf k}|_{\partial_i} = ({\bf b},\cdots,{\bf b})$, where ${\bf b}\in\mathbb Z^n$ and $\cev{{\bf b}} = {\bf b}.$ 
\end{enumerate}
It is easy to show that $\overline\Lambda_{\partial}$ is a subgroup of $\mathbb Z^{\overline V_{\lambda}}$ and 
$$\overline{{\rm tr}}_{\lambda}^A(\bar{\mathsf{B}}) = 
\{a^{\bf k}\mid \mathbf{k}\in \overline\Lambda_{\partial}\}.$$
Then Lemma \ref{reduced-boundary_center} implies 
$\{a^{\bf k}\mid k\in \overline\Lambda_{\partial}\}\subset\mathcal Z(\overline{\mathcal A}_v(\Sigma,\lambda))$.

We end this subsection by the following three lemmas in \cite{KW24}, which will be used in Theorem~\ref{center_torus-reduced}.

\begin{lem}[{\cite[Lemma 7.7]{KW24}}]\label{reduced-balance}
Let $\Si$ be a triangulable essentially bordered pb surface with a triangulation $\lambda$.
    Suppose $\mathbf{k}\in\mathbb Z^{\bar V_{\lambda}}$ 
    is balanced, and $e$ is a boundary component of $\Sigma$ containing vertices $x_1,\cdots,x_{n-1}$ (the orientation of $e$ gives the labeling of these vertices). Then there exists $l\in\mathbb Z$ such that $(\mathbf{k}(x_1),\mathbf{k}(x_2),\cdots,\mathbf{k}(x_{n-1})) - l(1,2,\cdots,n-1) =\bm{0}$ in $\mathbb Z_n$.
\end{lem}

\begin{lem}[{\cite[Lemma 7.9]{KW24}}]\label{lem-reverse-d}
    Suppose ${\bf a}=(a_1,\cdots,a_{n-1})\in\mathbb Z^{n-1}$ such that ${\bf a} = \cev{{\bf a}}$ in $\mathbb Z_{m'}$, and $n$ is odd.
Assume ${\bf a} = \bm{0}$ in $\mathbb Z_{d}$.
     Then there exists ${\bf b}\in\mathbb Z^{n-1}$ such that ${\bf b}+\cev{{\bf b}}= {\bf a}$ in $\mathbb Z_{m'}$ and 
    ${\bf b} = \bm{0}$ in $\mathbb Z_{d}$.
\end{lem}

We write $G = (G_1,\cdots,G_{n-1})^{T}$. It is easy to show $\overleftarrow{G_i} = G_{n-i}$ for $1\leq i\leq n-1$.
\begin{lem}[{\cite[Lemma 7.10]{KW24}}]\label{lem-G-reverse}
    (a) For any ${\bf k}=(k_1,\cdots,k_{n-1})\in\mathbb Z^{n-1}$, we have
    $\cev{{\bf k}}G={\bf k}G'=\overleftarrow{{\bf k} G}$. 

    (b) For any ${\bf k}\in\mathbb Z^{n-1}$, 
    ${\bf b}:={\bf k}(G+G')$ satisfies ${\bf b} = \cev{{\bf b}}$.

    (c) For any ${\bf c}\in\mathbb Z^{n-1}$, we have
    $$ {\bf c}G'=\cev{\bf c}G=\overleftarrow{{\bf c}G}\text{ and }
    ({\bf c}+\cev{{\bf c}})(G+G') = 2{\bf c}G + 2\overleftarrow{{\bf c}G}.$$
\end{lem}

\subsection{The center when $\hat q$ is a root of unity}

Let $\Si$ be an essentially bordered pb surface.
Recall that we introduced a triangulation $\mu$
in Section~\ref{sub:quantum-torus-reduced}. Actually, $\mu$ belongs to a special family of triangulations, i.e., we made specific choices for some ideal arcs near $\partial\Si$.
In the rest of this section, we will assume the triangulation $\lambda$ of $\Si$ is $\mu$.

Define 
\begin{align}\label{eq-def-Ym}
Y_{m'}= &\{\mathbf{k}=({\bf k}_1,{\bf k}_2)\in\overline{\Lambda}_\lambda \mid 
{\bf k}_1 =\bm{0} \text{ in }\mathbb Z_{m^\ast}\text{ and }
{\bf k}_2 =\bm{0} \text{ in }\mathbb Z_{m'}\},
\end{align}
where ${\bf k}|_{\partial_i}$ is the restriction of ${\bf k}$ to the boundary component $\partial_i$ of $\overline \Si$
and $r_i$ is the number of boundary components of $\Si$ contained in $\partial_i$, $\mathbf{k} = (\tfk_1,\tfk_2)\in \bZ^{\barV_{\lambda}}$ with $\tfk_1\in \bZ^{\obVl}$ and $\tfk_2\in \bZ^{\barV_\lambda\setminus \obVl}$.

Define \begin{equation}
\begin{split}
    \overline{\Gamma}_{m'}= 
        \{\mathbf{k}\in\mathbb Z^{V_{\lambda}'} \mid 
\mathbf{k}\sfK_{\lambda} \in Y_{m'}\},
    \label{eq:reducde-Lambda}
    \end{split}
\end{equation}
and 
\begin{align}\label{reduced-center-eq-a}
    \overline \Lambda_z=
        \overline{\Gamma}_{m'} + \overline \Lambda_{\partial}.
\end{align}

\begin{thm}\label{center_torus-reduced}
Let $\Sigma$ be a triangulable essentially
bordered pb surface without interior punctures, and $\lambda=\mu$ be a triangulation of $\Sigma$ introduced in Section~\ref{sub:quantum-torus-reduced}.
If $m'$ is even and $n$ is odd, we have 
\begin{align*}
\mathcal Z(\rA)= \mathbb C\text{-span}\{a^{\bf k}\mid {\bf k}\in\overline{\Lambda}_{z}\}
\end{align*}
\end{thm}
While the proof is similar to that of Theorem~\ref{center_torus}, we will use different matrices and discuss more for Theorem~\ref{center_torus-reduced}. 
Therefore, we will describe a full proof. 
\begin{proof}
   It is obvious that $\mathbb C\text{-span}\{a^{\bf k}\mid \mathbf{k}\in \overline\Gamma_{m'}\}\subset \mathcal Z(\rA).$
    Lemma \ref{reduced-boundary_center} implies 
    $\mathbb C\text{-span}\{a^{\bf k}\mid \mathbf{k}\in \overline\Lambda_{\partial}\}\subset \mathcal Z(\rA).$
    Then we have 
    $$\mathbb C\text{-span}\{a^{\bf k}\mid {\bf k}\in\overline{\Lambda}_{z}\}\subset\mathcal Z(\rA).$$

    Lemma \ref{quantum} implies $\mathcal Z(\rA) = \mathbb C\text{-span}\{a^{\tft}\mid  \langle {\bf t},{\bf t}'\rangle_{\overline{\sfP}_{\lambda}}=0\text{ in } \mathbb{Z}_{m''},\forall {\bf t}'\in\mathbb{Z}^{\overline V_{\lambda}}  \}$.
    Suppose $\tft_0\in\mathbb Z^{\overline V_{\lambda}}$ such that $\tft \overline{\sfP}_{\lambda} \tft_0^{T} = 0$ {in $\mathbb Z_{m''}$} for any $\tft\in \mathbb Z^{V_{\lambda}'}$, where we regard $\tft,\tft_0$ as row vectors. 
    Lemma~\ref{lem:invertible_KH} implies $\tft \barK_{\lambda} \barQ_{\lambda} (\tft_0\barK_{\lambda})^T = 0 $ in $\mathbb Z_{m''}$ for all $\tft\in \mathbb Z^{\overline V_{\lambda}}.$
    Set $\tfk _0= \tft_0\barK_{\lambda}$. Then we have $\barK_{\lambda}\barQ_{\lambda} \tfk _0^{T}=\bm{0}$ in $\mathbb Z_{m''}.$
    
    We regard $\mathbf{k}_0 = (
        \tfk_1,\tfk_2
    )\in \bZ^{\overline{V}_{\lambda}}$ with $\tfk_1\in \bZ^{\obVl}$ and $\tfk_2\in \bZ^{\overline V_{\lambda}\setminus \obVl}$.  From 
    \eqref{eq;matrix_P}, we have 
    \begin{equation}\label{eq-reduced-key}
    \begin{cases}
    -2n\bk_1^T+ P' \bk_2^T=\bm{0},\\
    P \bk_2^T=\bm{0},
    \end{cases}\text{ in $\mathbb Z_{m''}$}.
    \end{equation}
Equation \eqref{eq-reduced-key} is equivalent to the following one
\begin{equation}\label{eq-reduced-key1}
    \begin{cases}
    -2\bk_1^T+ \frac{1}{n}P' \bk_2^T=\bm{0},\\
    \frac{1}{n}P \bk_2^T=\bm{0},
    \end{cases}\text{ in $\mathbb Z_{m'}$}.
    \end{equation}

Suppose $\overline{\Sigma}$ has boundary components $\partial_1,\cdots,\partial_b$, and each boundary component $\partial _i$ contains $r_i$ punctures.
Suppose $\tfk_2 =  (
\tfk_{\partial_1},\tfk_{\partial_2},\cdots,\tfk_{\partial_b}
)$ where $\tfk_{\partial_i}\in\mathbb Z^{r_i(n-1)}$ is the row vector associated to $\partial_i$ for each $1\leq i\leq b$.
Then Lemma \ref{lem-reduced-P} and $P\tfk_2^T =\bm{0}\text{ in } \mathbb Z_{m''}$ implies
$P_i\tfk_{\partial_i}^T=\mathbf{0}\text{ in } \mathbb Z_{m''}$ for $1\leq i\leq b$.

(1) Suppose $r_i$ is even. Let us use $r$ to denote $r_i$ in this case.
Write $\tfk_{\partial_i} = ({\bf k}_i',{\bf k}_i'')$, where ${\bf k}_i'$ (resp. ${\bf k}_i''$) is the part associated to $W$
(resp. $U$).
Then Lemma \ref{lem-reduced-P} and $P_i\tfk_{\partial_i}^T=\bm{0}\text{ in } \mathbb Z_{m''}$ imply
\begin{align}\label{eq-reduced-AB}
\begin{cases}
    (A,{\frac{r}{2}})({\bf k}_i')^T-(A,{\frac{r}{2}})({\bf k}_i'')^T=\bm{0},\\
    (B,{\frac{r}{2}})({\bf k}_i')^T+ (A,{\frac{r}{2}})({\bf k}_i'')^T=\bm{0},
\end{cases}\text{ in $\mathbb Z_{m''}$}.
\end{align}
The first equality in  \eqref{eq-reduced-AB} implies
$n({\bf k}_i')^T-n({\bf k}_i'')^T=\bm{0} \text{ in }\mathbb Z_{m''}$. 
Then the second equality in  \eqref{eq-reduced-AB} and $(A,{\frac{r}{2}})={\rm diag}\{-nI,\dots,-nI\}$ implies
$$(B,{\frac{r}{2}})({\bf k}_i')^T+ (A,{\frac{r}{2}})({\bf k}_i'')^T = \big((B,{\frac{r}{2}})+ (A,{\frac{r}{2}})\big)({\bf k}_i')^T=\bm{0}\text{ in }\mathbb Z_{m''}.$$
Suppose $\tfk_{i}'= ({\bf b}_{i1},{\bf b}_{i2},\cdots,{\bf b}_{i\frac{r}{2}})$.
Then we have 
\begin{equation}\label{eq's-reduced}
\begin{cases}
-n\tfb_{i1} + n\tfb_{i2} =\bm{0},\\
 -n\tfb_{i2} + n\tfb_{i3} =\bm{0},\\
\;\vdots\\
-n\tfb_{i\,\frac{r}{2}-1} + n\tfb_{i\frac{r}{2}} =\bm{0},\\
-n\tfb_{i\frac{r}{2}} + n\tfb_{i1} =\bm{0},
\end{cases}\text{ in $\mathbb Z_{m''}$}.
\end{equation}
Equation \eqref{eq's-reduced} implies 
$\tfb_{ij} = \cdots=\tfb_{i\frac{r}{2}}$
in $\mathbb Z_{m'}$.
Thus there exists a vector $\tfb_i\in\mathbb Z^{n-1}$
such that $\tfk_{\partial_i} = (\tfb_i,\cdots,\tfb_i)$ in $\mathbb Z_{m'}$.

(2) Suppose $r_i$ is odd and $r_i>1$. We also use $r$ to denote $r_i$ in this case. 
Write $\tfk_{\partial_i} = ({\bf k}_i',{\bf k}_i'',{\bf k}_i)$, where ${\bf k}_i'$ (resp. ${\bf k}_i''$) is the part associated to $W_i$
(resp. $U_i$).
Then Lemma \ref{lem-reduced-P} and $P_i\tfk_{\partial_i}^T=\bm{0}\text{ in } \mathbb Z_{m''}$ imply
\begin{align}\label{eq-reduced-ABE}
\begin{cases}
    (A,\frac{r-1}{2})({\bf k}_i')^T-(A,\frac{r-1}{2})({\bf k}_i'')^T=\bm{0},\\
    (B_O,\frac{r-1}{2})({\bf k}_i')^T+ (A,\frac{r-1}{2})({\bf k}_i'')^T + (E^T,{\frac{r-1}{2}}) (\tfk_i)^T=\bm{0},\\
    (E,{\frac{r-1}{2}})(\tfk_i')^T-n\tfk_i^T=\bm{0}
\end{cases}\text{ in $\mathbb Z_{m''}$}.
\end{align}
The first equation in  \eqref{eq-reduced-ABE} implies
${\bf k}_i'={\bf k}_i'' \text{ in }\mathbb Z_{m'}$.
Suppose $\tfk_{i}'= ({\bf b}_{i1},{\bf b}_{i2},\cdots,{\bf b}_{i\frac{r-1}{2}})$. 
Then the first and second equations in  \eqref{eq-reduced-ABE} imply
\begin{equation}\label{eq's-reduced-1}
\begin{cases}
-n\tfb_{i1} + nI'\tfk_i =\bm{0},\\
-n\tfb_{i1} + n\tfb_{i2} =\bm{0},\\
 -n\tfb_{i2} + n\tfb_{i3} =\bm{0},\\
\;\vdots\\
-n\tfb_{i\frac{r-3}{2}} + n\tfb_{i\frac{r-1}{2}} =\bm{0},
\end{cases}\text{ in $\mathbb Z_{m''}$}.
\end{equation}
Equation \eqref{eq's-reduced-1} implies 
$\tfb_{ij}  = \cdots=\tfb_{i\frac{r-1}{2}}
=\overleftarrow{{\bf k}_i}$ in $\mathbb Z_{m'}$.
The third equation in \eqref{eq-reduced-ABE} implies
$\tfb_{i\frac{r-1}{2}}={\bf k}_i$ 
in $\mathbb Z_{m'}$.
Thus there exists a vector $\tfb_i\in\mathbb Z^{n-1}$
such that $\tfb_i=\cev{\tfb}_i$ in $\mathbb Z_{m'}$ and $\tfk_{\partial_i} = (\tfb_i,\cdots,\tfb_i)$ in $\mathbb Z_{m'}$.

(3) Suppose $r_i=1$. Then $P_i\tfk_{\partial i} =\bm{0}$ in $\mathbb Z_{m''}$ implies $\overleftarrow{\tfk_{\partial _i}} = \tfk_{\partial_i}$ in 
$\mathbb Z_{m'}$. We set $\tfb_i = \tfk_{\partial_i}$.

Now we have vectors ${\bf b}_i$ in all the cases (1)--(3).
When $r_i$ is even, 
using Lemma \ref{reduced-balance} and the techniques in the proofs of Theorem
\ref{center_torus} (when $m^\ast$ is odd) and Lemma~\ref{lem-overlineX} (when $m^\ast$ is even), one can show that, for each $1\leq i\leq b$, there exists ${\bf d}_i\in\mathbb Z^{n-1}$ such that $$2{\bf d}_i G ={\bf b}_i\text{ in $\mathbb Z_{d}$}.$$

When $r_i$ is odd, we have that $\tfb_i=\overleftarrow{\tfb_i}$ in $\mathbb Z_{m'}$. Lemma \ref{reduced-balance} shows that $\tfb_i=l_i(1,2,\cdots,n-1)$ in $\mathbb Z_{n}$. 
Since both $n$ and $d^\ast$ are odd, we have $l_i(\frac{n-1}{2}) = l_i(\frac{n+1}{2})\text{ in }\mathbb Z_{d^\ast}$, which shows $l_i=0\text{ in }\mathbb Z_{d^\ast}$. Lemma \ref{lem-k2-zero-reduced} implies that ${\bf k}_2={\bf 0}\text{ in }\mathbb Z_2$ and $l_i=0\text{ in }\mathbb Z_2$.

Define $\mathbf{d}= (\mathbf{d}_{\partial_1},\mathbf{d}_{\partial_2},\cdots,\mathbf{d}_{\partial_b})\in \mathbb Z^{\overline{V}_{\lambda}\setminus\obVl}$, where $\mathbf{d}_{\partial_i}\in\mathbb Z^{r_i(n-1)}$ is the vector associated to $\partial_i$,  such that 
\begin{equation}
    \mathbf{d}_{\partial_i} = 
        \begin{cases}    ({\bf 0},\cdots,{\bf 0})& \text{if $r_i$ is odd},\\
        (\mathbf{d}_i,\cdots,\mathbf{d}_i) & \text{if $r_i$ is even}.\\
        \end{cases}
    \end{equation}
    Set $\mathbf{d}' := (\mathbf{0},\mathbf{d})\in\mathbb Z^{\overline V_{\lambda}}$ and $\mathbf{f} := \mathbf{d}'\barK_{\lambda}$.
    From the definition of $\mathbf{d}$, we know $\barK_{\lambda}\barQ_{\lambda}\mathbf{f}^{T} = \bm{0}$ in $\mathbb Z_{m''}$.
        We regard $\mathbf{f} = (
        \mathbf{f}_1,\mathbf{f}_2
    )\in \mathbb Z^{V_{\lambda}}$ with $\mathbf{f}_1\in \bZ^{\obVl}$ and $\mathbf{f}_2\in \bZ^{\overline V_{\lambda}\setminus\obVl}$. 
    By replacing $\mathbf{k}_i$ with $\mathbf{f}_i\ (i=1,2)$, $\mathbf{f}_1$ and $\mathbf{f}_2$ satisfy Equation \eqref{eq-reduced-key}. From  Lemma \ref{reduced-K}, we have 
    $\mathbf{f}_2 = \mathbf{d}K_{\partial} = (\mathbf{d}_{\partial_1}S_1,\mathbf{d}_{\partial_2}S_2,\cdots,\mathbf{d}_{\partial_b}S_b)$.

  When $r_i$ is even, we have $$\mathbf{d}_{\partial_i}S_i = (2\mathbf{d}_{i}G,\cdots,2\mathbf{d}_{i}G) =
  \begin{cases}
      ({\bf b}_i,\cdots,{\bf b}_i)\text{ in }\mathbb Z_{d} & n\text{ is odd},\\
       ({\bf b}_i,\cdots,{\bf b}_i)\text{ in }\mathbb Z_{d^{\ast}} & n\text{ is even}.\\
  \end{cases}$$

Set $\mathbf{h} := \mathbf{f} - \mathbf{k}_0$. Then $\mathbf{h}  =
(\mathbf{h} _1,\mathbf{h} _2)$, where $\mathbf{h}_1 = \mathbf{f}_1 - \mathbf{k}_1$ and $\mathbf{h}_2 = \mathbf{f}_2 - \mathbf{k}_2$. Note that $\mathbf{h}_1$ and $\mathbf{h}_2$ satisfy Equation \eqref{eq-reduced-key}.
We regard $\mathbf{h}_2 =  (\tff_{\partial_1},\tff_{\partial_2},\cdots,\tff_{\partial_b})$,  where $\tff_{\partial_i}\in\mathbb Z^{r_i(n-1)}$ is the row vector associated to $\partial_i$ for each $1\leq i\leq b$.
From the above discussion, we know  
$\tff_{\partial_i} = (\tff'_{i},\cdots,\tff'_{i})$ in $\mathbb Z_{m'}$, where 
$\tff'_{i} \in\mathbb Z^{n-1}$ and 
$$\tff'_{i}=\bm{0}\text{ in $\mathbb Z_{d}$}.$$
When $r_i$ is odd, we also have $\tff'_i = \overleftarrow{\tff'_i}$ in $\mathbb Z_{m'}$, and  
Lemma \ref{lem-reverse-d} implies there exists
${\bf e}_i\in\mathbb Z^{n-1}$ such that  ${\bf e}_i+\overleftarrow{{\bf e}_i} = \tff'_i\text{ in }\mathbb Z_{m'}$ and 
$${\bf e}_i=\bm{0}\text{ in }\mathbb Z_{d}
$$
When $r_i$ is even, define ${\bf e}_i$ to be $\tff'_i$.

Using the technique in the proof of Theorem
\ref{center_torus} (when $m^\ast$ is odd) and Lemma~\ref{lem-overlineX} (when $m^\ast$ is even), one can show there exists 
$\mathbf{x}_i''\in \mathbb Z^{n-1}$ such that, for each $1\leq i\leq b$,
$$2\mathbf{x}_i''G = {\bf e}_i\text{ in }\mathbb Z_{m'}.
$$ 

Set $\mathbf{x}_2= (\mathbf{x}_{\partial_1},\mathbf{x}_{\partial_2},\cdots,\mathbf{x}_{\partial_b})\in \mathbb Z^{\overline V_{\lambda}\setminus\obVl}$, where $\mathbf{x}_{\partial_i}\in\mathbb Z^{r_i(n-1)}$ is the vector associated to $\partial_i$,  such that 
\begin{equation}
    \mathbf{x}_{\partial_i} = 
        \begin{cases}    (\mathbf{x}_i''+\overleftarrow{\mathbf{x}_i''},\cdots,\mathbf{x}_i''+\overleftarrow{\mathbf{x}_i''})& \text{if $r_i$ is odd},\\
        (\mathbf{x}_i'',\cdots,\mathbf{x}_i'') & \text{if $r_i$ is even}.\\
        \end{cases}
    \end{equation}
Set $\mathbf{x} = (\bm{0},\mathbf{x}_2)\in\mathbb Z^{\overline V_{\lambda}}$, and define $\mathbf{y} = \mathbf{x}\barK_{\lambda}$.

As in the proof of Theorem \ref{center_torus}, we have 
\begin{align}\label{eq-y-minus-h}
    {\bf y}-{\bf h}\in
    Y_{m'}.
\end{align}
Since both of $\mathbf{y}$ and $\mathbf{h}$ are balanced, then 
$\mathbf{y} - \mathbf{h} = \mathbf{z}\barK_{\lambda}$ for some 
$\mathbf{z}\in \mathbb Z^{\overline V_{\lambda}}$. We have $\mathbf{z}\in\overline \Gamma_{m'}$ since Equation \eqref{eq-y-minus-h}.

We have $\mathbf{h} = \mathbf{f}-\mathbf{k}_0 = \mathbf{d}'\barK_{\lambda}- \mathbf{t}_0\barK_{\lambda}$, and 
$\mathbf{h} = \mathbf{y}-\mathbf{z}\barK_{\lambda} = \mathbf{x}\barK_{\lambda}- \mathbf{z}\barK_{\lambda}$. We also have $\mathbf{t}_0=
\mathbf{d}'-\mathbf{x}+\mathbf{z}$ since $\barK_{\lambda}$ is invertible shown in Lemma~\ref{lem:invertible_KH}.
Then $\mathbf{t}_0\in \overline{\Gamma}_{m'}+\overline{\Lambda}_\partial$ from $\mathbf{d}',\mathbf{x}\in\overline\Lambda_{\partial}$ and $\mathbf{z}\in\overline  \Gamma_{m'}$.
\end{proof}

\subsection{PI-degree}
To compute the PI-degrees for reduced stated $\SL(n)$-skein algebras, we recall and prepare some claims. 
\begin{lem}\cite[Lemma 7.15]{KW24}\label{reduced-exact}
Suppose $k$ is a positive integer and $\gcd(k,n) = l$. Set $N = kn/l$.
   Then there is a short exact sequence 
$$0\rightarrow N\mathbb Z^{\overline V_{\lambda}}\xrightarrow{ 
 L}\overline \Lambda_{\lambda}\cap k\mathbb Z^{\barV_{\lambda}} \xrightarrow{J} Z^1(\overline \Sigma,\mathbb Z_n)_l \rightarrow 0,$$
 where $L$ is the natural embedding and $Z^1(\overline \Sigma,\mathbb Z_n)_l=l(C^1(\overline\Sigma,\mathbb Z_n))\cap Z^1(\overline \Sigma,\mathbb Z_n)$ and $\overline \Lambda_\lambda$ is the balanced part.
 \end{lem}

 The triangulation $\lambda$ gives a cell decomposition of $\overline\Sigma$. We orient the $1$-simplices so that the orientations of the $1$-simplices contained in $\partial \overline \Sigma$ match with that of $\overline\Sigma$. The orientation of $\overline\Sigma$ gives those of the $2$-simplices. 
For every $1$-simplex $e$, we label the vertices of the $n$-triangulation in $e$ as $v_1^{e},v_2^{e},\cdots,v_{n-1}^{e}$  consecutively using the orientation of $e$ such that the orientation of $e$ is given from $v_1^{e}$ to $v_{n-1}^e$. 

Let us review the definition of $J$ given in \cite{KW24}. 
For any $\mathbf{k}\in \overline{\Lambda}_{\lambda}\cap 
k\mathbb Z^{\overline{V}_{\lambda}}$, there exists $s^{\mathbf{k}}_e\in\mathbb Z$ such that
$s^{\mathbf{k}}_e(1,2,\cdots,n-1) = (\mathbf{k}(v_1^e),\cdots, \mathbf{k}(v_{n-1}^e)) \text{ in } \mathbb Z_n$ (note that $s^{\mathbf{k}}_e$ is unique as an element in $\mathbb Z_n$).
Let $s^{\mathbf{k}} \in C^1(\overline\Sigma,\mathbb Z_n)$ be an element such that every 1-simplex $e$ is assigned with $s^{\mathbf{k}}_e\in l\mathbb Z_n$.
 Then define
$$J\colon\overline{\Lambda}_{\lambda}\cap k\mathbb Z^{\overline{V}_{\lambda}} \rightarrow Z^1(\overline \Sigma,\mathbb Z_n)_{l},\quad \mathbf{k}\mapsto s^{\mathbf{k}}.$$

Note that $Y_{m'}$ is a subgroup of 
$\overline{\Lambda}_{\lambda}\cap m^{\ast}\mathbb Z^{\barV_{\lambda}}$, where $Y_{m'}$ was defined as \eqref{eq-def-Ym}.
Set $k=m^\ast$, then $l=\frac{d}{2}$.
We use $J'$ to denote the restriction of $J$ to $Y_{m'}$.
Define
$$\overline{\Omega}_{m'}=\{{\bf k}=({\bf k}_1,{\bf k}_2)\in m^\ast\mathbb Z^{\barV_\lambda}\mid {\bf k}_2={\bf 0}\text{ in }\mathbb Z_{m'}\},$$
where $\tfk_1\in \bZ^{\obVl}$
and ${\bf k}_2\in \bZ^{\barV_\lambda\setminus \obVl}$.
It is easy to see that 
$Y_{m'} = \overline{\Omega}_{m'}\cap m^{\ast}\mathbb Z^{\barV_\lambda}\cap \overline{\Lambda}_\lambda.$
From $N\mathbb Z^{\barV_\lambda} \subset \overline{\Lambda}_\lambda \cap m^{\ast}\mathbb Z^{\barV_\lambda}\ (N = 2nm^{\ast}/d)$ and $\overline{\Omega}_{m'}\subset m^{\ast}\mathbb Z^{\barV_\lambda}$, we have 
\begin{align}\label{reduced-kernel-J-prime}
    \ker J' = N\mathbb Z^{\barV_\lambda}  \cap \overline{\Omega}_{m'}\cap \overline{\Lambda}_\lambda \cap m^{\ast}\mathbb Z^{\barV_\lambda} = N\mathbb Z^{\barV_\lambda}  \cap \overline{\Omega}_{m'}.
\end{align}

Set
$$C^1_{\partial,d}(\overline{\Si},\mathbb Z_n)=\{s\in C^1(\overline{\Si},\mathbb Z_n)\mid s(e)\text{ is a multiple of $d$ in $\mathbb{Z}_{n}$ for any $e\in \lambda$ such that $e\subset \partial\overline{\Si}$}\}.$$

Recall that $n'=\frac{2n}{d}$
and $m=\frac{m'}{d}$, see Section \ref{notation}.

\begin{lem}\label{lem-image-J-reduced}
    We have 
    $$\im J'=\begin{cases}
        \im J & n'\text{ is odd},\\
        Z^1(\overline\Sigma,\mathbb Z_n)_{\frac{d}{2}} \cap C^1_{\partial,d}(\overline{\Si},\mathbb Z_n) & n' \text{ is even}.
    \end{cases}$$
\end{lem}
\begin{proof}
Obviously, we have $\im J'\subset \im J= Z^1(\overline\Sigma,\mathbb Z_n)_{\frac{d}{2}}$ from  Lemma~\ref{reduced-exact}.
Then we shall show that 
$\im J'\subset Z^1(\overline\Sigma,\mathbb Z_n)_{\frac{d}{2}} \cap C^1_{\partial,d}(\overline{\Si},\mathbb Z_n)$ when $n$ is a multiple of $d$.
Let $e$ be a boundary edge, and let ${\bf k}\in Y_{m'}$.
Then $s_e^{\bf k}$ is a multiple of $d$ since $n$ is a multiple of $d$. 

Note that every element in $C^1(\overline\Sigma,\mathbb Z_n)$ is represented by a map from the set of all the 1-simplices to $\mathbb Z_n$. Suppose $c\in Z^1(\overline\Sigma,\mathbb Z_n)_{\frac{d}{2}}$. For any 1-simplex $e$, we choose $t_e\in \mathbb Z$ such that $c(e) = t_e\in\mathbb Z_n$. Since $c\in \frac{d}{2}C^1(\overline\Sigma,\mathbb Z_n)$, we have $t_e$ is a multiple of $\frac{d}{2}$ for each 1-simplex $e$. For each 2-simplex $\tau$, suppose $\tau$, $e_1,\;e_2,\;e_3$ look like in the left picture in Figure \ref{Fig;coord_ijk}. Assume that the orientation of $e_2$ is the one induced from $\tau$ and the orientations of $e_1$ and $e_3$ are the ones opposite to the orientations induced from $\tau$. 
    Since $c\in Z^1(\overline\Sigma,\mathbb Z_n)$, we have $-t_{e_1}+t_{e_2}-t_{e_3} =0\text{ in } \mathbb Z_n$. 
    Set $y_1 = -t_{e_1}$, $y_2 = 0$, $y_3 = -t_{e_2}$. 
    Then we have the following equations in $\mathbb Z_n$;
    \begin{align}\label{reduced-eqy}
        y_2-y_1 = t_{e_1},\quad y_2-y_3 = t_{e_2},\quad y_1-y_3=t_{e_3}.
    \end{align}
Note that each $y_i$ is a multiple of $\frac{d}{2}$. Since $\gcd(2n/d,m'/d) = 1$, then equation $\frac{2y_1}{d}\mathbf{pr}_1+\frac{2y_2}{d}\mathbf{pr}_2+\frac{2y_3}{d}\mathbf{pr}_3+\frac{2n}{d}\mathbf{x} = \mathbf{0}$ has a unique solution $\mathbf{x}=\mathbf{a}_{\tau}\in\mathbb Z^{\overline V_{\tau}}$ in $\mathbb Z_{\frac{m'}{d}}$.  
Then we have  $y_1\mathbf{pr}_1+y_2\mathbf{pr}_2+y_3\mathbf{pr}_3+n\mathbf{a}_{\tau} = \mathbf{0}$ in $\mathbb Z_{m^{\ast}}$. 

Suppose that $\tau$ does not contain boundary edges, define 
    $\mathbf{k}_{\tau} = y_1\mathbf{pr}_1+y_2\mathbf{pr}_2+y_3\mathbf{pr}_3+n\mathbf{a}_{\tau}\in m^\ast\mathbb Z^{\overline{V}_{\tau}}$. Then Equation \eqref{reduced-eqy} implies we have the following equations in $\mathbb Z_n$;
\begin{align}
    \mathbf{k}_{\tau}(v_i^{e_j}) = (-1)^j(y_j-y_{j+1})i=c(e_j)i\quad (j=1,2,3),
\end{align}
where the indices are $\bZ_3$-cyclic.

Suppose that $\tau$ contains boundary edges.

When $n$ is not a multiple of $d$, then $n'=\frac{2n}{d}$ is odd. We have $\gcd(n,m')=\gcd(n,m^{*})=\frac{d}{2}$ and 
$\gcd(\frac{2n}{d},2m)=1$. 
Then equation $\frac{2y_1}{d}\mathbf{pr}_1+\frac{2y_2}{d}\mathbf{pr}_2+\frac{2y_3}{d}\mathbf{pr}_3+\frac{2n}{d}\mathbf{x} = \mathbf{0}$ has a unique solution $\mathbf{x}=\mathbf{b}_{\tau}\in\mathbb Z^{\overline V_{\tau}}$ in $\mathbb Z_{2m}$.  
Then we have  $y_1\mathbf{pr}_1+y_2\mathbf{pr}_2+y_3\mathbf{pr}_3+n\mathbf{b}_{\tau} = \mathbf{0}$ in $\mathbb Z_{m'}$. 
Define 
    $\mathbf{k}_{\tau} = y_1\mathbf{pr}_1+y_2\mathbf{pr}_2+y_3\mathbf{pr}_3+n\mathbf{b}_{\tau}\in m^\ast\mathbb Z^{\overline{V}_{\tau}}$.

When $n$ is a multiple of $d$, then $m$ is odd.
When $\tau$ contains only one boundary edge, we label the unique boundary edge by $e_1$. When $\tau$ contains two boundary edges, we label the two boundary edges by $e_1$ and $e_2$.
 We have $\gcd(\frac{4n}{d},m)=1$. 
Then equation $\frac{2y_1}{d}\mathbf{pr}_1+\frac{2y_2}{d}\mathbf{pr}_2+\frac{2y_3}{d}\mathbf{pr}_3+\frac{4n}{d}\mathbf{x} = \mathbf{0}$ 
has a unique solution $\mathbf{x}=\mathbf{c}_{\tau}\in\mathbb Z^{\overline V_{\tau}}$ in $\mathbb Z_{m}$.  
Define 
    $\mathbf{k}_{\tau} = y_1\mathbf{pr}_1+y_2\mathbf{pr}_2+y_3\mathbf{pr}_3+2n\mathbf{c}_{\tau}\in m^\ast\mathbb Z^{\overline{V}_{\tau}}$. 
{\bf Case 1} when $\tau$ contains only one boundary edge:
Then $y_1=-t_{e_1}$ is a multiple of $d$.
Note that $y_2=0$ and ${\bf pr}_3(v)=0$ for any vertex $v$ in $e_1$. Then 
$(\frac{2y_1}{d}\mathbf{pr}_1+\frac{2y_2}{d}\mathbf{pr}_2+\frac{2y_3}{d}\mathbf{pr}_3+\frac{4n}{d}\mathbf{x})(v)=0\text{ in }\mathbb Z_2$ and 
$(\frac{2y_1}{d}\mathbf{pr}_1+\frac{2y_2}{d}\mathbf{pr}_2+\frac{2y_3}{d}\mathbf{pr}_3+\frac{4n}{d}\mathbf{x})(v)=0\text{ in }\mathbb Z_{2m}$ since $m$ is odd.
Then ${\bf k}_\tau(v)=0\text{ in }\mathbb Z_{m'}$.
{\bf Case 2} when $\tau$ contains at least two boundary edges:
Then $y_1$, $y_2$, and $y_3$ are multiples of $d$.
Similarly, we have ${\bf k}_\tau(v)=0\text{ in }\mathbb Z_{m'}$ when $v$ is contained in a boundary edge.

     The last paragraph of the proof of Lemma \ref{lem-image-J} will complete the proof.
 \end{proof}

\begin{prop}\label{propY}
    We have $$\displaystyle \left|\frac{\overline{\Lambda}_{\lambda}}{Y_{m'}}\right|=
\begin{cases}
2^{(n-1)(\#\partial\Sigma)-r(\Sigma)}m^{|\overline{V}_\lambda|} d^{r(\Sigma)}
    & n' \text{ is odd},\medskip\\   
  2^{-2g-b+1} m^{|\overline{V}_\lambda|}d^{r(\Sigma)}
    & n' \text{ is even},
    \end{cases}$$
    where $b$ is the number of boundary components of $\overline\Si$ and $r(\Si)$ is defined in \eqref{eq-def-r-sigma}.
\end{prop}
\begin{proof}
    Using Equation~\eqref{reduced-kernel-J-prime} and Lemma~\ref{lem-image-J-reduced}, then the proof of Proposition~\ref{prop5.4} works here. Note that, in this case,  $|W|$ is replaced with $(n-1)(\#\partial\Sigma)$ in the proof of Proposition~\ref{prop5.4}.
\end{proof}

\begin{rem}\label{rem-G}
    Suppose ${\bf k}\in\mathbb Z^{n-1}$ such that ${\bf k}=\cev{{\bf k}}$.
Lemma \ref{lem-G-reverse} implies $${\bf k}(G+G') = {\bf k}G+{\bf k}G'
= {\bf k}G+\cev{{\bf k}}G = 2{\bf k}G.$$
\end{rem}

\begin{rem}\label{reduced-rem_partial}
Proposition~\ref{prop:LY23_11.10} implies there is a group isomorphism $\overline \varphi\colon\mathbb Z^{\overline V_{\lambda}}\rightarrow \overline\Lambda_{\lambda}$, defined by $\overline\varphi(\textbf{k}) = \textbf{k}\barK_{\lambda}$ for $\textbf{k}\in \mathbb Z^{\overline V_{\lambda}}$.
Then $\left|\dfrac{\mathbb Z^{\overline V_{\lambda}}}{\overline\Lambda_z}\right| =
\left|\dfrac{\overline\Lambda_{\lambda}}{\overline\varphi(\overline\Lambda_z)}\right|.$
Recall $\overline\Lambda_z = \overline\Lambda_{m'}+\langle\overline\Lambda_{\partial}\rangle,$ where $\langle\overline\Lambda_{\partial}\rangle$ is the subgroup of $\mathbb Z^{\overline V_{\lambda}}$ generated by $\overline\Lambda_{\partial}$. From the definition of $\overline\Lambda_{m'}$, we have $\overline\varphi(\overline\Lambda_z) = (\overline\Lambda_{\lambda}\cap m'\mathbb Z^{\overline V_{\lambda}})+\overline\varphi(\langle\overline\Lambda_{\partial}\rangle)$. 

Let us understand $\overline{\varphi}(\langle \overline\Lambda_{\partial}\rangle)$. 
Suppose $\partial_1,\cdots,\partial_t$ (resp. $\partial_{t+1},\cdots,\partial_b$) are boundary components with even (resp. odd )number of boundary punctures. 
For any element $\textbf{k}\in \langle\overline \Lambda_{\partial}\rangle$,
we write $\mathbf{k} = (\tfk_1,\tfk_2)\in \bZ^{\overline V_{\lambda}}$ with $\tfk_1\in \bZ^{\obVl}$ and $\tfk_2\in \bZ^{\overline V_{\lambda}\setminus\obVl}$. 
Set $\textbf{d}_i:=\textbf{k}_2|_{\partial_i}$ for $1\leq i\leq b$.
Then each $\textbf{d}_i:=\textbf{k}_2|_{\partial_i} = ({\bf c}_i,\cdots,{\bf c}_i)$  where ${\bf c}_i\in\bZ^{n-1}$.
We have ${\bf c}_i = \overleftarrow{{\bf c}_i}$ for $t+1\leq i\leq b$.

 Suppose $\overline \varphi(\tfk) = (\textbf{u}_1,\textbf{u}_2)$, where $\textbf{u}_1\in \bZ^{\obVl}$ and $\textbf{u}_2\in \bZ^{\overline V_{\lambda}\setminus\obVl}$. Then 
$(\textbf{u}_1,\textbf{u}_2)$ satisfies Equation \eqref{eq-reduced-key} because $a^{\tfk}$ is in the center of $\rA$.

Equation \eqref{eq;matrix_P} and Lemma \ref{reduced-K} imply 
$\textbf{u}_2 = (\textbf{d}_1S_1,\cdots,\textbf{d}_bS_b)$.
Lemma \ref{reduced-K} and Remark \ref{rem-G} imply $\textbf{d}_iS_i= (2{\bf e}_i G,\cdots,2{\bf e}_i G)$. 
\end{rem}

Set \begin{align}
\overleftarrow{\mathbb Z_{m'}^{n-1}} =
\{{\bf k}\in \mathbb Z_{m'}^{n-1}\mid {\bf k}=\cev{{\bf k}}\in \mathbb Z_{m'}^{n-1}\}. \label{leftZ}
\end{align}

\begin{lem}\label{lem-reduced-mu}
    For $\cev\nu\colon\overleftarrow{\mathbb Z_{m'}^{n-1}}\rightarrow \overleftarrow{\mathbb Z_{m'}^{n-1}},\ \mathbf{p}\mapsto 2\mathbf{p} G$, we have $$|\im\cev\nu| = \begin{cases}
    2m^{\lfloor\frac{n}{2}\rfloor} & n\text{ is even and }n'\text{ is odd.}\\
        m^{\lfloor\frac{n}{2}\rfloor} & \text{otherwise}.
    \end{cases}$$
\end{lem}
\begin{proof}

    Suppose $\{{\bf c}_1,\cdots,{\bf c}_{n-1}\}$ is the standard basis for $\mathbb Z^{n-1}$, i.e., all the entries of each ${\bf c}_i$ are zeros except the $i$-th one being $1$.

    Suppose $n$ is even. Define 
    ${\bf a}_i = {\bf c}_i+{\bf c}_{i+1}+\cdots+{\bf c}_{n-i-1}+{\bf c}_{n-i}$ for $1\leq i\leq \frac{n-1}{2}$, and ${\bf a}_{\frac{n}{2}}
    ={\bf c}_{\frac{n}{2}}.$ 
    Recall that we write $G=(G_1^T,\cdots,G_{n-1}^T)^T$ and $\overleftarrow{G_i} = G_{n-i}$ for $1\leq i\leq n-1$.
    It is easy to check 
    $$G_{i}+G_{n-i} = n({\bf a}_1+\cdots+{\bf a}_i)\text{ for } 1\leq i\leq \frac{n}{2}-1,\text{ and } G_{\frac{n}{2}} = \frac{n}{2}
    ({\bf a}_1+\cdots+{\bf a}_{\frac{n}{2}}).$$

    For any ${\bf b} = (b_1,\cdots,b_{n-1})\in\overleftarrow{\mathbb Z_{m'}^{n-1}}$, we have 
    \begin{equation}\label{eq-Zm}
    \begin{split}
        &\cev\nu({\bf b})\\ = &2\left(b_1G_1+\cdots+b_{n-1} G_{n-1}\right)
        =2\left(b_1(G_1+ G_{n-1}) +\cdots+ b_{\frac{n}{2}-1}(G_{\frac{n}{2}-1}+ G_{\frac{n}{2}+1})+ b_{\frac{n}{2}} G_{\frac{n}{2}}\right)\\
        =&2(n b_1{\bf a}_1+ n b_2({\bf a}_1 +{\bf a}_2)+\cdots
        +nb_{\frac{n}{2}-1}({\bf a}_1 +{\bf a}_2+\cdots+{\bf a}_{\frac{n}{2}-1})+\frac{n}{2}b_{\frac{n}{2}}({\bf a}_1 +{\bf a}_2+\cdots+{\bf a}_{\frac{n}{2}}))\\
        =&n\big( 2b_1{\bf a}_1+ 2b_2({\bf a}_1 +{\bf a}_2)+\cdots
        +2b_{\frac{n}{2}-1}({\bf a}_1 +{\bf a}_2+\cdots+{\bf a}_{\frac{n}{2}-1})+b_{\frac{n}{2}}({\bf a}_1 +{\bf a}_2+\cdots+{\bf a}_{\frac{n}{2}})\big).
        \end{split}
    \end{equation}
    Note that we regard $Z_{m}^{n-1}$ as a module over $\mathbb Z_{m}$.
    Obviously ${\bf a}_1, {\bf a}_1 +{\bf a}_2,\cdots, {\bf a}_1 +{\bf a}_2+\cdots+{\bf a}_{\frac{n}{2}}$ are $\mathbb Z_m$-linearly independent over $\mathbb Z_{m}$.
    Then \eqref{eq-Zm} implies 
    \begin{equation}\label{ker-mu}
        \begin{split}
             &\cev\nu({\bf b})=\bm{0}\in\mathbb Z_{m'}^{n-1}\\
             \Leftrightarrow&
             2nb_i=0\text{ in }\mathbb Z_{m'}\text{ for } 1\leq i\leq \frac{n}{2}-1,\text{ and }
             nb_{\frac{n}{2}}=0\text{ in }\mathbb Z_{m'}\\
         \Leftrightarrow& 
         \begin{cases}
             b_1=b_2=\cdots=b_{\frac{n}{2}} =0\text{ in }\mathbb Z_{m} & n'\text{ is even,}\\
              b_1=b_2=\cdots=b_{\frac{n}{2}-1} =0\text{ in }\mathbb Z_{m}\text{ and }b_{\frac{n}{2}}=0\text{ in }\mathbb Z_{2m} &\text{$n'$ is odd.}
         \end{cases}
        \end{split}
    \end{equation}
    
    Define $\begin{cases}
        f\colon\overleftarrow{\mathbb Z_{m'}^{n-1}}\rightarrow \mathbb Z_{m}^{\frac{n}{2}}&\text{$n'$ is even,}\\
        f\colon\overleftarrow{\mathbb Z_{m'}^{n-1}}\rightarrow \mathbb Z_{m}^{\frac{n}{2}-1}\times \mathbb Z_{2m}&\text{$n'$ is odd,}
    \end{cases}$ such that $f({\bf b}) = (b_1,\cdots,b_{\frac{n}{2}})$ for any ${\bf b}=(b_1,\cdots,b_{n-1})\in \overleftarrow{\mathbb Z_{m'}^{n-1}}$.
    Then \eqref{ker-mu} implies $\ker\cev{\nu} = \ker f$.
    Define $g\colon\mathbb Z_{m'}^{\frac{n}{2}}\rightarrow \overleftarrow{\mathbb Z_{m'}^{n-1}}$ such that
    $g({\bf d}) = (d_1,\cdots,d_{\frac{n}{2}-1},d_{\frac{n}{2}},d_{\frac{n}{2}-1},\cdots,d_1)$ for any ${\bf d}=(d_1,\cdots,d_{\frac{n}{2}-1},d_{\frac{n}{2}})\in \mathbb Z_{m'}^{\frac{n}{2}}$. Note that $g$ is a group isomorphism. 
    One can easily check that $f\circ g$ is the projection from 
    $\mathbb Z_{m'}^{\frac{n}{2}}$ to $\mathbb Z_{m}^{\frac{n}{2}}$ or $\mathbb Z_{m}^{\frac{n}{2}-1}\times \mathbb Z_{2m}$.
    Then $$|\ker\cev{\nu}|=|\ker f |=|\ker(f\circ g)| =
    \begin{cases}
         d^{\frac{n}{2}}&\text{$n'$ is even,}\\
         d^{\frac{n}{2}-1}d^\ast&\text{$n'$ is odd.}
    \end{cases}
   $$
    Thus we have 
    $$|\im\cev{\nu}|=\left|\frac{\overleftarrow{\mathbb Z_{m'}^{n-1}}}{\ker\cev\nu}\right|=\left|\frac{\mathbb Z_{m'}^{\frac{n}{2}}}{\ker\cev\nu}\right|
    = \begin{cases}
        m^{\lfloor\frac{n}{2}\rfloor} & n'\text{ is even,}\\
        2m^{\lfloor\frac{n}{2}\rfloor} & n'\text{ is odd.}
    \end{cases}.$$

    Using the same technique as above, one can show $|\im\cev\nu| = 
    m^{\frac{n-1}{2}}=m^{\lfloor\frac{n}{2}\rfloor}$ when $n$ is odd.
    Note that there would be no $d_{\frac{n}{2}}$ in \eqref{eq-Zm} and \eqref{ker-mu} when $n$ is odd.
    
\end{proof}

\begin{lem}\label{reduced-PI2}
Suppose $\overline\Sigma$ contains $t$ even boundary components. Then we have 
$$ \left| \frac{Y_{m'}+\overline{\varphi}(\overline{\Lambda}_{\partial})}{Y_{m'}}\right|
= \begin{cases}
   2^{b-t}(m^\ast)^t m^{(n-2)t} m^{(b-t)\lfloor\frac{n}{2}\rfloor} & n\text{ is even and }n'\text{ is odd.}\\
       (m^\ast)^t m^{(n-2)t} m^{(b-t)\lfloor\frac{n}{2}\rfloor} & \text{otherwise}.
    \end{cases}$$
\end{lem}
\begin{proof}
Define $$\nu\colon\mathbb Z_{m'}^{n-1}\rightarrow \mathbb Z_{m'}^{n-1},\;\nu(\textbf{p}) = 2\textbf{p} G\text{ and }
\cev{\nu}\colon\overleftarrow{\mathbb Z_{m'}^{n-1}}\rightarrow \overleftarrow{\mathbb Z_{m'}^{n-1}},\;\cev{\nu}(\textbf{p}) = 2\textbf{p}G,$$ 
where $\overleftarrow{\mathbb Z_{m'}^{n-1}}$ is defined as  \eqref{leftZ}.

In the rest of the proof, we follow the notations in Remark \ref{reduced-rem_partial}.

For any element $\textbf{u}=(\textbf{u}_1,\textbf{u}_2)\in \overline\varphi(\overline\Lambda_{\partial})$, we have $\textbf{u}_2 = (\textbf{d}_1S_1,\cdots,\textbf{d}_bS_b)$, where $\textbf{d}_iS_i= (2{\bf e}_i G,\cdots,2{\bf e}_i G)$.
Note that ${\bf e}_i\in\mathbb Z^{n-1}$ for $1\leq i\leq b$ and 
${\bf e}_j = \overleftarrow{{\bf e}_j}$ for $t+1\leq j\leq b$. 

Define $\overline\theta(\textbf{u}) = (2{\bf e}_1 G,\cdots,2{\bf e}_b G)\in (\im\nu)^t\times (\im\cev\nu)^{b-t}$.
The arguments in the proofs of Lemma \ref{lem5.10} and \cite[Lemma 7.19]{KW24} show that $$\overline\theta\colon \frac{Y_{m'}+\overline{\varphi}(\overline{\Lambda}_{\partial})}{Y_{m'}}\rightarrow (\im\nu)^t\times (\im\cev\nu)^{b-t}$$ is a group isomorphism.
Then Lemmas \ref{lem;Im_mu} and \ref{lem-reduced-mu} show that
$$\left| \frac{Y_{m'}+\overline{\varphi}(\langle\overline{\Lambda}_{\partial}\rangle)}{Y_{m'}}\right|=|\im\nu|^t |(\im\cev\nu)|^{b-t} = \begin{cases}
   2^{b-t}(m^\ast)^t m^{(n-2)t} m^{(b-t)\lfloor\frac{n}{2}\rfloor} & n\text{ is even and }n'\text{ is odd.}\\
       (m^\ast)^t m^{(n-2)t} m^{(b-t)\lfloor\frac{n}{2}\rfloor} & \text{otherwise}.
    \end{cases}$$

\end{proof}

\begin{thm}\label{thm-PI-reducedA}
Assume that $m'$ is even and $n$ is odd.
Let $\Sigma$ be a triangulable essentially
bordered pb surface without interior punctures, and $\lambda=\mu$ be a triangulation of $\Sigma$
introduced in Section~\ref{sub:quantum-torus-reduced}.
Suppose $\overline\Sigma$ has $b$ boundary components among which there are $t$ even boundary components. Then we have 
$$\rankZ \rA=
   2^{(n-1)(\sharp\partial\Si)-r(\Sigma)+t}d^{r(\Sigma)-t}m^{|\overline{V}_\lambda|-t(n-1)-(b-t)\lfloor\frac{n}{2}\rfloor}$$
where $d,m$ are defined in Section~\ref{notation} and $|\overline V_{\lambda}|=(n^2-1)r(\Sigma)-\binom{n}{2}(\#\partial\Sigma)$ is given in \eqref{eq:cardinarity}. 
\end{thm}
\begin{proof}

From Lemma \ref{lem5.1}, 
we have 
$$\rankZ\rA=\left|\dfrac{\mathbb Z^{V_{\lambda}'}}{\overline{\Lambda}_z}\right|=\left|\dfrac{\overline{\Lambda}_{\lambda}}{\overline{\varphi}(\overline{\Lambda}_z)}\right|,$$
where $\overline{\Lambda}_z$ is defined in \eqref{reduced-center-eq-a}, 
and $\varphi(\overline{\Lambda}_z)=\overline{\Gamma}_{m'}$ by definition.

We have 
\begin{align}
\left|\frac{\overline{\Lambda}_{\lambda}}{Y_{m'}}\right| = \left|\dfrac{\overline{\Lambda}_{\lambda}}{\overline{\varphi}(\overline{\Lambda}_z)}\right| \left|\dfrac{\overline{\varphi}(\overline{\Lambda}_z)}{Y_{m'}}\right|,
\label{eq:card_odd_red}
\end{align}
where $\overline{\varphi}(\overline{\Lambda}_z)=Y_{m'}+\overline{\varphi}$ from \eqref{reduced-center-eq-a}. 
Proposition \ref{propY} and Lemma \ref{reduced-PI2} imply 
 \begin{eqnarray}\label{eq-rank-A-case1_red}
     \rankZ \rA&=& 2^{(n-1)(\#\partial\Sigma)-r(\Sigma)}(m^{\ast})^{-t}d^{r(\Sigma)}
     m^{|V_\lambda|-t(n-2)-(b-t)\lfloor\frac{n}{2}\rfloor}\nonumber\\
     &=&2^{(n-1)(\#\partial\Sigma)-r(\Sigma)+t}d^{r(\Sigma)-t}m^{|\overline{V}_\lambda|-t(n-1)-(b-t)\lfloor\frac{n}{2}\rfloor} ,
 \end{eqnarray} 
where $m^\ast=dm/2$ and $|V_{\lambda}|=(n^2-1)r(\Sigma)$ in the last equality. 
\end{proof}

\section{Matrix decomposition} 
In this section, we give matrix decompositions 
of $\mathsf{P}_\lambda$
and $\overline{\mathsf{P}}_\lambda$, which are an analog of \cite{BL07, KW24, Wan25}. 

\def\TP{\mathbb T(\mathsf{P})}

\begin{thm}\cite[Theorem 4.1]{New72}\label{thm-decom-symmetric}
    Suppose $\mathsf{P}$ is an anti-symmetric integer matrix of size $r$.
    There exists an integral matrix $\mathsf{X}$ with 
    $\det\mathsf{X} =\pm 1$ such that 
    \begin{align}\label{eq-decom-P}
        \mathsf{X}^T \mathsf{P}\mathsf{X} = 
    \text{diag}\left\{
    \begin{pmatrix}
        0   & h_1\\
        -h_1& 0
    \end{pmatrix},\cdots,
    \begin{pmatrix}
        0   & h_k\\
        -h_k& 0
    \end{pmatrix},
    0,\cdots,0
    \right\},
    \end{align}
    where $h_i\mid h_{i+1}$ for $1\leq i\leq k-1$.
\end{thm}

We refer to the decomposition in \eqref{eq-decom-P} as the {\bf anti-symmetric matrix decomposition}.

\begin{lem}\cite[Lemma 4.22]{Wan25}\label{lem-2power-eq}
    Let $\{n_1,\cdots,n_k\}$ and 
    $\{m_1,\cdots,m_k\}$ be sequences of positive integers such that $n_i\mid n_{i+1}$ and $m_i\mid m_{i+1}$ for $1\leq i\leq k-1$.
    Suppose $\prod_{1\leq i\leq k}\frac{l}{\gcd(l,n_i)} = \prod_{1\leq i\leq k}\frac{l}{\gcd(l,m_i)}$ for any positive integer $l$. Then
    $m_i= n_i$ for each $1\leq i\leq k$.
\end{lem}

\begin{rem}
\cite[Lemma 8.9]{KW24} states a similar result as in 
Lemma~\ref{lem-2power-eq} except that the $l$ in
\cite[Lemma 8.9]{KW24} is any odd positive integer
and $m_i=2^{k_i}n_i$ for some integer $k_i$.
\end{rem}

\begin{thm}\cite[Theorem 6.15]{KW24}\label{thm:rank-odd}
Suppose $m'$ is odd.
Let $\Sigma$ be a triangulable essentially
bordered pb surface without interior punctures, $\lambda$ be a triangulation of $\Sigma$, and $r(\Sigma) := \# (\partial \Sigma) - \chi(\Sigma)$, where $\chi(\Sigma)$ denotes the Euler characteristic of $\Sigma$.
Suppose $\overline\Sigma$ contains $t$ even boundary components. 
We have $$\rankZ \cS_n(\Sigma)=\rankZ \A= d^{r(\Sigma)-t}m^{(n^2-1)r(\Sigma)-t(n-1)}, $$
where $d,m$ are defined in Section~\ref{notation}.
\end{thm}

Anti-symmetric matrices $\mathsf{P}_\lambda$
and $\overline{\mathsf{P}}_\lambda$ are defined in \eqref{eq-anti-matric-P-def}. 
From Lemma~\ref{lem:invertible_KH}, their entries are in $n\mathbb Z$.
From Theorem \ref{thm-decom-symmetric}, we can suppose the
anti-symmetric matrix decomposition of $\mathsf{P}_\lambda$ is 
\begin{align}\label{anti-decom-P}
\text{diag}\left\{
    \begin{pmatrix}
        0   &n z_1\\
        -nz_1& 0
    \end{pmatrix},\cdots,
    \begin{pmatrix}
        0   & nz_k\\
        -nz_k& 0
    \end{pmatrix},
    0,\cdots,0
    \right\}
\end{align} and 
the
anti-symmetric matrix decomposition of $\overline{\mathsf{P}}_\lambda$ is 
\begin{align}\label{anti-decom-barP}
\text{diag}\left\{
    \begin{pmatrix}
        0   &n \bar z_1\\
        -n\bar z_1& 0
    \end{pmatrix},\cdots,
    \begin{pmatrix}
        0   & n\bar z_l\\
        -n\bar z_l& 0
    \end{pmatrix},
    0,\cdots,0
    \right\}.
\end{align}

Define
\begin{align*}
    w_i=
\begin{cases}
1 &\text{for $1\leq i\leq \dfrac{r(\Si)-t}{2}$},\medskip\\
n &\text{for $\dfrac{r(\Si)-t}{2}+1\leq  i\leq  \dfrac{(n^2-1)r(\Sigma)-t(n-1)}{2}$},
\end{cases}
\end{align*}
and
$$\bar w_i=
\begin{cases}
1 &\text{for $1\leq i\leq \dfrac{r(\Si)-t}{2}$},\medskip\\
n &\text{for $\dfrac{r(\Si)-t}{2}+1\leq i \leq\dfrac{|\overline V_{\lambda}|-t(n-1)-(b-t)\lfloor\frac{n}{2}\rfloor}2$,}
\end{cases}$$
where
$|\overline V_{\lambda}|=(n^2-1)r(\Sigma)-\binom{n}{2}(\#\partial\Sigma)$ is given in \eqref{eq:cardinarity}.

\begin{cor}\label{prop-anti-decom-P}
Let $\Sigma$ be a triangulable essentially bordered pb surface without interior punctures, $\lambda$ be a triangulation of $\Sigma$, and $r(\Sigma) := \# (\partial \Sigma) - \chi(\Sigma)$, where $\chi(\Sigma)$ denotes the Euler characteristic of $\Sigma$. 
Suppose that $\overline\Sigma$ has $b$ boundary components and among them there are $t$ even boundary components.
We assume $b=t$ when $n$ is even.
With the
anti-symmetric matrix decomposition of $\mathsf{P}_\lambda$ as in \eqref{anti-decom-P}. Then:
\begin{enumerate}
    \item When $n$ is odd, we have
    \begin{align}\label{eq-def-zi}
z_i=
\begin{cases}
w_i &\text{for $1\leq i\leq \dfrac{|W|}{2}$},\medskip\\
2 w_i &\text{for $\dfrac{|W|}{2}+1\leq i\leq \dfrac{(n^2-1)r(\Si) -b(n-1)}{2}$}\medskip\\
4 w_i &\text{for $\dfrac{(n^2-1)r(\Si) -b(n-1)}{2}+1\leq i\leq  \dfrac{(n^2-1)r(\Sigma)-t(n-1)}{2}$}\\
\end{cases}
\end{align}

\item When $n$ is even, we have
\begin{align}\label{eq-def-zi-even}
z_i=
\begin{cases}
1 &\text{for $1\leq i\leq \dfrac{r(\Sigma)-t-2g}{2}$},\medskip\\
2  &\text{for $\dfrac{r(\Sigma)-t-2g}{2}+1\leq i\leq \dfrac{r(\Sigma)-t}{2}$}\medskip\\
n &\text{for $\dfrac{r(\Sigma)-t}{2}+1\leq i\leq  \dfrac{|W|+2g}{2}$}\medskip\\
2n &\text{for $\dfrac{|W|+2g}{2}+1\leq i\leq  \dfrac{(n^2-1)r(\Sigma)-t(n-1)}{2}$}
\end{cases}
\end{align}
\end{enumerate}
\end{cor}
\begin{proof}
(1)
From Theorems \ref{thm:rank} and \ref{thm:rank-odd}, we have that  
$$\rankZ \A= 
\begin{cases}
d^{r(\Sigma)-t}m^{(n^2-1)r(\Sigma)-t(n-1)} & \text{$m'$ is odd,}\\
    2^{|W|-r(\Sigma)+t}d^{r(\Sigma)-t}m^{(n^2-1)r(\Sigma)-t(n-1)} &
    \text{$m'$ is even and $m^{\ast}$ is odd,}\\
   2^{|W|-r(\Sigma)+t+(b-t)(1-n)}
d^{r(\Si)-t}
   m^{(n^2-1)r(\Sigma)-t(n-1)} &
   \text{$m'$ is even and $m^{\ast}$ is even.}
\end{cases}
$$
From simple computations, we have 
$$\prod_{1\leq i\leq k} (\dfrac{m''}{\gcd(m'',nz_i)})^2=
\prod_{1\leq i\leq k} (\dfrac{m'}{\gcd(m',z_i)})^2=\rankZ \A,$$
where $z_i$ are defined in Equation~\eqref{eq-def-zi} and $k=\dfrac{(n^2-1)r(\Sigma)-t(n-1)}{2}$.
Then Lemma \ref{lem-2power-eq} completes the proof for the case when $n$ is odd.

(2)
From Theorems \ref{thm:rank} and \ref{thm:rank-odd} with $b=t$, we have that  
$$\rankZ \A= 
\begin{cases}
d^{r(\Sigma)-t}m^{(n^2-1)r(\Sigma)-t(n-1)} & \text{$m'$ is odd,}\\
    2^{-2g}d^{r(\Sigma)-t}m^{(n^2-1)r(\Sigma)-t(n-1)} &
    \text{$m'$ is even and $n'$ is even,}\\
   2^{|W|-r(\Sigma)+t}
d^{r(\Si)-t}
   m^{(n^2-1)r(\Sigma)-t(n-1)} &
   \text{$m'$ is even and $n'$ is odd.}
\end{cases}
$$
Then the same argument as in (1) works here.

\end{proof}

\begin{thm}\cite[Theorem 7.21]{KW24}\label{thm-PI-reducedA-odd}
Assume that $m'$ is odd.
Let $\Sigma$ be a triangulable essentially
bordered pb surface without interior punctures, and $\lambda=\mu$ be a triangulation of $\Sigma$
introduced in Section~\ref{sub:quantum-torus-reduced}.
Suppose $\overline\Sigma$ has $b$ boundary components among which there are $t$ even boundary components. Then we have 
$$\rankZ \rA=d^{r(\Sigma)-t} m^{|\overline V_{\lambda}|-t(n-1)-(b-t)\lfloor\frac{n}{2}\rfloor},$$
where $d,m$ are defined in Section~\ref{notation} and $|\overline V_{\lambda}|=(n^2-1)r(\Sigma)-\binom{n}{2}(\#\partial\Sigma)$ is given in \eqref{eq:cardinarity}. 
\end{thm}

\begin{rem}
    In \cite[Theorems 6.15 and 7.21]{KW24}, we required that $m''$ is odd. Since the parity of $m'$ is more essential in the proof than that of $m''$, the proofs work for the case when $m'$ is odd with some minor changes of the proofs.
\end{rem}

\begin{cor}
\label{prop-anti-decom-P-reduced}
Let $\Sigma$ be a triangulable essentially bordered pb surface without interior punctures, $\lambda=\mu$ be a triangulation of $\Sigma$
introduced in Section~\ref{sub:quantum-torus-reduced}, and $r(\Sigma) := \# (\partial \Sigma) - \chi(\Sigma)$, where $\chi(\Sigma)$ denotes the Euler characteristic of $\Sigma$. 
Suppose that $\overline\Sigma$ has $b$ boundary components and among them there are $t$ even boundary components.
We assume $n$ is odd.
With the
anti-symmetric matrix decomposition of $\overline{\mathsf{P}}_\lambda$ as in \eqref{anti-decom-barP}. 
We have
    \begin{align}\label{red-eq-def-zi}
\bar z_i=
\begin{cases}
\bar w_i &\text{for $1\leq i\leq \dfrac{(n-1)(\sharp\partial\Si)}{2}$},\medskip\\
2 \bar w_i &\text{for $\dfrac{(n-1)(\sharp\partial\Si)}{2}+1\leq i\leq \dfrac{|\overline V_{\lambda}|-t(n-1)-(b-t)\lfloor\frac{n}{2}\rfloor}2$,}
\end{cases}
\end{align}
where $|\overline V_{\lambda}|=(n^2-1)r(\Sigma)-\binom{n}{2}(\#\partial\Sigma)$ is given in \eqref{eq:cardinarity}.
\end{cor}
\begin{proof}
    Using Theorems~\ref{thm-PI-reducedA} and \ref{thm-PI-reducedA-odd}, the proof of Corollary~\ref{prop-anti-decom-P} works here.
\end{proof}

\end{document}